\documentclass[11pt,letterpaper,leqno]{article}
\usepackage{amsmath,amsthm}
\usepackage{amssymb,amsfonts,amscd,mathrsfs,latexsym}
\usepackage{mathtools}
\usepackage{thmtools}
\usepackage{graphicx}
\usepackage{color}
\usepackage{xcolor}
\usepackage[sort&compress]{natbib}
\usepackage{cite}
\usepackage{enumitem}
\usepackage{relsize}
\usepackage{setspace}
\usepackage{accents}
\usepackage{authblk}
\usepackage{adjustbox}
\usepackage{caption}
\usepackage{array}
\usepackage[colorlinks,citecolor=blue,urlcolor=blue,
linkcolor=blue]{hyperref}
\usepackage{hypernat}

\newtheorem{thrm}{Theorem}[section]
\newtheorem*{thrm*}{Theorem}
\newtheorem{lemm}[thrm]{Lemma}
\newtheorem*{lemm*}{Lemma}
\newtheorem{prop}[thrm]{Proposition}
\newtheorem*{prop*}{Proposition}

\newtheorem*{corl*}{Corollary}

\newtheorem*{claim*}{Claim}

\theoremstyle{definition}
\newtheorem{defn}[thrm]{Definition}
\newtheorem*{defn*}{Definition}

\newtheorem*{exmp*}{Example}

\newtheorem*{rmrk*}{Remark}

\renewcommand{\qed}{\hfill \mbox{\raggedright \rule{0.1in}{0.1in}}}

\DeclareMathOperator{\tr}{tr}

\newcommand{\overbar}[1]{\mkern 1.5mu\overline{\mkern-1.5mu#1\mkern-1.5mu}\mkern 1.5mu}
\newcommand{\ov}[1]{\overbar{#1}}
\newcommand{\un}[1]{\underaccent{\bar}{#1}}
\newcommand{\bds}[1]{\boldsymbol{#1}}
\newcommand{\unwt}[1]{\underaccent{\sim}{#1}}
\newcommand{\upwt}[1]{\accentset{\sim}{#1}}

\numberwithin{equation}{section}

\begin{document}

\title{\bf{Strategic Growth with Recursive Preferences:\\
Decreasing Marginal Impatience}}

\author[1]{Luis Alcal\'{a}
\thanks{Corresponding author. Email: lalcala@unsl.edu.ar}}

\author[2]{Fernando Tohm\'{e}}

\author[2]{Carlos Dab\'{u}s}

\affil[1]{Departamento de Matem\'{a}tica, Universidad Nacional de San Luis\newline
San Luis, Argentina}

\affil[2]{Departamento de Econom\'{i}a, Universidad Nacional del Sur  and CONICET\newline
Bah\'{i}a Blanca, Buenos Aires, Argentina}

\date{ }

\maketitle

\begin{abstract}
\noindent We study the interaction between strategy, heterogeneity and growth in a two-agent model of capital accumulation. Preferences are represented by recursive utility functions with decreasing marginal impatience. The stationary equilibria of this dynamic game are analyzed under two alternative information structures: one in which agents precommit to future actions, and another one where agents use Markovian strategies. In both cases, we develop sufficient conditions to prove the existence of equilibria and characterize their stability properties. The precommitment case is characterized by monotone convergence, but Markovian equilibria may exhibit nonmonotonic paths, even in the long-run.\medskip

\medskip

\medskip
\end{abstract}
\thispagestyle{empty}

\newpage

\section{Introduction}
\label{sec:intro}

We study the interaction between strategy, heterogeneity and growth in a two-agent model of capital accumulation, where preferences are represented by recursive utility functions with decreasing marginal impatience and agents receive different shares of income. The stationary equilibria of this dynamic game are studied under two alternative information structures: the precommitment or open-loop model and the Markovian or closed-loop model. In the first case, agents commit to an entire sequence of choices at the beginning of their interaction. In the second model, strategies specify a course of action based on the history of past choices.

Compared with neoclassical models of optimal growth, representative agents models with recursive preferences have certain features that better explain the behavior of actual economies. For instance, \citet{mantel99} proved the existence of finitely many steady states that are locally asymptotically stable with separate regions of monotone convergence. This model also shows dependence on initial conditions, and the possibility of falling into poverty traps. But this is far from a realistic depiction of the diversity found in growth experiences. On one hand, economic variables do not always exhibit monotone growth  paths. On the other hand, heterogeneous time preferences cannot be easily reduced to a representation with a single agent. These two features are not unrelated: the \emph{mutual influence} between economic agents makes capital paths interdependent, therefore more prone to nonmonotonic behavior.

Although the model we develop is provided with a number of structural assumptions, substantial technical challenges are involved. We exploit similarities between the current framework and models of nonoptimal infinite-horizon economies with capital accumulation. To paraphrase \citet{greenhuff95}: The usual problem in these environments is that agents fail to take into account how their actions influence the behavior of other agents. In order to solve an individual's dynamic program, one needs to know the equilibrium law of motion governing the aggregate state, but to know this in turn requires knowledge of  individual decision rules. Proving existence of equilibria can then be problematic, since it can be difficult to establish that the individual and aggregate laws of motion coincide along an equilibrium path.

Two additional elements complicate the resolution of our model. One of them is the presence of heterogeneity in preferences and income shares across agents; the other one is the interdependence of strategy spaces. Despite these difficulties, we are able to prove the existence of stationary equilibria under precommitment strategies and Markovian strategies, and give a characterization of these equilibria based on general functional forms. Thus, an important contribution of this paper is to develop a method for solving stationary equilibria in a class of capital accumulation games with recursive preferences, particularly for the case of equilibria in Markovian strategies (see Section \ref{sec:Markov}). The method we develop results from a combination of different approaches, which include the  dynamical systems approach, operator methods, and fixed point theory.

More specifically, we introduce sufficient conditions to prove the existence of stationary equilibria and analyze their stability properties. It is illustrative to compare this equilibria with the equilibria that agents would obtain by themselves (the ``autarky'' setting). This approach gives the model a frame of reference and also allows a heuristic analysis of how changes in the players' strategy spaces affect equilibrium outcomes. Assuming that one agent is uniformly more impatient than the other, the stationary precommitment level of capital lies below the  corresponding levels in autarky. In turn, the stationary level of capital in a Markovian equilibrium is lower than in the precommitment case, indicating that there is a price to be paid when agents monitor their actions along the entire equilibrium path.

In a close paper to ours, \citet{camachoetal13} also analyze a strategic growth model with recursive preferences, in which discount factors are a function of capital held by each agent, but preferences are homogeneous. Output shares are determined by the proportion of aggregate capital that each agent holds. Hence, these shares not only differ among agents, but also vary over time. An important contribution of their paper is that it dispenses with the assumption of concavity of the value function and applies the methods of supermodular games to determine the existence of stationary equilibria. It is also shown, via numerical examples, that open-loop equilibria ``tend to be symmetric'' and strategic interaction removes the multiplicity and indeterminacy that may arise in single agent models. Furthermore, the presence of time-varying output shares does not generate asymmetric equilibria. 

The plan of the paper is as follows. The next subsection   reviews the literature, both on recursive preferences in infinite horizon capital accumulation models, and on the methods used to solve related optimization problems. Section \ref{sec:environment} introduces the problem and main assumptions. An analysis of the space of feasible strategies available to the agents is given in Section \ref{sec:stratspace}, together with a characterization of the autarky problem. Section \ref{sec:precommit} studies equilibria in precommitment strategies, including the problems of existence and stability. In Section \ref{sec:Markov}, the analysis is dedicated to Markovian strategies, again covering different aspects ranging from the existence of equilibria to the issue of their stability. Section \ref{sec:conclusions} concludes.

\subsection{Related Literature}

Since Ramsey's seminal paper, optimal economic growth has been represented by a single agent that maximizes a time-additive  intertemporal welfare function with a constant rate of time preference. This basic construction has remained the fundamental building block of models of optimal growth. An alternative to constant discounting was to consider welfare functions with a variable rate of time preference. The theoretical foundation of this approach was laid out by \citet{koopmans60}. Its effects on optimal growth were analyzed by \citet{bealskoop69} and \citet{iwai72}, and, more recently, \citet{mantel93,mantel95,mantel99}, \citet{beckermull97}, \citet{das03}, \citet{stern06}, and \citet{eroletal11}. The key difference among these approaches lies on how the rate of time preference is modeled as a function of consumption or capital.

Models that assume increasing marginal impatience were pioneered by \citet{uzawa68} with later contributions by  \citet{lucasstokey84}, \citet{epstein87}, and \citet{obstfeld90}. These models have a certain appeal, because their equilibria are globally stable, independently of initial conditions. On the other hand, following a suggestion of \citet{fisher30} in his \emph{Theory of Interest}, \citet{mantel99} postulated that the degree of impatience (measured as the rate of time preference) should be a decreasing function of consumption and thus, indirectly, of income. The lower the income level, the higher the sacrifice of postponing present consumption in exchange for future consumption. Recent theoretical models and experimental evidence support \emph{decreasing marginal impatience} as the most plausible assumption.\footnote{A different way of including variable discount rates has been largely explored in the literature on behavioral economics. For example, \citet{laibson97} postulates that time inconsistent intertemporal choices may lead agents to have hyperbolic discounting, rather than geometric discounting. Preferences in this paper are time invariant and satisfy the stationarity postulate from \citet{koopmans60}, so they are time consistent.}

As mentioned earlier, representative agent models with decreasing marginal impatience generally produce multiple steady states and exhibit dependence on initial conditions, which includes the possibility of yielding poverty traps.\footnote{These properties were actively investigated by the late Rolf Mantel \citep{mantel67,mantel93,mantel95,mantel99}. Although Mantel's work along these lines is far less known than his celebrated results on the Sonnenschein-Mantel-Debreu Theorem in General Equilibrium Theory, it follows from the same foundational concerns. For a detailed treatment, see \citet{tohme06}.} The recent work by \citet{eroletal11} assumes that the rate of time preference is a function of the capital stock. There it is shown that a poverty trap may arise even with a convex technology. Also, under a convex-concave technology, the optimal path can exhibit global convergence to a unique stationary point.

Another basic framework for studying optimal growth is the Ramsey equilibrium model with constant but heterogeneous discount rates. Under homogeneous preferences, the steady state reflects the degree of impatience of the representative agent. But this approach has an important shortcoming, since the capital distribution in the economy is indeterminate. In the case of many agents having different discount rates, stationary equilibria are characterized by what is known as \emph{Ramsey's conjecture}, where only the most patient owns the capital stock of the economy and all others have none and live off their wage incomes. The distribution of capital is determinate, but it ends up being concentrated in a single agent. A thorough review of this literature can be found in \citet{becker06} and the references therein.\footnote{See also \citet{beckeretal14} for an up-to-date study of turnpike properties, monotone income conditions, and exceptions to the Ramsey equilibrium in discrete time. They point out that continuous time models give different predictions, another issue that was explored in \citet{mantel99}.} 

In our model, the inclusion of recursive preferences and strategic interaction implies a non-degenerate stationary capital distribution, so the Ramsey equilibrium result is avoided. An alternative in this direction has been given by \citet{sorger02,sorger08}, assuming that agents eventually realize that they exercise some market power in the capital market, but still behave competitively in the labor market. This leads to open-loop equilibria in which all agents in the economy own positive amounts of capital. \citet{pichlersorger09} extend the analysis to Markovian equilibria. However, they are not able to obtain analytical results. Based on numerical simulations, they show that a Markov perfect equilibrium can be affected by preferences, while the open-loop equlibrium is determined by technology alone. They also find some evidence of nonmonotonic behavior. 

\citet{docknishim04} develop a two-agent model with  externalities in the form of knowledge spillovers and   heterogeneous technological parameters. Assuming certain specific functional forms for preferences and technologies, they show that dynamics under open-loop and Markov equilibria are remarkably similar: they are characterized by a balanced growth path, a unique stable equilibrium, and monotonic behavior of aggregate variables. However, when external effects are relatively large, there exist oscillatory paths. On the other hand, growth rates are larger in Markovian equilibria than in open-loop equilibria.

Another model of strategic growth is given in  \citet{docknishim05}, who also study a two-agent economy with a nonconcave production function. They find that in an open-loop equilibrium, there is a threshold value for the initial capital stock such that above the threshold, the solution converges to an efficient steady state, but below the threshold the capital stock converges to zero. Moreover, there exists a unique interior globally stable Markov perfect equilibrium. Results for the latter are also obtained with specific functional forms for preferences and technology.

Recent work by \citet{drugeonwign15} has brought a novel perspective on this problem, by introducing temptation motives and self-control costs in a standard Ramsey model, building on the preference representation developed by \citet{gulpesendorf04}. It is shown that when agents are heterogeneous in their discount factors and temptation motives, equilibria are characterized by nondegenerate long-run distributions of consumption and wealth. In fact, the most patient individual may end up having lower consumption levels than the other agents, if their temptation motives are sufficiently low.

In terms of the formal methods applied to finding equilibria in this paper, we build on an early contribution by \citet{rosen65}, who studies the existence and uniqueness of $n$-person generalized concave games. Each player's strategy space depends on the strategy of the other players, like in the abstract economies of Arrow-Debreu and McKenzie. Existence follows from a generalization of Kakutani's fixed-point theorem, while uniqueness is ensured by a property called diagonal strict concavity.

The dynamical systems approach was developed by \citet{beckerfoias98} to analyze the Ramsey equilibrium under  geometric discounting and heterogeneous discount rates. The construction of an equilibrium is based on the turnpike property. In particular, a functional equation is derived from the Euler condition of the most patient agent, which defines a nonlinear operator with a monotonicity property. The solutions are found applying a method called \emph{implicit programming}. \citet{beckerfoias07} have also applied these methods to solve a model of strategic growth and find similar dynamics. 

We have also mentioned \citet{greenhuff95}, who addressed the existence problem of a nonoptimal stochastic growth model. They follow some innovative work of \citet{coleman91}, which has been extended by \citet{coleman97,coleman00}, \citet{dattmirmreff02}, and \citet{mirmorref08}. Formally, the aggregate law of motion is a state variable that determines, together with individual state variables, individual decision rules. This interaction between individual and aggregate behavior is embedded into the Euler equation of the representative agent. Imposing consistency between individual and aggregate law of motions, and working recursively with the Euler equation, yields a monotonic operator. We take some insights from this approach, and extend some results to include heterogeneity and strategic behavior.

\section{The Environment}
\label{sec:environment}

Time is discrete and agents are infinitely lived. The economy considered here consists of two agents, one storable good that can be consumed or used as capital for production, and a single productive unit. We assume, as in an Arrow-Debreu-McKenzie private ownership economy, that agents receive a constant proportion of output as income shares. The idea is to characterize an equilibrium path for this  economy as a function of preferences, technology, initial capital, and income shares. 

Throughout the paper, variables and parameters indexed by $i$  correspond to agent $i$, while those indexed by $j$ correspond to the other agent, which is always interpreted as $j \neq i = 1,2$. Each agent $i$ has a prospective utility function over nonnegative consumption sequences $\bds{c}^i:=\{c_t^i\}_{t=0}^\infty$ given by 
\begin{equation}
\label{eq:Widef}
\mathrm{W}_0^i(\bds{c}^i)=\sum_{t=0}^{\infty} \left(\prod_{s=0}^{t-1}
\alpha_i(c_s^i)\right) u_i(c_t^i),
\end{equation}
where $u_i(c_t^i)$ is agent $i$'s instantaneous utility of consuming $c_t^i$, and the real-valued function $\alpha_i(c_t^i)$ is the subjective factor of time preference. Then, the one-period discount rate for agent $i$ can be defined as $\rho_i(c_t^i):= 1/\alpha_i(c_t^i)-1$. 

The following assumptions describe prospective utility in terms of instantaneous utility and the discount factor for each $i=1,2$.
\begin{enumerate}[label=(U\arabic*)]
\item\label{asm:alpha} $\alpha_i:\mathbb{R}_+ \to \mathbb{R}_{++}$ is continuous, strictly increasing and strictly concave on $\mathbb{R}_+$, twice continuously differentiable on $\mathbb{R}_{++}$ with $\alpha_i''(c) < 0 < \alpha_i'(c)$ for all $c > 0$,
\item\label{asm:alphap} $\alpha_i(0)>0$; $\sup_{c > 0}\ \alpha_i(c) = \ov{\alpha}_i$, for some constant $ 0 < \ov{\alpha}_i < 1$,
\item\label{asm:utility} $u_i:\mathbb{R}_+ \to \mathbb{R}_+$ is continuous, strictly increasing and strictly concave on $\mathbb{R}_+$; and twice continuously differentiable on $\mathbb{R}_{++}$ with $u_i''(c) < 0 < u_i'(c)$ for all $c > 0$,
\item\label{asm:uprime} $u_i(0)=0$; $\lim_{c \to 0^+} u_i'(c)=+\infty$; $\lim_{c \to +\infty} u_i'(c)=0$,
\end{enumerate}

Since $\alpha_i(c_t^i)$ is increasing in consumption, it can also be thought as being increasing in income. Hence, $\rho_i(c_t^i)$ decreases with consumption and income, satisfying the property of decreasing impatience. Under these assumptions, prospective utility defines a welfare path $\{\mathrm{W}_t^i\}_{t=0}^\infty$ by solving the following difference equation:
\begin{align}
\label{eq:Wirec}
&\mathrm{W}_t^i = u_i(c_t) + \alpha_i(c_t)\mathrm{W}_{t+1}^i, &t=0,1,\ldots
\end{align}
given the initial condition $(c_0^i,\mathrm{W}_0^i)$. The function on the right-hand side of \eqref{eq:Wirec} is called an \emph{utility  aggregator}, in this case defined as $\mathrm{U}^i(c,\mathrm{W}^i) := u_i(c) + \alpha_i(c)\mathrm{W}^i$. This utility aggregator is continuous, increasing in its arguments, and satisfies a Lipschitz condition of order one. Moreover, assumptions \ref{asm:alpha}--\ref{asm:uprime} imply that successive approximations of $\mathrm{U}^i$ converge uniformly  to a unique prospective utility.\footnote{Proofs of these statements can be found in \citet{boyd90}.}

Agents are endowed with $(k_0^i,k_0^j)$ units of capital at the beginning of the first period. There is a single productive unit in the economy where both agents pool their capital to produce output. Aggregate capital is denoted by $k_t:=k_t^i+k_t^j$ for every $t$. Technology consists of a neoclassical aggregate production function $f$ which satisfies the following:
\begin{enumerate}[label=(T\arabic*)]
\item\label{asm:prodf} $f:\mathbb{R}_+ \to \mathbb{R}_+$ is continuous, strictly increasing and strictly concave on $\mathbb{R}_+$; and twice continuously differentiable on $\mathbb{R}_{++}$ with $f''(k) < 0 < f'(k)$ for all $k > 0$,
\item\label{asm:prodfp} $f(0)=0$; $\lim_{k \to 0^+} f'(k) > \max\left\{\alpha_i(0)^{-1},\alpha_j(0)^{-1}\right\}$; $\lim_{k \to +\infty} f'(k) < 1$.
\end{enumerate}
It is immediately verified that these properties imply the existence of a maximum sustainable capital level, i.e., there is a $k_m > 0$ such that $f(k_m)=k_m$, $f(k) < k$ for all $k > k_m$, and $f(k) > k$ for all $k < k_m$ . This fact and \ref{asm:prodfp}, together with \ref{asm:alpha}--\ref{asm:alphap}, guarantee the existence of an interior stationary equilibrium. For simplicity, we further assume that output $f(k)$ is net of depreciation and maintenance costs.

The single good in this economy is used both for consumption and accumulation. Agents make investment decisions separately. Every period $t$, each agent receives a share $\theta^i \geq 0$ of total output. Income shares are constant over time and satisfy $\theta^i + \theta^j=1$. Without loss of generality, we assume $\theta^i > \theta^j$ throughout the paper. It is further assumed that these shares are public information, fully enforceable and not open to renegotiation.

In a sequential formulation of the problem, each agent chooses infinite sequences of consumption and savings to maximize prospective utility, given by \eqref{eq:Widef}, taking   the initial conditions and the actions of the other player as given. Then, the sequential program for agent $i$  has the form
\begin{align}
\label{eq:SPi}
\max_{\left\{c_t^i,k_{t+1}^i\right\}_{t=0}^\infty}\quad
  &\sum_{t=0}^{\infty} \left(\prod_{s=0}^{t-1}
\alpha_i(c_s^i)\right) u_i(c_t^i)\tag{SP$_i$}\\[5pt]
\text{s.t. }\quad & k_{t+1}^i \leq \theta^i f(k_t^i+k_t^j)-c_t^i,&t=0,1,\ldots \nonumber\\[5pt]
&c_t^i \geq 0,\ k_{t+1}^i \geq 0, & t=0,1,\ldots\nonumber
\\[5pt]
&k_0^i \geq 0,\  \text{given},\nonumber
\end{align}
where $\{k_t^j\}_{t=0}^\infty$ is the strategy followed by the  other player. 

A capital path $\bds{k}^i:=\{k_t^i\}_{t=0}^\infty$ is said to be \emph{feasible from $(k_0^i,k_0^j) \geq 0$}, given $\bds{k}^j \geq 0$, if it satisfies
\begin{equation*}
0 \leq k_{t+1}^i \leq \theta^if(k_t^i+k_t^j), \qquad t=0,1\ldots
\end{equation*}
Let $\Pi^i\big[(k_0^i,k_0^j);\bds{k}^j\big]$ be the set of all feasible capital paths for player $i$. A consumption path $\bds{c}^i:=\{c_t^i\}_{t=0}^\infty$ is said to be \emph{feasible from $(k_0^i,k_0^j)$} if there is a sequence $\bds{k}^i$ in $\Pi^i\big[(k_0^i,k_0^j);\bds{k}^j\big]$ such that
\begin{equation*}
0 \leq c_t^i \leq \theta^if(k_t^i+k_t^j)-k_{t+1}^i, \qquad t=0,1,\ldots
\end{equation*}

Since preferences are strictly monotone, the first  inequality restriction in \eqref{eq:SPi} will be binding. Moreover, Inada conditions guarantee an interior solution for $c_t^i$ and $k_{t+1}^i$, for all $t$. As a result, \eqref{eq:SPi} can be written as
\begin{equation*}
\max_{\{k_{t+1}^i\}_{t=0}^\infty \in \Pi^i[(k_0^i,k_0^j);\bds{k}^j]}\ \sum_{t=0}^\infty \left[\prod_{s=0}^{t-1}
\alpha_i\left(\theta^i f(k_t^i+k_t^j)-k_{t+1}^i\right)\right]
u_i\left(\theta^i f(k_t^i+k_t^j)-k_{t+1}^i\right).
\end{equation*}
Given that the set $\Pi^i$ is compact in the product topology defined over the space of bounded non-negative sequences $\bds{k^i}$, and the objective function is continuous, this problem is well defined. In a strategic environment, the solution to this problem critically depends on the information structure of the game. Throughout the paper, we restrict our attention to pure strategies under two alternative information structures: precommitment or open-loop strategies, and Markovian or closed-loop strategies. For the case of precommitment, optimal strategies are given by individual capital paths, $\boldsymbol{k}^{i}$ and $\boldsymbol{k}^j$, and the aggregate capital path that results from these decisions yields an equilibrium in open-loop strategies, or simply, an open-loop equilibrium.\footnote{In what follows, we use the terms open-loop equilibrium and precommitment equilibrium interchangeably.} In a stationary Markovian equilibrium, strategies are functions of the state of the system. The state contains all the relevant information from past consumption and savings decisions, as well as the evolution of the state in previous periods. 

As will be shown below, there are certain advantages of using a recursive formulation of the sequential problem \eqref{eq:SPi} for developing the solution methods introduced in this paper. At the beginning of any period, individual choices are basically constrained by \emph{aggregate capital} $k_t$, since past accumulation decisions are of no consequence for current or future payoffs. In the language of the recursive approach, aggregate capital represents the state variable of the dynamic optimization problem, while consumption levels  $c_t^i$ and $c_t^j$ are the decision variables. If the feasible sets for individual consumption are carefully constructed, then a feasible consumption profile $(c_t^i,c_t^j)$ immediately implies that the level of aggregate capital next period $k_{t+1}$ is also feasible. Thus, keeping track of individual capital holdings becomes unnecessary. To see this, note that adding up the resource constraints from each \eqref{eq:SPi}, and taking nonnegativity constraints into account, yields
\begin{equation*}
c_t^i + c_t^j \leq f(k_t) - k_{t+1}, \qquad t=0,1,\ldots,
\end{equation*}
which will be satisfied with equality along an optimal path. Also, if the equilibrium is interior, we have that $c_t^i>0$, $c_t^j > 0$, and $0 < k_{t+1} < f(k_t)$, and this in turn implies $0 < c_t^i  < f(k_t) - c_t^j$, for all $t$.

Given the existence of a maximum sustainable capital $k_m$, the state space can be restricted to the closed interval $K:=[0,k_m]$. At the beginning of each period, player $i$  chooses $c^i$ over the correspondence $[0,f(k)-c^j]$ to solve a dynamic program (DP$_i$) of the form
\begin{align}
\label{eq:mainOL}
v_i(k)=& \max_{0 \leq c^i \leq f(k)-c^j}\
 \Big\{u_i(c^i)+\alpha_i(c^i)\,v_i(f(k)-c^i-c^j):\
 k \in K,\ 0 \leq c^j \leq f(k)\Big\},\tag{DP$_i$}
\end{align}
and player $j$ solves an analogous program. A solution to \eqref{eq:mainOL} yields player $i$'s best response to a strategy $c^j$, for a given initial value $k$ of the state variable. Clearly, the strategy space of each player depends on the strategy followed by the other player, so this is a game that belongs to the class of generalized games in the sense of \citet{arrowdebreu54} and \citet{mckenzie59}, also known as games of strategic dependence.\footnote{Another branch of the literature --typically related to differential games-- refers to this class as ``games with coupled constraints.'' This terminology is also used by \citet{rosen65}.} Furthermore, as players share a common technology to produce output, such interdependence is a natural element of the game. This, however, creates some technical difficulties which are discussed next.

\section{The Strategy Space}
\label{sec:stratspace}

Let $S^i \subset \mathbb{R}_+$ and $S^j \subset \mathbb{R}_+$ be the sets of all feasible strategies for players $i$ and $j$, respectively. A feasible strategy profile $c:=(c^i,c^j)$ is an element of the product space $S:=S^i \times S^j$. In the first part of this discussion we consider a fixed value $k \in K$, so we can apply the static framework developed by \citet{rosen65} for concave $n$-person games. For this problem to be well defined, admissible strategies must be selected from a convex, compact subset $S \subset \mathbb{R}_+^2$, and  each player's payoff function $u_i$ has to be continuous in $(c^i,c^j)$ and concave in $c^i$ for each fixed value of the strategy followed by the other player. If $P^i$ denotes the projection of $S$ on $\mathbb{R}_+$, Rosen assumes that preferences are defined over the set $P = P^i \times P^j$, which is a superset of $S$. A shortcoming of using $P$ as the strategy space is that it will contain infeasible strategies.

\citet{banksduggan04} extended the result obtained by Rosen  and solved the problem of having a strategy space which includes infeasible strategies. The set $S$ is considered a primitive of the problem, instead of being a construction, and  it is assumed to be nonempty, convex and compact. Preferences are defined only on the set of feasible strategy profiles, and are represented by a continuous and quasi-concave utility function $u_i: S \to \mathbb{R}$. Provided that the feasibility correspondence of each agent is lower hemicontinuous, the authors prove the existence of a Nash equilibrium from standard results.\footnote{See also \citet{carlson02} on this particular subject.}

What is the feasible strategy space in our problem? To attempt an answer to this question, first consider the following auxiliary program related to \eqref{eq:SPi} for a single player, say $i$. We will refer to this case as player $i$'s \emph{autarky} decision problem, and denote it by $\mathcal{A}^i$. Essentially, this is a world where only player $i$ exists and output is not shared with anyone else, a problem that can be seen as a single-agent optimal growth model with recursive preferences. This class of models has been analyzed at length by \citet{mantel93,mantel95,mantel99}, and more recently by \citet{eroletal11}, among others. We freely borrow their results to characterize $\mathcal{A}^i$. The reader is referred to those papers for details.

Let $\tilde{v}_i:K \to \mathbb{R}$ denote the value function corresponding to the optimization problem associated with $\mathcal{A}^i$, with a Bellman equation given by
\begin{equation*}
\tilde{v}_i(k)=\max_{0 \leq x^i \leq f(k)}\
 \Big\{u_i(x^i)+\alpha_i(x^i)\tilde{v}_i(f(k)-x^i)\Big\}.
\end{equation*}
This value function satisfies $\tilde{v}_i(0)=0$ and $\tilde{v}_i(k)> 0$ if $k > 0$. Moreover, it is strictly increasing and continuous on $[0,f(k)]$ and differentiable on $(0,f(k))$. Given $k_0 > 0$, an optimal consumption path  satisfies the Euler equation and envelope condition
\begin{align}
\label{eq:Euleraut}
&u_i'(c_t^i)+\alpha_i'(c_t^i)\tilde{v}_i(f(k_t)-c_t^i)=\alpha_i(c_t^i)\tilde{v}_i'(f(k_t)-c_t^i), & t=0,1,\ldots,
\\[5pt]
\label{eq:envelaut}
&\tilde{v}_i'(k_t)=[u_i'(c_t^i)+\alpha_i'(c_t^i)\tilde{v}_i(f(k_t)-c_t^i)]f'(f(k_t)-c_t^i), & t=0,1,\ldots,
\end{align}
and $k_t$ evolves according to $k_{t+1} = f(k_t) - c_t^i$, for all $t$. For any optimal stationary point $k$, it must be true that $f(k) \geq k$, so that $k \in [0,k_m]$. The continuity of $\alpha_i$, $u_i$, $f$, and $\tilde{v}_i$ implies that if $k = \lim_{t \to \infty} k_t$, then
\begin{equation}
\label{eq:SPaut}
\alpha_i(f(k)-k)f'(k) = 1
\end{equation}
must hold.
Since $0 < \alpha_i(f(k)-k) < 1$ for every $k \in K$, it follows that $f'(k) > 1$, so the equilibrium is dynamically efficient.

Our assumptions on preferences \ref{asm:alpha}--\ref{asm:uprime} and technology \ref{asm:prodf}--\ref{asm:prodfp} imply the existence of a stationary point $k \in (0,k_m)$. Moreover, if the left-hand side of \eqref{eq:SPaut} is strictly decreasing on some interval $I \subseteq (0,k_m)$ containing $k$, such stationary point is locally stable and locally isolated, two properties that will prove useful in this and later sections. Differentiating the left-hand side of  \eqref{eq:SPaut} and making appropriate substitutions, this condition is equivalent to
\begin{equation}
\label{eq:eta1st}
\frac{\alpha_i'(f(k)-k)}{\alpha_i(f(k)-k)}(f'(k)-1)+\frac{f''(k)}{f'(k)} < 0,
\end{equation}
for all $k \in I$. To obtain meaningful and comparable results, we will further assume that discount factors satisfy $\alpha_i(c^i) \geq \alpha_j(c^i)$ for all $c^i=c^j \in f(I)$ such that $c^i=c^j=c$. Hence, if they consume equal amounts, agent $i$ is uniformly more patient than player $j$ over $f(I)$. This assumption will be kept for the remaining of the paper, adjusting the definition of $I$ as needed.

By looking at \eqref{eq:SPaut}, a stationary point $k$ can be thought as the intersection of the marginal product curve $f'(k)$ and the graph of $(\alpha_i(f(k)-k))^{-1}$, as   depicted in Figure~\ref{fig:stateqAU}. Points A and B represent the autarky equilibrium for player $i$ and player $j$, respectively. Although these equilibria, by definition, are independent from each other, they can be used as reference points to characterize equilibria with strategic interactions. Based on geometrical properties, there is an intuitive interpretation that can be obtained from the stability condition \eqref{eq:eta1st}. Taking into account that $(1/\alpha_i)'=-\alpha_i'/\alpha_i^2$ and substituting \eqref{eq:SPaut} into \eqref{eq:eta1st}, the latter can be written as
\begin{equation}
\label{eq:slopesaut}
\left|f''(k)\right|> \left|\left(\alpha_i(f(k)-k)^{-1}\right)'\right|.
\end{equation}
This means that, in absolute value, the slope of the marginal product curve must be larger than the slope of the curve $(1/\alpha_i)$ taken as a function of $k$.\footnote{In terms of the subjective discount rate, \eqref{eq:slopesaut} can be expressed as $\left|f''(k)\right| > \rho_i(f(k)-k)\left|\rho_i'(f(k)-k)\right|$.} The fact that $1/\alpha_i \leq 1/\alpha_j$, together with the strict concavity of $f$, immediately implies that the agent with a higher level of patience has more capital in autarky, as can be seen in Figure \ref{fig:stateqAU}. 

\begin{figure}[!h]
\centering
\includegraphics[width=0.75\textwidth]{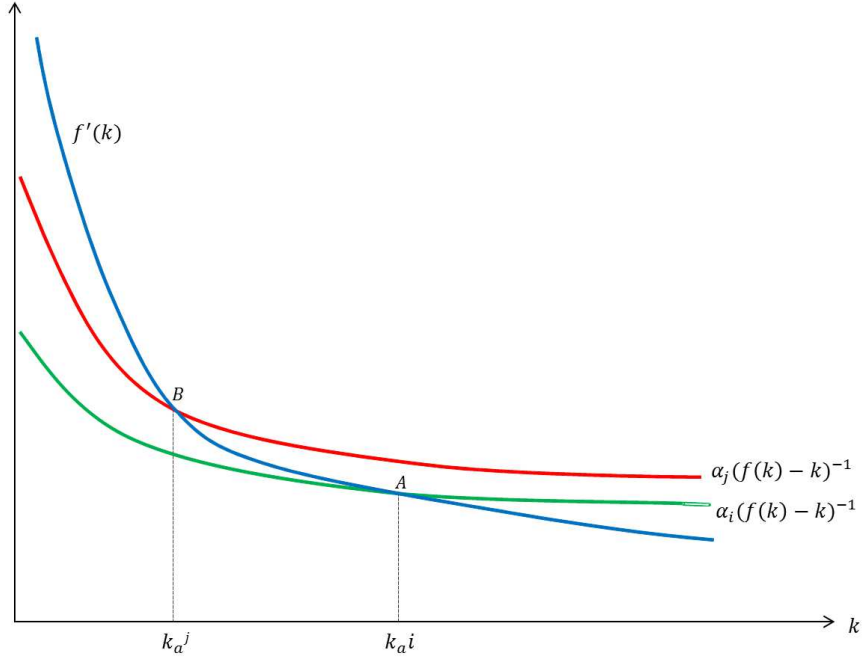}
\caption{Autarky equilibria}
\label{fig:stateqAU}
\end{figure}

Note that the assumption on discount factors can be interpreted as a ``sorting condition'' in the sense that marginal utility of (current) consumption is higher for the less patient player. If we identify an agent's ``type'' with impatience, the higher type consumes more (i.e., accumulates less capital) in equilibrium. This property extends to strategic equilibria, as will be shown in later sections. The characterization of  autarkic equilibria is summarized in the following lemma, whose proof can be found in Appendix \ref{sec:appSP}.

\begin{lemm}
\label{lem:stateqAU}
Suppose that $I$ is a subinterval of $K$ such that $\alpha_j(c^j) \leq \alpha_i(c^i)$ holds for all $c^i=c^j=c \in f(I)$. Then, there exist stationary equilibria $k_a^i,k_a^j \in I$ for programs $\mathcal{A}^i$ and $\mathcal{A}^j$, respectively. Moreover, $0 < k_a^j \leq k_a^i$ and the corresponding  equilibrium levels of consumption satisfy $0 < c_a^j \leq c_a^i$. 
\end{lemm}

Next, we present an intuitive argument for the construction of the strategy space that makes use of autarky equilibria. For this, fix $k \in (0,k_m)$ in \eqref{eq:Euleraut} and define the functions $LHS_i(c^i;f(k)):= u_i'(c^i)+\alpha_i'(c^i)\tilde{v}_i(f(k)-c^i)$, and $RHS_i(c^i;f(k)):= \alpha_i(c^i)\tilde{v}_i'(f(k)-c^i)$, for all $c^i > 0$. Clearly, $LHS_i$ is strictly decreasing with $LHS_i(f(k);\cdot)=u_i(f(k))+\alpha_i'(f(k))\tilde{v}_i(0) > 0$ and, by the Inada condition in \ref{asm:uprime}, $\lim_{c^i \to 0^+} LHS_i(c^i;\cdot)=+\infty$. On the other hand, $RHS_i$ is strictly increasing with $RHS_i(0;\cdot)=\alpha_i(0)\tilde{v}_i'(f(k)) > 0$, and
$\lim_{c^i \to f(k)^-} RHS_i(c^i;\cdot)=+\infty$, given that $\tilde{v}_i' \to +\infty$ as $k \to 0^+$, which follows from the fact that $\tilde{v}_i(0)=0$ and assumptions  \ref{asm:alphap} and \ref{asm:uprime}. These conditions guarantee the existence of an equilibrium consumption level $d^i \in (0,f(k))$ at which the graphs of $LHS_i$ and $RHS_i$ intersect, so $d^i$ is implicitly determined by 
\begin{equation}
\label{eq:EulerU}
u_i'(d^i)+\alpha_i'(d^i)\tilde{v}_i(f(k)-d^i)=\alpha_i(d^i)\tilde{v}_i'(f(k)-d^i).
\end{equation}
An equilibrium for the autarky case corresponds to point $D$ in Figure \ref{fig:stratspace}.

\begin{figure}[h!]
\centering
\includegraphics[width=0.9\textwidth]{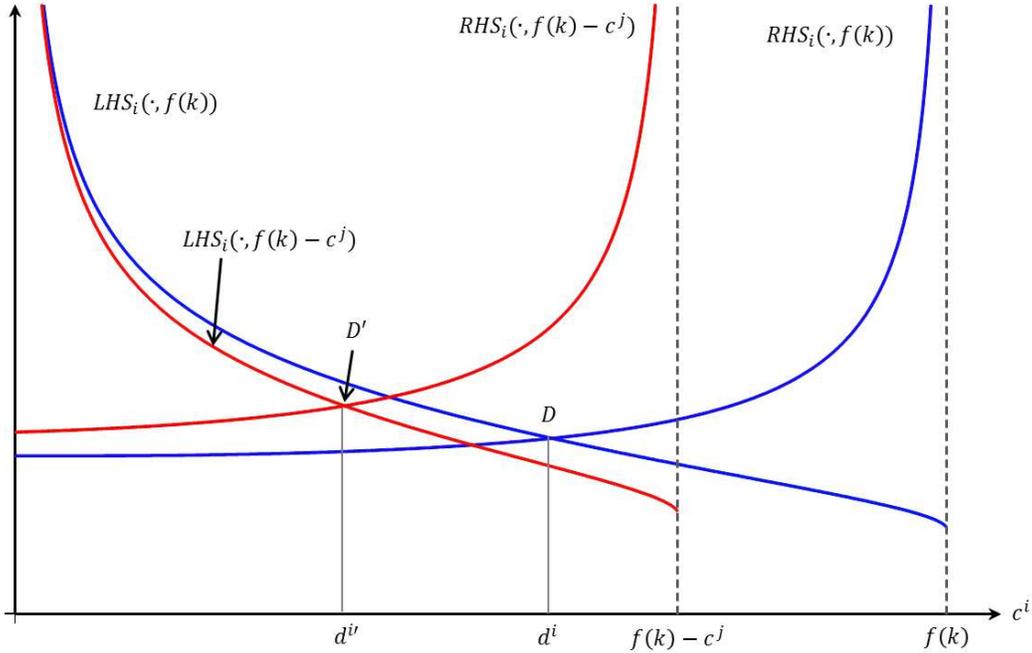}
\caption{The strategy space}
\label{fig:stratspace}
\end{figure}

Now assume that player $j$ enters the game and chooses a  consumption level satisfying $0 < c^j < \theta^j f(k)$. The irruption of player $j$ into the scene has two fundamental effects on player $i$'s program: first, $i$'s total output share falls from one to $\theta^i < 1$; second, the available choices of consumption and saving for player $i$ are more restrictive. Ignoring second-order effects, this can be seen in the graph as horizontal shifts of the $LHS_i$ and $RHS_i$ curves, and also a shift of the vertical asymptote of $RHS_i$ at $f(k)$ to the left to $f(k)-c^j$. As a result, there is an unambiguous decrease in consumption from $d^i$ to $d^{i'}$, which is depicted as point $D'$ in Figure \ref{fig:stratspace}. Note that we have assumed $0 < c^i < f(k)-c^j$ to obtain feasible consumption choices.

Slightly abusing notation, we can define an implicit function $d_i(k)$ from \eqref{eq:EulerU} for all $k \in K$, which is increasing over that interval. For a fixed level of $k$, provided that $0 \leq d_i(k) \leq \theta^if(k)$ on $K$, the strategy space for an individual player can be defined as $S^i:=[0,d_i(k)]$, a convex compact subset of $\mathbb{R}_+$. Then, preferences defined over the product space $S := S^i \times S^j$ exclude infeasible strategies. Now, if $k$ is allowed to vary, there is no guarantee that the resulting map will be convex in $k$. One possibility is to define $S^i$ as the convex hull of this set over $K$. Alternatively, and keeping in mind transitional dynamics, defining an individual strategy space as $S^i:=[0,d_i(k_m)]$ also seems to be a reasonable compromise.

\section{Precommitment Strategies}
\label{sec:precommit}

The game in consideration is a triplet $\mathcal{P}:=\left(N,S,\{\mathrm{W}_0^i\}_{i \in N}\right)$, where $N=\{1,2\}$ is the set of players, $S:=\prod_{i \in N} S^i$ the strategy space, and $\mathrm{W}_0^i$ the payoff function for each player $i \in N$, given by \eqref{eq:Widef}. A precommitment or open-loop strategy for the game $\mathcal{P}$ is a sequence of actions that depend only on the initial state and the date. Both players simultaneously commit at the beginning of the game to a completely specified list of actions to be played without any possibility of revision during the entire course of the game. Given an initial condition $(k_0^i,k_0^j) \geq 0$, each agent chooses an optimal consumption path, taking the strategy of the other player as given, seeking to maximize the discounted sum of instantaneous utilities.

\subsection{Nash Equilibrium}
\label{sec:nasheq}

Our approach to construct equilibria is close in spirit to the one developed by \citet{rosen65} for $n$-person concave games, extended to a dynamic setting. A similar methodology was also applied by \citet{beckerfoias98,beckerfoias07}, who refer to this as the \emph{dynamical system approach}, to solve and characterize an optimal growth problem. One of the main characteristics of this approach is that it works with Euler equations rather than value functions directly. 

Let $\mathrm{W}_0^i(\bds{c}^i|\bds{c}^j)$ denote the  discounted utility that player $i$ obtains at $t=0$ from following a strategy $\bds{c}^i$ when player $j$'s strategy is given by $\bds{c}^j$. A Nash equilibrium is defined next.

\begin{defn}
\label{def:nasheq}
A feasible strategy profile $(\bds{c}^{i*},\bds{c}^{j*}) \in \Pi^i \times \Pi^j$ is a \emph{Nash equilibrium} for the game in  precommitment strategies $\mathcal{P}$ if
\begin{equation*}
\mathrm{W}_0^i(\bds{c}^{i*}|\bds{c}^{j*}) \geq \mathrm{W}_0^i(\bds{c}^i|\bds{c}^{j*}), \qquad \text{for each } \bds{c}^i \in \Pi^i\ \text{and}\ j \neq i =1,2.
\end{equation*}
\end{defn}

Consider the modified sequential problem for player $i$
\begin{align}
\label{eq:SPiprime}
\max_{\{c_t^i,k_{t+1}\}_{t=0}^\infty}\quad &\sum_{t=0}^{\infty} \left(\prod_{s=0}^{t-1}
\alpha_i(c_s^i)\right) u_i(c_t^i)\tag{SP$_i'$}\\[5pt]
\text{s.t. }\quad & c_t^i \leq f(k_t)-c_t^j-k_{t+1},&t=0,1,\ldots\nonumber\\[5pt]
&c_t^i \geq 0,\ k_{t+1} \geq 0, & t=0,1,\ldots\nonumber
\\[5pt]
&k_0 \geq 0,\ \bds{c}^j:=\{c_t^j\}_{t=0}^\infty\ \text{given}\nonumber.
\end{align}
A solution to \eqref{eq:SPiprime} in open-loop strategies will be constructed from the necessary first-order conditions for an interior equilibrium, the Euler equation and envelope condition
\begin{align}
\label{eq:EulerOL}
&u_i'(c_t^i)+\alpha_i'(c_t^i)v_i(f(k_t)-c_t^i-c_t^j)= \alpha_i(c_t^i)v_i'(f(k_t)-c_t^i-c_t^j),\\[5pt]
\label{eq:envelOL}
&v_i'(k_t) = \left[u_i'(c_t^i)+\alpha_i'(c_t^i)v_i(f(k_t)-c_t^i-c_t^j)\right]f'(k_t),
\end{align}
respectively. The second-order condition for optimality is given by
\begin{align*}
u_i''(c_t^i)+\alpha_i''(c_t^i)v_i(k_{t+1})-2\alpha_i'(c_t^i)v_i'(k_{t+1})+\alpha_i(c_t^i)v_i''(k_{t+1}) \leq 0, \quad \text{and}\\[5pt]
v_i''(k_t) = \left[u_i''(c_t^i)+\alpha_i''(c_t^i)v_i(k_{t+1})\right]f'(k_t)^2 + \left[u_i'(c_t^i)+\alpha_i'(c_t^i)v_i(k_{t+1})\right]f''(k_t).
\end{align*}
Our assumptions on preferences and technology imply that $v_i' > 0$ and $v_i''< 0$, hence each player's value function $v_i(k)$, $i=1,2$, is strictly increasing and strictly concave for all $k \in K$.

To ease the notation burden, we define the functions $U^i: S^i \times S^j \times K \to \mathbb{R}$, and $W^i: S^i \times S^j \times K \to \mathbb{R}$ by
\begin{align}
\label{eq:Uisupdef}
U^i(c_t^i,c_t^j,k_t) &:= u_i'(c_t^i)+\alpha_i'(c_t^i)v_i(f(k_t)-c_t^i-c_t^j),\\[5pt]
\label{eq:Wisupdef}
W^i(c_t^i,c_t^j,k_t) &:= u_i''(c_t^i)+\alpha_i''(c_t^i)
v_i(f(k_t)-c_t^i-c_t^j).
\end{align}
Then, an interior solution for the game $\mathcal{P}$ is a feasible sequence $\{(c_t^i,c_t^j,k_t)\}_{t=0}^\infty$ characterized by the following three conditions
\begin{subequations}
\label{eq:focOL}
\begin{align}
\label{eq:focOLi}
U^i(c_t^i,c_t^j,k_t)&= \alpha_i(c_t^i)U^i(c_{t+1}^i,c_{t+1}^j,k_{t+1})f'(k_{t+1}),\\[5pt]
\label{eq:focOLj}
U^j(c_t^j,c_t^i,k_t)&= \alpha_j(c_t^j)U^j(c_{t+1}^j,c_{t+1}^i,k_{t+1})f'(k_{t+1}),\\[5pt]
\label{eq:rcOL}
k_{t+1} &= f(k_t)-c_t^i-c_t^j.
\end{align}
\end{subequations}

Note that $v_i'(k_t)=U^i(c_t^i,c_t^j,k_t)f'(k_t)$, hence the partial derivatives of $U^i$ are given by
\begin{subequations}
\label{eq:partialU}
\begin{align}
\label{eq:partialU1}
U_1^i(c_t^i,c_t^j,k_t)&=W^i(c_t^i,c_t^j,k_t)-\alpha_i'(c_t^i)U^i(c_{t+1}^i,c_{t+1}^j,k_{t+1})f'(k_{t+1}),
\\[5pt]
\label{eq:partialU2}
U_2^i(c_t^i,c_t^j,k_t)&=-\alpha_i'(c_t^i)U^i(c_{t+1}^i,c_{t+1}^j,k_{t+1})f'(k_{t+1}), \\[5pt]
\label{eq:partialU3}
U_3^i(c_t^i,c_t^j,k_t)&=\alpha_i'(c_t^i)U^i(c_{t+1}^i,c_{t+1}^j,k_{t+1})f'(k_{t+1}),
\end{align}
\end{subequations}
where $U_l^i$ denotes de partial derivative of $U^i$ with respect to its $l$-th argument. 

By the implicit function theorem, the first-order conditions \eqref{eq:focOL} for an equilibrium can be written as
\begin{subequations}
\label{eq:FOCOL}
\begin{align}
\label{eq:FOCOLi}
&U^i(c_t^i,c_t^j,k_t)=\alpha_i(c_t^i)\,U^i(F^i(c_t^i,c_t^j,k_t),F^j(c_t^j,c_t^i,k_t),
F^k(c_t^i,c_t^j,k_t))f'(F^k(c_t^i,c_t^j,k_t)),\\[5pt]
\label{eq:FOCOLj}
&U^j(c_t^j,c_t^i,k_t)=\alpha_j(c_t^j)
\,U^j(F^j(c_t^j,c_t^i,k_t),F^i(c_t^i,c_t^j,k_t),
F^k(c_t^i,c_t^j,k_t))f'(F^k(c_t^i,c_t^j,k_t)),
\end{align}
\end{subequations}
for some differentiable functions $F^i:S^i \times S^j \times K \to S^i$ and $F^j:S^j \times S^i \times K \to S^j$, where $F^k:S^i \times S^j \times K \to K$ is given by
\begin{equation}
\label{eq:Fk}
F^k(c_t^i,c_t^j,k_t):=f(k_t)-c_t^i-c_t^j.
\end{equation}
The map $F^i$ indicates agent $i$ how much to consume next period, given  current period's consumption profile $(c_t^i,c_t^j)$ and capital stock $k_t$, and similarly for $F^j$, whereas $F^k$ is the aggregate resource restriction. We use this formulation of the problem to implicitly define a discrete dynamical system. Therefore, the equilibrium conditions 
\begin{align*}
c_{t+1}^i&=F^i(c_t^i,c_t^j,k_t),\\[5pt]
c_{t+1}^j&=F^j(c_t^j,c_t^i,k_t),\\[5pt] 
k_{t+1} &= F^k(c_t^i,c_t^j,k_t),
\end{align*}
define a map $\Phi$ from $S^i \times S^j \times K$ into itself, which can also be written as
\begin{equation}
\label{eq:Phi}
(c_{t+1}^i,c_{t+1}^j,k_{t+1})=\Phi(c_t^i,c_t^j,k_t),
\end{equation}
for all $t$. An equilibrium path can then be defined in terms of this map.

\begin{defn}
\label{def:eqpathOL}
An \emph{equilibrium path} for the game $\mathcal{P}$ is a feasible sequence $\{(c_t^i,c_t^j,k_t)\}_{t=0}^\infty$ such that $(c_t^i,c_t^j,k_t) \in S^i \times S^j \times K$ for all $t$, and $\{(c_{t+1}^i,c_{t+1}^j,k_{t+1})\}_{t=0}^\infty=
\{\Phi^t(c_0^i,c_0^j,k_0)\}_{t=0}^\infty$,
where $\Phi^t$ is the $t$-th iterate of $\Phi$.
\end{defn}

\subsection{Stationary Equilibrium}
\label{sec:stateqOL}

It is easily verified that if $\Phi(\ov{c}^i,\ov{c}^j,\ov{k})=(\ov{c}^i,\ov{c}^j,\ov{k})$ for some $t$, this implies $c_t^i=c_{t+1}^i=\ov{c}^i$, $c_t^j=c_{t+1}^j=\ov{c}^j$, and $k_t=k_{t+1}=\ov{k}$, for all $t$. Hence, a fixed point of $\Phi$ on $S^i \times S^j \times K$ corresponds to a stationary point of the dynamic system given by \eqref{eq:Phi}. The latter, together with \eqref{eq:focOL}, imply that a stationary point $(\ov{c}^i,\ov{c}^j,\ov{k})$ is defined by
\begin{subequations}
\label{eq:stptOL}
\begin{align}
\label{eq:stptOLi}
1&=\alpha_i(\ov{c}^i)f'(\ov{k})\\[5pt]
\label{eq:stptOLj}
1&=\alpha_j(\ov{c}^j)f'(\ov{k}),\\[5pt]
\label{eq:stptOLk}
\ov{k}&=f(\ov{k})-\ov{c}^i-\ov{c}^j.
\end{align}
\end{subequations}

Local properties of the nonlinear system \eqref{eq:Phi} near a hyperbolic stationary point can be examined based on the \emph{stable manifold theorem}. Before proceeding, we need some preliminary calculations. From \eqref{eq:partialU}, the partial derivatives of $U^i$ evaluated at $(\ov{c}^i,\ov{c}^j,\ov{k})$ are given by 
$U_1^i =W^i-(\alpha_i'/\alpha_i)U^i$,
$U_2^i = -(\alpha_i'/\alpha_i)U^i$, and $U_3^i =(\alpha_i'/\alpha_i)U^if'$, where $U^i=u_i'+\alpha'v_i > 0$ and $W^i:=u_i''+\alpha_i''v_i < 0$. Differentiating \eqref{eq:FOCOLi} and \eqref{eq:FOCOLj} we obtain, after appropriate substitutions, the following system of equations in the unknowns $F_l^i$ with $l=1,2,3$ and $j \neq i=1,2$,
\begin{align}
\label{eq:systOL}
\begin{bmatrix}
\omega_i & & \delta_i\\[5pt]
\delta_j & & \omega_j
\end{bmatrix}
\begin{bmatrix}
F_1^i & F_2^i & F_3^i \\[5pt]
F_2^j & F_1^j & F_3^j
\end{bmatrix}
&=
\begin{bmatrix}
\omega_i-\eta_i & -\eta_i           & \eta_i f' \\[5pt]
-\eta_j         & \omega_j - \eta_j & \eta_j f'
\end{bmatrix},
\end{align}
where
\begin{equation}
\label{eq:notatioOL}
\omega_i:=\frac{\alpha_i'}{\alpha_i}-\frac{W^i}{U^i}>0, \quad
\delta_i:= \frac{\alpha_i'}{\alpha_i}>0, \quad \eta_i:=\frac{\alpha_i'}{\alpha_i}(f'-1)+\frac{f''}{f'}.
\end{equation}
Note the relation between $\eta_i$ and equation \eqref{eq:eta1st}, which determines the stability of the autarky equilibrium for each player. These conditions have a central role in the upcoming analysis. 

The determinant of the coefficient matrix in \eqref{eq:systOL} is given by
$\Delta_0:=\omega_i\omega_j-\delta_i\delta_j > 0$, so we can solve the system for $F_l^i$ and $F_l^j$, $l=1,2,3$, which yields
\begin{subequations}
\label{eq:Fls}
\begin{align}
\label{eq:F1}
F_1^i&=\frac{1}{\Delta_0}\left[\omega_i\omega_j-(\eta_i\omega_j-\delta_i\eta_j)\right], &
F_1^j&=\frac{1}{\Delta_0}\left[\omega_i\omega_j-(\omega_i\eta_j-\eta_i\delta_j)\right],\\[5pt]
\label{eq:F2}
F_2^i&=-\frac{1}{\Delta_0}\left[\delta_i\omega_j+(\eta_i\omega_j-\delta_i\eta_j)\right],&
F_2^j&=-\frac{1}{\Delta_0}\left[\omega_i\delta_j+(\omega_i\eta_j-\eta_i\delta_j)\right],\\[5pt]
\label{eq:F3}
F_3^i&=\frac{1}{\Delta_0}\left(\eta_i\omega_j-\delta_i\eta_j\right)f',&
F_3^j&=\frac{1}{\Delta_0}\left(\omega_i\eta_j-\eta_i\delta_j\right) f'.
\end{align}
\end{subequations}
There is a nice symmetry in these formulas: any partial derivative of $F^j$ can be obtained by switching subscripts in the corresponding partial derivative of $F^i$, and viceversa.

It is of interest to determine the signs of the equations in \eqref{eq:Fls} to gain some intuition on the model and analyze strategic interactions. For this, having a framework of reference can be quite useful. Note that in the limiting case where $\alpha_i$ and $\alpha_j$ approach constant functions, the current model will tend to behave like a model with constant discounting, at least in a qualitative sense. In other words, if we take $\alpha_i'\to 0$, the conditions given in \eqref{eq:notatioOL} will approximate to
\begin{equation*}
\omega_i=-u_i''/u_i' > 0, \quad \delta_i=0,\quad \text{and} \quad \eta_i = f''/f' < 0, 
\end{equation*}
and analogously for $\alpha_j' \to 0$. 

Using a model similar to ours, but assuming preferences with fixed heterogeneous discount factors, \citet{beckerfoias07} proved that these conditions lead to an asymptotic equilibrium where the most patient player holds all the capital in the economy, and the equilibrium capital path is eventually monotonic, the so called turnpike property. The novelty of their result is that optimal paths satisfy the turnpike property \emph{with or without strategic interaction}. Our strategy will be to keep the parameter values\footnote{This is a slight abuse of terminology, since $\delta_i$, $\omega_i$ and $\eta_i$ are not exactly parameters, but expressions derived from the model's primitives evaluated at a stationary point $(\ov{c}^i,\ov{c}^j,\ov{k})$. If we were to assume specific functional forms for the fundamentals, they would ultimately depend on certain parameters.} relatively close to the case of constant discounting, i.e., $\omega_i > 0$, $\eta_i < 0$, and, if needed for comparison purposes, $\delta_i$ will be taken to be small, but not zero. In any circumstance, the assumption that discount factors satisfy $\alpha_i \geq \alpha_j$ will remain unchanged.

\subsection{Local Regularity and Cross-Marginal Impatience}
\label{sec:LRCMI}

Note that the sign of all derivatives of $F^i$ in \eqref{eq:Fls} depend on the term $(\eta_i\omega_j - \delta_i\eta_j)$. For certain combinations of these parameters, we are able to determine some of those signs. In fact, they depend on two factors which we identify with two properties called ``local regularity'' and ``cross-marginal impatience,''
\begin{align}
\eta_i\omega_j-\delta_i\eta_j &= \left[\frac{\alpha_i'}{\alpha_i}(f'-1)+\frac{f''}{f'}\right]\left(\frac{\alpha_j'}{\alpha_j}-\frac{W^j}{U^j}\right)-\left(\frac{\alpha_i'}{\alpha_i}\right)\left[\frac{\alpha_j'}{\alpha_j}(f'-1)+\frac{f''}{f'}\right],
\nonumber\\[5pt]
\label{eq:crossimp}
&=\underbrace{\left[\frac{\alpha_i'}{\alpha_i}(f'-1)+\frac{f''}{f'}\right]}_\text{local regularity}\left(-\frac{W^j}{U^j}\right)+\underbrace{\left(\frac{\alpha_j'}{\alpha_j}-\frac{\alpha_i'}{\alpha_i}\right)}_\text{\shortstack{cross-marginal \\ impatience}}\left(\frac{f''}{f'}\right).
\end{align}
In what follows we give formal definitions and characterize these concepts.

\begin{defn}
A stationary equilibrium $(\ov{c}^i,\ov{c}^j,\ov{k}) \in S^i \times S^j \times K$ for the game $\mathcal{P}$ satisfies \emph{local regularity} (LR) if there is an open rectangle $I^i \times I^j \times I^k \subset S^i \times S^j \times K$ containing $(\ov{c}^i,\ov{c}^j,\ov{k})$ such that
\begin{equation}
\label{eq:LRdef}
\frac{\alpha_i'(c^i)}{\alpha_i(c^i)}\left(f'(k)-1\right)+\frac{f''(k)}{f'(k)}<0, \qquad i=1,2,
\end{equation}
holds for each $(c^i,c^j,k) \in I^i \times I^j \times I^k$.
\end{defn}

Local regularity is equivalent to $\eta_i < 0$, $i=1,2$, which is now a familiar condition. Loosely speaking, LR states that a stationary equilibrium for the game $\mathcal{P}$ lies in a neighborhood of each player's autarky equilibrium, in a region where each of the latter are locally stable. This intuition is confirmed in Section \ref{sec:discussOL} with two results. In Proposition~\ref{prp:localreg}, it is shown that $\ov{k}$ lies below the levels of autarky equilibria $k_a^i$ and $k_a^j$. On the other hand, Proposition~\ref{prp:locstabOL} proves that LR implies that $(\ov{c}^i,\ov{c}^j,\ov{k})$ satisfies a (weak) local stability condition. 

Cross-marginal impatience (CMI) measures how the marginal rate of intertemporal substitution between current and future consumption for one player changes in response to a change in the strategy followed by the other player. Given that the analysis is carried out in a neighborhood of a stationary point, CMI is defined in terms of discount factors $(\alpha_i,\alpha_j)$ only. For this reason, not much is lost if we consider two consecutive periods, and define the  intertemporal substitution between $c_t^i$ and $c_{t+1}^i$ for player $i$ as
\begin{align}
\label{eq:MRSdef}
\mu_{t,t+1}^i:=\frac{U^i(c_t^i,c_t^j,k_t)}{\alpha_i(c_t^i)U^i(c_{t+1}^i,c_{t+1}^j,k_{t+1})}, \qquad t=0,1,\ldots
\end{align}
For a stationary path, this is simply $\ov{\mu}^i=\left[\alpha_i(\ov{c}^i)\right]^{-1}$. The obvious distinction between this formulation and a constant discount factor, say $0 < \beta < 1$, where the marginal rate of substitution $\beta^{-1}$ is also constant, is that $\ov{\mu}^i$ decreases with the level of stationary consumption. 

Before characterizing CMI, in order to capture various aspects of marginal impatience, consider the following exercise.  Suppose that \eqref{eq:LRdef} and $\alpha_j'(c)/\alpha_j(c)-\alpha_i'(c)/\alpha_i(c) \geq 0$ hold for all $c$ in some interval $I$, so the right-hand side of \eqref{eq:crossimp} is negative. Next, differentiate \eqref{eq:MRSdef} and evaluate the result at $(\ov{c}^i,\ov{c}^j,\ov{k})$. After substitutions, we have
\begin{align*}
\frac{\partial{\mu_{t,t+1}^i}}{\partial{c_t^i}}&=-\left(2\,\frac{\alpha_i'}{\alpha_i^2}-\frac{W^i}{\alpha_iU^i}\right) < 0, &  \frac{\partial{\mu_{t,t+1}^i}}{\partial{c_{t+1}^i}}&=\frac{\alpha_i'}{\alpha_i^2}-\frac{W^i}{\alpha_iU^i} > 0,
\\
\frac{\partial{\mu_{t,t+1}^i}}{\partial{c_t^j}}&=-\frac{\alpha_i'}{\alpha_i^2} < 0, &  \frac{\partial{\mu_{t,t+1}^i}}{\partial{c_{t+1}^j}}&=\frac{\alpha_i'}{\alpha_i^2} > 0,\\
\frac{\partial{\mu_{t,t+1}^i}}{\partial{k_t}}&=\frac{\alpha_i'}{\alpha_i^3} > 0, & \frac{\partial{\mu_{t,t+1}^i}}{\partial{k_{t+1}}}&=-\frac{\alpha_i'}{\alpha_i^3}< 0.
\end{align*}
The first line of partial derivatives above reflects the property of decreasing marginal impatience (DMI): agents become less impatient if  future consumption is substituted for current consumption. The second line contains cross-partial derivatives, which are naturally related to the concept of CMI:  they capture each player's best response to changes in the strategy of the other player. In this scenario, player $i$ gets more patient as player $j$ gives up future consumption for current consumption, hence preferences satisfy \emph{decreasing cross-marginal impatience} (DCMI).  Finally, impatience tends to decrease as current capital replaces future capital. There is an interesting observation about the size of these effects. From the second and third lines, note that increasing $c_t^j$ and $c_{t+1}^j$ (or $k_t$ and $k_{t+1}$) by the same amount has a null effect on impatience for player $i$, they simply cancel each other out. Therefore, changes in these variables that tend to have permanent effects will not affect impatience. But adding up the values on the first line above implies that equally increasing $c_t^i$ and $c_{t+1}^i$ still makes agent $i$ less impatient (the net effect in absolute value is $\alpha_i'/\alpha_i^2$), since the discount factor increases utility in every future period. We are ready to formally define the second property.

\begin{defn}
\label{def:NUCMI}
Assume that $\alpha_i$, $i=1,2$ satisfy \ref{asm:alpha}--\ref{asm:alphap}. Let $I \in S^i \cap S^j$ be an open interval such that $\alpha_i(c) \geq \alpha_j(c)$ for all $c \in I$. It is said that $(\alpha_i,\alpha_j)$ exhibits \emph{non-upward cross-marginal impatience} (NUCMI) if
\begin{equation}
\label{eq:NUCMI}
\frac{\partial{\mu_{t,t+1}^i}(c)}{\partial{c_t^j}} \geq \frac{\partial{\mu_{t,t+1}^j}(c)}{\partial{c_t^i}}, \qquad \text{all}\ c \in I.
\end{equation}
The property of \emph{non-downward cross-marginal impatience}  (NDMI) is defined analogously but with the inequality reversed.\footnote{A related condition, albeit presented in statistical terms, called \emph{non-upward-crossing}, has been recently introduced by \citet{chadeswink16} in a study of the first-order approach to the classical moral hazard problem.}
\end{defn}

While DCMI relates to a first-order effect on $\mu_{t,t+1}^i$ from a change in current consumption, NUCMI measures how cross-marginal impatience varies with the level of impatience, so that if impatience were a continuous variable, it would be characterized by a second cross-partial derivative.\footnote{This should not be confused with the following second cross-partial derivative
\begin{equation*}
\frac{\partial^2{\mu_{t,t+1}^i}}{\partial{c_t^j}\partial{c_t^i}}=\frac{U_{12}^i}{\alpha_iU^i}\left(1-\frac{\alpha_i'}{\alpha_i}\right)=-\frac{\alpha_i'}{\alpha_i^2}\left(1-\frac{\alpha_i'}{\alpha_i}\right),
\end{equation*}
which depends on $\alpha_i$ but not on $\alpha_j$.
}
An intuitive way to understand NUCMI, is that when \eqref{eq:NUCMI} holds, then
\begin{align*}
\left|\left(\frac{1}{\alpha_i(c^i)}\right)'\right|=\frac{\alpha_i'(c^i)\hfill}{\alpha_i(c^i)^2} \leq \frac{\alpha_j'(c^j)\hfill}{\alpha_j(c^j)^2}=\left|\left(\frac{1}{\alpha_j(c^j)}\right)'\right|, \qquad \text{for all}\ c^i=c^j=c,
\end{align*}
that is, in absolute value, the slope of the curve $(1/\alpha_j)$ must be larger (i.e., more inelastic) than the slope of $(1/\alpha_i)$ for equal levels of consumption. See Figure \ref{fig:stateqAU} for a graphical interpretation of this condition.

The following example analyzes the implications of assuming LR and NUCMI on the model developed in this section. First, we have from \eqref{eq:Fls} that $\partial{c_{t+1}^i}/\partial{c_t^i}=F_1^i > 0$, so future consumption for player $i$ increases after a small increase in $c_t^i$. Current income does not change, so capital accumulation $k_{t+1}$ must fall. As future income $f(k_{t+1})$ also falls, there is a negative income effect (in the sense of permanent income) that tends to decrease $c_{t+1}^i$. A lower level of accumulation in turn increases  $f'(k_{t+1})$, which tends to increase $c_{t+1}^i$. It seems that there should be a rather strong substitution effect for this change to generate incentives for future consumption. On the other hand, $\partial{c_{t+1}^i}/\partial{k_t}=F_3^i < 0$, so future consumption is negatively related with current capital stock. An increase in $k_t$ has a positive income effect that tends to increase $c_{t+1}^i$, but in the current period it can only decrease investment, since $(c_t^i,c_t^j)$ remain constant by assumption. Again, we have an income effect and a substitution effect operating in opposite directions with respect to $c_{t+1}^i$. The former should be quite strong to explain the negative sign. Based on standard income and substitution effects, these two results can not be reconciled. We are not able to determine the sign of $F_2^i$, unless additional assumptions are made. And the same applies to $F_l^j$, $l=1,2,3$.

What distinctive features of our model could help to explain the apparent contradiction described above? An effect that strengthens the substitution effect between current and future consumption is the presence of decreasing marginal impatience: higher $c_t^i$ implies a higher discount factor $\alpha_i$, which favors consumption in future periods. Additional forces arise from the strategic interaction and its effects on the strategy space. An increase in $c_t^i$ decreases the marginal rate of substitution for player $j$, and this tends to decrease future consumption $c_{t+1}^j$. Taking into account our earlier discussion in Section \ref{sec:stratspace}, the strategy space in $(t+1)$ for player $i$ expands as a result of lower $c_{t+1}^j$, which tends to rise $c_{t+1}^i$. Moreover, the future marginal rate of substitution for $i$ increases. Although this effect is more distant in time, under NUCMI it could be strong enough to increase $c_{t+1}^i$. All these effects increase future consumption, as predicted by the positive value for $F_1^i$. Similar arguments can be used to explain how these additional channels contribute to a negative sign of $F_3^i$.

There are obviously additional direct and indirect effects not included in the previous analysis. They all act together in complex intertwined ways. But some of these new channels offer plausible explanations for the predictions of the model under particular assumptions. It is worth mentioning that strategy alone without recursive preferences is not enough to generate long-lived effects from these additional channels, for two reasons. The first one has to do with the turnpike property (more specifically, Ramsey's conjecture), since constant discount factors prevent any cross-marginal effects. Second, the reinforcing effects that arise from changes in the strategy space will be absent.

\subsection{Discussion}
\label{sec:discussOL}

In this section we discuss and extend the previous results to obtain further insight into the characterization of open-loop equilibria. We begin with a comparison between each player's autarky stationary equilibrium and the stationary equilibrium of the precommitment game.

\begin{prop}
\label{prp:localreg}
Suppose there is an open interval $I \subset S^i \cap S^j$ such that $\alpha_i(c) \geq \alpha_j(c)$ for all $c \in I$, and LR is satisfied on $\Omega:=I \times I \times K$. If the game $\mathcal{P}$ has a non-trivial stationary equilibrium $(\ov{c}^i,\ov{c}^j,\ov{k})$ on $\Omega$, then the following holds
\begin{enumerate}[label=\emph{(\roman*)},leftmargin=*]
\item $\ov{c}^i \leq \ov{c}^j$
\item $\ov{k} \leq k_a^j \leq k_a^i$
\item $c_a^i \leq \ov{c}^i$ and $c_a^j \leq \ov{c}^j$,
\end{enumerate}
where $(c_a^i,k_a^i)$ and $(c_a^j,k_a^j)$ are the stationary equilibria for autarky programs $\mathcal{A}^i$ and $\mathcal{A}^j$, respectively.
\end{prop}

\begin{figure}[h!]
\centering
\includegraphics[width=0.75\textwidth]{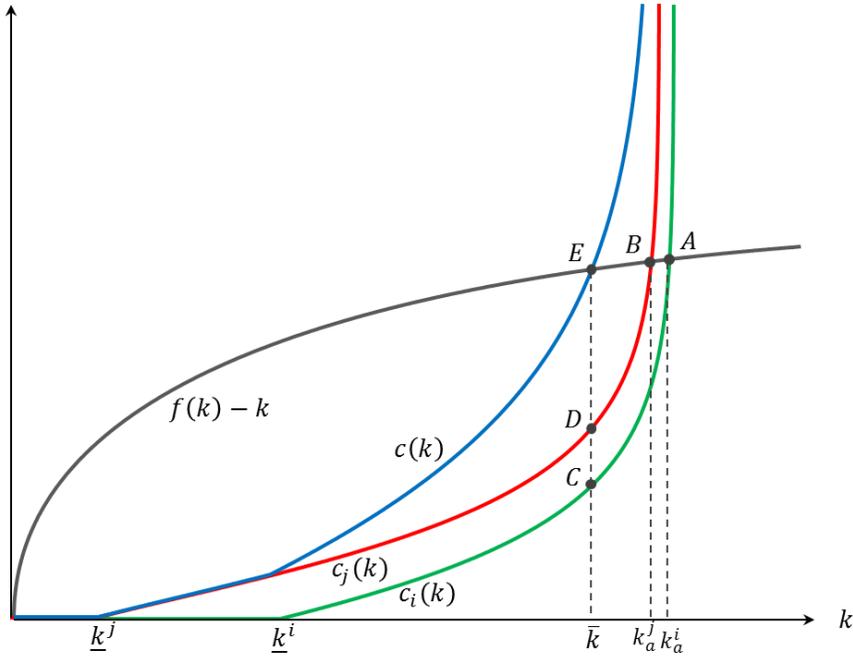}
\caption{Open-loop stationary equilibrium}
\label{fig:stateqOL}
\end{figure}

The conclusions of Proposition \ref{prp:localreg} can be  illustrated as follows. From the equilibrium condition \eqref{eq:stptOL}, we can write 
\begin{equation*}
c_i(k)=\alpha_i^{-1}\left(\frac{1}{f'(k)}\right) \quad \text{and} \quad c_j(k)=\alpha_j^{-1}\left(\frac{1}{f'(k)}\right),
\end{equation*}
where $\alpha_i^{-1}(\cdot)$ and $\alpha_j^{-1}(\cdot)$ denote  inverse functions. These functions are well defined since $\alpha_i$, $\alpha_j$ and $f'$ are strictly monotone. As a result, $c_i(k)$ and $c_j(k)$ are both strictly increasing. Moreover, it is easy to see that $c_i(k) \leq f(k)-k$ for all $0 \leq k \leq k_a^i$ and $c_j(k) \leq f(k)-k$ for all $0 \leq k \leq k_a^j$ holds. A stationary equilibrium for game $\mathcal{P}$ can be represented graphically as the intersection of two curves: aggregate consumption, defined by $c(k):=c_i(k)+c_j(k)$, and aggregate net output, which is $f(k)-k$. This is depicted as point $E$ in Figure \ref{fig:stateqOL}, whereas $C$ and $D$ are  the consumption levels $c_i(\ov{k})$ and $c_j(\ov{k})$ in such equilibrium. Points $A$ and $B$ represent a stationary autarky equilibrium for players $i$ and $j$, respectively, that correspond to the same points in Figure \ref{fig:stateqAU}. Note the lower bounds $\un{k}^i$ and $\un{k}^j$, which satisfy the following conditions $ \alpha_i(0)f'(\un{k}^i)=1$ and $\alpha_j(0)f'(\un{k}^j)=1$.

In a neighborhood of a stationary point $\ov{k} \in (0,k_m)$ it is easily verified that
\begin{equation}
\label{eq:cikprime}
c_i'(\ov{k})=-\frac{f''(\ov{k})/f'(\ov{k})}{\alpha_i'(c_i(\ov{k}))/\alpha_i(c_i(\ov{k}))} \quad \text{and} \quad c_j'(\ov{k})=-\frac{f''(\ov{k})/f'(\ov{k})}{\alpha_j'(c_j(\ov{k}))/\alpha_j(c_j(\ov{k}))},
\end{equation}
therefore NUCMI implies that the $c_i(k)$ curve is steeper than $c_j(k)$ near $\ov{k}$.\footnote{Note that LR implies that $c_i'(k) > (f'(k)-1)$ near $\ov{k}$, since by  \eqref{eq:cikprime}
\[
\frac{\alpha_i'(c_i(\ov{k}))}{\alpha_i(c_i(\ov{k}))}\,c_i'(\ov{k})+\frac{f''(\ov{k})}{f'(\ov{k})}=0.
\]
But replacing $c_i'$ with $(f'-1)$ above, and letting $k$ vary on $I^i$, transforms this into a strict inequality, hence the result follows.}

Before closing this section, we extend to the current framework the approach used to analyze the local stability of autarky equilibria in Section \ref{sec:stratspace}. For the  game this is a weaker notion of stability, because the responses to any perturbation are limited to the discount factor $\alpha_i(c^i)$ and the marginal product of capital $f'(k)$. To perform a full-fledged stability analysis, we need to characterize the stable manifold. Later we show that this is a rather complicated task. 

Let $(\ov{c}^i,\ov{c}^j,\ov{k})$ be a stationary equilibrium given by \eqref{eq:stptOL}. From the implicit function theorem, there is an open rectangle $I^i \times I^j \times I^k$ containing $(\ov{c}^i,\ov{c}^j,\ov{k})$, and continuously differentiable functions $c_i:I^k \to I^i$ and $c_j:I^k \to I^j$ such that
\begin{align*}
\alpha_i\left[F^i(c_i(\ov{k}),c_j(\ov{k}),\ov{k})
\right]f'(\ov{k})=1,\\[5pt]
\alpha_j\left[F^j(c_j(\ov{k}),c_i(\ov{k}),\ov{k})
\right]f'(\ov{k})=1,\\[5pt]
f(\ov{k})-c_i(\ov{k})-c_j(\ov{k})=\ov{k}.
\end{align*}
Let $s(\ov{k}):=f(\ov{k})-c_i(\ov{k})-c_j(\ov{k})$ denote the aggregate savings function. The local stability condition is obtained in two steps: first, implicitly differentiating the first two lines of the above system of equations, which yields  $c_i'(\ov{k})$ and $c_j'(\ov{k})$; and, second, verifying that in a neighborhood of the stationary point, $|s'(\ov{k})| < 1$ holds. It is readily verified that $s'(\ov{k}) > 0$, so the local stability condition is equivalent to
\begin{equation}
\label{eq:locstabOL}
\frac{\alpha_i'}{\alpha_i}\,\frac{\alpha_j'}{\alpha_j}\left(f'-1\right)+\left(\frac{\alpha_i'}{\alpha_i}+\frac{\alpha_j'}{\alpha_j}\right)\frac{f''}{f'} < 0.
\end{equation}
The following result shows that LR is sufficient for local stability.

\begin{prop}
\label{prp:locstabOL}
Let $(\ov{c}^i,\ov{c}^j,\ov{k}) \in  S^i \times S^j \times K$ be a stationary equilibrium for the game $\mathcal{P}$ that satisfies LR on an open rectangle $I^i \times I^j \times I^k \subset S^i \times S^j \times K$. Then, the stationary equilibrium is locally stable.
\end{prop}

\begin{proof}
See Appendix \ref{sec:appOL}.
\end{proof}

\subsection{Local Stable Manifold}
\label{sec:manifoldOL}

To shorten notation, let $x_t:=(c_t^i,c_t^j,k_t)$ for all $t$, so for a given initial condition $x_0 \geq 0$, an equilibrium path is obtained from successive iterations of the map
\begin{equation}
\label{eq:Phidef2}
x_{t+1} = \Phi(x_t), \qquad t=0,1,\ldots
\end{equation}
We stated earlier that if $\ov{x}=(\ov{c}^j,\ov{c}^i,\ov{k})$ is a fixed point of $\Phi$ in $X:=S^i \times S^j \times K$, then $\ov{x}$ is a stationary equilibrium for the game $\mathcal{P}$. The local analysis of a nonlinear dynamical system determine conditions under which an equilibrium path converges to a stationary point $\ov{x}$. That is,  $\lim_{t \to \infty} \Phi^t(x_0) = \ov{x}$, for $x_0$ sufficiently close to this stationary point. 

Some additional assumptions besides \ref{asm:alpha}--\ref{asm:uprime} and \ref{asm:prodf}--\ref{asm:prodfp} are needed for this result, so we state them in a common notation below:
\begin{enumerate}[label=(P\arabic*)]
\item\label{asm:etaneg} $\eta_i < 0$ and $\eta_j < 0$,
\item\label{asm:crossimp} $\delta_i \leq \delta_j$,
\item\label{asm:postrace} $f'-(2\omega_i\omega_j+(\delta_i-\omega_i)\eta_j+\eta_i(\delta_j-\omega_j)) \geq 0$.
\end{enumerate}
Note that \ref{asm:etaneg} and \ref{asm:crossimp} are local regularity and non-upward cross-marginal impatience, respectively, introduced in the previous section. Assumption \ref{asm:postrace} guarantees that the characteristic polynomial associated to the linearized system $\Phi$ in a neighborhood of $\ov{x}$ has real roots.

\begin{thrm}
\label{thm:manifoldOL}
Let $\ov{x}$ be a stationary equilibrium of \eqref{eq:Phidef2} and assume that \ref{asm:etaneg}--\ref{asm:postrace} are satisfied. Then, there exists a neighborhood $\mathcal{N}$ of $\ov{x}$ and a continuously differentiable function $\phi:\mathcal{N} \to \mathbb{R}_+^2$, for which the matrix $[\phi_l^i(\ov{x})]$ has full rank, such that if $\{x_t\}_{t=0}^\infty$ is an equilibrium path for $\mathcal{P}$ with $x_0 \in \mathcal{N}$ and $\phi(x_0)=0$, then $\lim_{t \to \infty} x_t = \ov{x}$.
\end{thrm}

Details of the proof can be found in Appendix \ref{sec:appOL}, so we present an outline of the main steps. First, we evaluate a linearized version of the dynamical system \eqref{eq:Phidef2} and find the associated eigenvalues as the roots of a third-degree polynomial. Second, we verify that the conditions for the stable manifold theorem hold. And third, we briefly explain how the stable manifold can be characterized.

We begin by taking a first-order approximation of $\Phi$ near $\ov{x}$ and finding the eigenvalues of the Jacobian matrix $D\Phi(\ov{x})$. Let $A$ denote the $3 \times 3$ matrix defined by
\begin{equation*}
A:=D\Phi(\ov{x})=
\begin{bmatrix}
\ov{F}_1^i & \ov{F}_2^i & \ov{F}_3^i\\
\ov{F}_2^j & \ov{F}_1^j & \ov{F}_3^j\\
\ov{F}_1^k & \ov{F}_2^k & \ov{F}_3^k\\
\end{bmatrix}
,
\end{equation*}
where all entries are evaluated at $\ov{x}$, which is denoted by placing a bar over each entry. Note that the first two rows of $D\Phi$ are given by \eqref{eq:Fls}, and the third row is obtained by differentiating \eqref{eq:Fk} and evaluating the result at $\ov{x}$, which yields
\begin{equation*}
\ov{F}_1^k = -1, \quad \ov{F}_2^k = -1, \quad \text{and} \quad\ov{F}_3^k = f'(\ov{k}).
\end{equation*}
In order to find eigenvalues of $A$, note that its trace and determinant are
\begin{align}
\label{eq:trA}
\mathrm{tr}(A)&=\frac{1}{\Delta_0}\left[2\omega_i\omega_j -(\eta_i\omega_j-\delta_i\eta_j)-(\omega_i\eta_j-\eta_i\delta_j)\right]+f',\\[5pt]
\label{eq:detA}
\det(A)&=\frac{\omega_i\omega_j}{\Delta_0}f'.
\end{align}
By the Cayley-Hamilton theorem, the coefficients of the characteristic polynomial can be expressed in terms of traces of powers of $A$, one of them being the determinant of $A$.\footnote{The theorem states that any invertible $n \times n$ matrix $A$ over the real field satisfies its own characteristic polynomial.} In particular, for $n=3$, the characteristic polynomial is given by
\begin{equation}
\label{eq:poly3}
p(\lambda):=\lambda^3 - \tr(A)\lambda^2 + \tfrac{1}{2}\left(\tr^2(A)-\tr(A^2)\right)\lambda - \det(A),
\end{equation}
and has the eigenvalues of $A$ as roots.

As shown in Figure \ref{fig:3degpoly}, the characteristic polynomial satisfies $p(-1) < 0$, $p(0) < 0$ and $p(1) > 0$. Moreover, it reaches local extrema at points $M$ and $m$ that correspond to a local maximum at $r_1$ and a local minimum at $r_2$, and a point of inflection at $r_3$ (not depicted) with  $0 < r_1 < 1 < r_3 < r_2$. In conclusion, the eigenvalues associated to $p(\lambda)$ are real and  satisfy
\begin{equation}
\label{eq:eigenvOL}
0 < \ov{\lambda}_1 < 1 < \ov{\lambda}_2 < \ov{\lambda}_3.
\end{equation}
The fact that all eigenvalues have moduli different than one implies that $\ov{x}$ is a hyperbolic fixed point.

\begin{figure}[h!]
\centering
\includegraphics[width=0.8\textwidth]{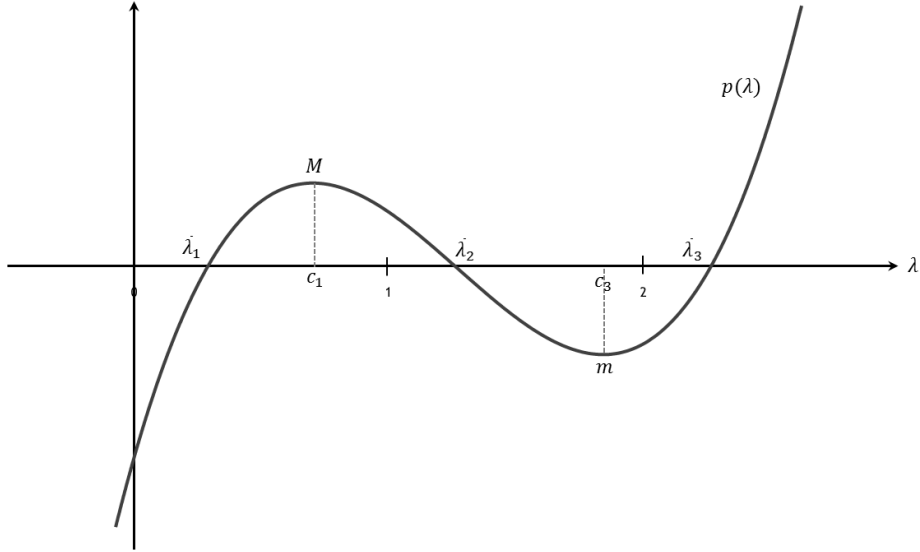}
\caption{Characteristic polynomial and eigenvalues}
\label{fig:3degpoly}
\end{figure}

Since $\det(A) \neq 0$, this matrix is locally invertible and, given that all elements of $D\Phi(\ov{x})$ exist and are continuous in some neighborhood $\mathcal{U}$ of $\ov{x}$, the map $\Phi$ is a local diffeomorphism.\footnote{Given two manifolds $M,N$, a differentiable map $\Phi:M \to N$ is a diffeomorphism if it is a bijection and its inverse $\Phi^{-1}$ is also differentiable.} Consequently, there exists a neighborhood $\mathcal{N} \subset \mathcal{U}$ and a continuously differentiable function $\phi: \mathcal{N} \to \mathbb{R}^2$. It remains to verify that the Jacobian matrix of $\phi$
\begin{equation*}
D\phi(\ov{x}):=
\begin{bmatrix}
\phi_1^i(\ov{x}) & \phi_2^i(\ov{x})\\[5pt]
\phi_1^j(\ov{x}) & \phi_2^j(\ov{x})
\end{bmatrix}
\end{equation*}
has full rank. After lengthy calculations, these derivatives can be solved analytically in terms of the model's parameters and the eigenvalues $\ov{\lambda}_l$, $l=1,2,3$, but the results turn out to be long and complicated to characterize. A detailed analysis of the stable manifold is available in Appendix \ref{sec:appOL}.

\section{Markovian Strategies}
\label{sec:Markov}

The game in Markovian strategies, or simply, the \emph{Markovian game}, is described by a tuple
$\mathcal{M}:=\left(N,K,\tau,\{A^i,\mathrm{W}^i\}_{i \in N}\right)$, where $N=\{1,2\}$ as before, $K$ is the state space, $\tau$ a transition function, and for each $i \in N$, $A^i$ represents the action space, and $\mathrm{W}^i$ are the payoff functions. In particular, we have that $K=[0,k_m]$, $\tau_t:=f(k_t)-c_t^i-c_t^j$, $A_i:=[0,\theta^if(k_m)]$, and $\mathrm{W}^i$ is given in recursive form by \eqref{eq:Wirec}.

 Let $g_i:K \to K$ denote a stationary Markovian strategy for player $i$ and $V_i: K \to \mathbb{R}_+$ the value function associated with \eqref{eq:mainOL} under this information structure. Both functions belong to the space of real valued $C^2(K)$ functions. Optimal strategies are characterized by the first-order and envelope conditions, which are the analogues to \eqref{eq:EulerOL}--\eqref{eq:envelOL} in the precommitment case,
\begin{align}
\label{eq:EulerM}
&u_i'(g_i(k))+\alpha_i'(g_i(k))V_i(f(k)-g_i(k)-g_j(k))
=\alpha_i(g_i(k))V_i'(f(k)-g_i(k)-g_j(k)),\\[5pt]
\label{eq:envelM}
&V_i'(k)=\left[u_i'(g_i(k))+\alpha_i'(g_i(k))V_i(f(k)-g_i(k)-g_j(k))\right]\left[f'(k)-g_j'(k)\right].
\end{align}
Note that $g_j'(k)$ enters player $i$'s equilibrium condition, adding a new element of strategic interaction. It is easy to show that $V_i$ is strictly increasing, which implies $f'(k)-g_j'(k)>0$ in equilibrium. Moreover, later we show that $g_j'(k) > 0$ and $0 < f'(k)-g_j'(k) < f'(k)$ hold in a neighborhood of any stationary point, hence the marginal value of an additional unit of capital for player $i$, measured by $V_i'(k)$, decreases compared to the precommitment solution. The last optimality condition is given by the accumulation equation for aggregate capital, $k' = f(k)-g_i(k)-g_j(k)$, where $k'$ denotes next period's value.

Combining \eqref{eq:EulerM} with \eqref{eq:envelM} and differentiating the result, we obtain the second-order condition for an optimum
\begin{align}
\label{eq:SOCEulMi}
u_i''(g_i(k))&+\alpha_i''(g_i(k))V_i(k')
-2\alpha_i'(g_i(k))V_i'(k')+\alpha_i(g_i(k))V_i''(k') \leq 0,
\end{align}
where $V_i'$ is given by \eqref{eq:envelM} and
\begin{align}
\label{eq:SOCenvMi}
V_i''(k)&=[u_i''(g_i(k))+\alpha_i''(g_i(k))V_i(k')](f'(k)-g_j'(k))^2 \nonumber\\
&+[u_i'(g_i(k))+\alpha_i'(g_i(k))V_i'(k')](f''(k)-g_j''(k)).
\end{align}
Hence \eqref{eq:SOCEulMi} is satisfied if the value function is increasing and concave, the latter depending on the sign of $(f''(k)-g_j''(k))$ at the optimum. In other words, a sufficient condition for the concavity of player $i$'s value function  $V_i$ is that player $j$'s optimal strategy $g_j$ must be less concave than the production function $f$. However, the condition can be relaxed, since the first term in \eqref{eq:SOCenvMi} is strictly negative under the current assumptions.

Optimality conditions can be written in terms of the functionals defined in Section \ref{sec:nasheq}. For $x \in K$, $y \in [0,f(x)]$, and $z \in [0,f(y)]$, use \eqref{eq:envelM} to eliminate $V_i'(y)$ from \eqref{eq:EulerM}, so the corresponding first-order conditions for each player can be written as
\begin{subequations}
\label{eq:FOCM}
\begin{align}
\label{eq:FOCMi}
U^i(g_i(x),g_j(x),x)&=\alpha_i(g_i(x))U^i(g_i(y),g_j(y),y)\left[f' (y)-g_j'(y)\right],\\[5pt]
\label{eq:FOCMj}
U^j(g_j(x),g_i(x),x)&=\alpha_j(g_j(x))U^j(g_j(y),g_i(y),y)\left[f'(y)-g_i'(y))\right].
\end{align}
\end{subequations}
There is no obvious method of attack on this particular system of differential functional equations. Proving existence seems quite challenging and raises an interesting issue which might be worth exploring in future research. Here we describe an approach that seems suitable for the problem at hand and has not been applied in economics, at least that we are aware of. The first step consists in transforming \eqref{eq:FOCM} into a system of integral equations. Then, a solution is obtained as a \emph{coupled fixed point} of a nonlinear multivalued mapping on an appropriately defined space.

The concept of a coupled fixed point was introduced by \citet{guolaksh87}, and the theory of coupled fixed points in ordered metric spaces is a growing research field, recently developed by \citet{bhasklaksh06}, \citet{nietoetal07}, and \citet{lakshciric09}. Recent applications to the solution of  nonlinear integral equations can be found in  \citet{vanthuan11} and \citet{petrucetal13}. The theory combines elements from contractive applications and monotone  operators, both of which are commonly used in fixed point theory. It also points out some directions for the solution methods developed in Section \ref{sec:solmeth}.

Given that $f(0)=0$ by \ref{asm:prodfp} and assuming $g_i(0)=g_j(0)=0$, it is possible to express the system given in \eqref{eq:FOCM} as
\begin{align*}
g_j(y)&= f(y) - \int_0^y\frac{U^i(g_i(x),g_j(x),x)}{\alpha_i(g_i(x))U^i(g_i(s),g_j(s),s)}\,ds,\\[5pt]
g_i(y)&= f(y) - \int_0^y \frac{U^j(g_j(x),g_i(x),x)}{\alpha_j(g_j(x))U^j(g_j(s),g_i(s),s)}\,ds,
\end{align*}
which has the form of a system of nonlinear Volterra integral equations in $g_i(y)$ and $g_j(y)$, with $(g_i(x),g_j(x),x)$ taken as given.\footnote{A nonlinear Volterra integral equation is a functional equation of the form
\begin{equation*}
u(t)=\phi(t)+\int_a^t\!\!K(t,s,u(s))\,ds, \quad t \in [a,b],
\end{equation*}
where $u(t)$ is the function to be determined.
}
Let $C(K)$ be the space of bounded and continuous functions defined on the closed interval $K$. Now define a couple of operators on the product space, $\mathcal{V}^i: C(K) \times C(K) \to C(K)$ and $\mathcal{V}^j: C(K) \times C(K) \to C(K)$ as
\begin{subequations}
\label{eq:Voltops}
\begin{align}
\label{eq:Voltopi}
\mathcal{V}^i[g_i,g_j](y)&= f(y) - \int_0^y\frac{U^i(g_i(x),g_j(x),x)}{\alpha_i(g_i(x))U^i(g_i(s),g_j(s),s)}\,ds,\\[5pt]
\label{eq:Voltopj}
\mathcal{V}^j[g_j,g_i](y)&= f(y) - \int_0^y\frac{U^j(g_j(x),g_i(x),x)}{\alpha_j(g_j(x))U^j(g_j(s),g_i(s),s)}\,ds.
\end{align}
\end{subequations}
These belong to a class of nonlinear integral operators known as \emph{Volterra operators}. Suppose there exists a metric $d$ such that $(C(K),d)$ is a complete metric space, and endow the product space $C(K) \times C(K)$ with an order relation denoted by $\preceq$. A solution to the system of integral equations $(g_i,g_j) \in C(K) \times C(K)$ is called a \emph{coupled fixed point} of the operators $\mathcal{V}^i$ and $\mathcal{V}^j$ in the product space if $g_i=\mathcal{V}^j\left[g_i,g_j\right]$ and $g_j=\mathcal{V}^i\left[g_j,g_i\right]$.

It is apparent that any fixed point of \eqref{eq:Voltops} will be a function of both \emph{current and future values} of the aggregate state, i.e., $x$ and $y$, something that may restrict the applicability of this method to characterize a solution. In fact, solving a system of nonlinear Volterra integrals often requires numerical methods. We will circumvent these difficulties by extending the solution approach introduced in Section \ref{sec:precommit} for the precommitment game. Specifically, we construct a map relating current consumption and capital holdings with future consumption, keeping the Markovian information structure intact.

\subsection{A Solution Method Based on First-Order Conditions}
\label{sec:solmeth}

The solution method proposed in this section combines  different elements and techniques from the literature that were mentioned earlier, such as implicit programming, the dynamical system approach to concave games, and, particularly, some results from equilibria of non-optimal economies in dynamic general equilibrium. But, as far as we know, this is a novel approach to solve dynamic games with heterogeneous players.

Let $G^i: K \times K \to S^i$ be a function representing current consumption for player $i$ when the current state is $k$ and the state next period is expected to follow a stationary rule $k' = g(k)$, and let $G_l^i$ denote the derivative of $G^i$ with respect to its $l$-th argument, with $l=1,2$, and $j\neq i=1,2$. In this and the following two sections, we characterize these functions and perform a stability analysis of the dynamical system. More precisions on the aggregate savings function $g$ will be given later in Section \ref{sec:stateqM}. 
We begin by transforming \eqref{eq:FOCM} into a system of nonlinear partial differential equations. For this, we replace $g_i(x)$ with $G^i(x,y)$, $g_i(y)$ with $G^i(y,z)$, and $g_i'(y)$ with $G_1^i(y,z)$ for $i=1,2$
\begin{subequations}
\label{eq:focM}
\begin{align}
\label{eq:focMi}
U^i(G^i(x,y),G^j(x,y),x)&=\alpha_i(G^i(x,y))
U^i(G^i(y,z),G^j(y,z),y)[f'(y)-G_1^j(y,z)],\\
\label{eq:focMj}
U^j(G^j(x,y),G^i(x,y),x)&=\alpha_j(G^j(x,y))
U^j(G^j(y,z),G^i(y,z),y)[f'(y)-G_1^i(y,z)].
\end{align}
\end{subequations}
Given that we consider stationary strategies, aggregate states $y \in [0,f(x)]$ and $z \in [0,f(y)]$ satisfy the following conditions for each $x \in K$,
\begin{align}
\label{eq:aggMy}
y&=f(x)-G^i(x,y)-G^j(x,y),\\[5pt]
\label{eq:aggMz}
z&=f(y)-G^i(y,z)-G^j(y,z).
\end{align}
Here we can see how the implicit programming approach is being applied to the current problem. For instance, equation \eqref{eq:aggMy} determines the evolution of the aggregate state from $x$ to $y$, which in turn also depends on an endogenous variable, because $y$ is an argument of the players' Markov strategies. To sum up, \eqref{eq:focM}, \eqref{eq:aggMy}, and \eqref{eq:aggMz}, represent the  optimality conditions when players take the values of $(x,y,z)$ as given.

In order to define a stationary equilibrium for $\mathcal{M}$,   we implicitly define the savings function $g:K \to K$, by replacing $y$ with $g(x)$ and $z$ with $g^2(x):=g(g(x))$ in \eqref{eq:aggMy}--\eqref{eq:aggMz}, which implies
\begin{align}
\label{eq:hdef1}
g(x)&=f(x)-G^i\left(x,g(x)\right)-G^j\left(x,g(x)\right),
\\[5pt]
\label{eq:hdef2}
g^2(x)&=f(g(x))-G^i\left(g(x),g^2(x)\right)-G^j\left(g(x),g^2(x)\right).
\end{align}
Note that if $G^i$ and $G^j$ are continuous on $K \times K$, then \eqref{eq:hdef2} is implied by \eqref{eq:hdef1}. These conditions allow us to define a map
\begin{equation}
\label{eq:Psi}
\Psi\left[k,g(k),g^2(k)\right]=0, \qquad \text{for all } k \in K,
\end{equation}
so that a stationary Markov equilibrium can be characterized in terms of this map as a function $g(k)$ satisfying \eqref{eq:hdef1}. The following result shows that a fixed point of $g$ on $K$ is a stationary point of the dynamical system induced by equilibrium Markov strategies. For a proof, see Appendix \ref{sec:appM}.
\begin{prop}
\label{prp:kstardef}
If $g(k)$ is a continuous function satisfying \eqref{eq:Psi},  then $g$ has a stationary point $k^* \in K$ defined by
\begin{subequations}
\label{eq:kstardef}
\begin{align}
\label{eq:kstar1}
\alpha_i\left(G^i(k^*,k^*)\right)[f'(k^*)-G_1^j(k^*,k^*)]&=1,
\\[5pt]
\label{eq:kstar2}
\alpha_j\left(G^j(k^*,k^*)\right)[f'(k^*)-G_1^i(k^*,k^*)]&=1,
\\[5pt]
\label{eq:kstar3}
f(k^*)-G^i(k^*,k^*)-G^j(k^*,k^*)&=k^*,
\end{align}
\end{subequations}
where $G^i(k,k')$ and $G^j(k,k')$ are solutions to  \eqref{eq:focM} on some open rectangle $I \times I' \subset K \times K$ containing the point $(k^*,k^*)$.
\end{prop}

Characterizing equilibria and their stability properties in this way, as we have seen with open-loop equilibria, involves substantial notation and a significant number of preliminary calculations. For this reason, and to allow comparison with previous results, we use the same definitions and notation introduced in Section \ref{sec:precommit}. However, a word of caution is in order: variables with the same functional form may have different underlying structures, depending on whether   they are defined in terms of an open-loop equilibrium or a Markovian equilibrium. We will attempt to avoid any ambiguity, trying to keep the advantages of having a common notation for both types of equilibria.

In the remaining of this section, we show that for $(k,k')$ sufficiently close to a stationary point $(k^*,k^*)$, Markov optimal strategies $G^i,G^j$ can be restricted to the space of $C^2(I \times I)$ functions that satisfy the following properties:
\begin{enumerate}[label=(M\arabic*)]
\item\label{asm:M1} $G^i(k,k') \geq 0$, $G^j(k,k') \geq 0$, and at least one inequality is strict,
\item\label{asm:M2} $G^i(k,k') + G^j(k,k') < f(k)$,
\item\label{asm:M3} $0 < G_1^i(k,k') < f'(k)$, $0 < G_1^j(k,k') < f'(k)$.
\end{enumerate}
The partial derivatives of $U^i$ have the same functional form as in \eqref{eq:partialU}, but the functions $v_i$ must be replaced by $V_i$. Then, evaluated at a stationary point $k^* \in K$, we have
\begin{align*}
U_1^i=W^i-\frac{\alpha_i'}{\alpha_i}\,U^i < 0, \quad
U_2^i=-\frac{\alpha_i'}{\alpha_i}\,U^i < 0, \quad \text{and} \quad
U_3^i=\frac{\alpha_i'}{\alpha_i}\,U^if' > 0,
\end{align*}
where $U^i:=u_i'+\alpha_i'V_i$ and $W^i:=u_i''+\alpha_i''V_i$.

We begin by differentiating \eqref{eq:focM} with respect to $x$ and evaluating the result at $x=y=z=k^*$. After substitutions, this leads to a system of equations in the unknowns $G_1^i$ and $G_1^j$, whose coefficients are functions of the model's primitives evaluated at a stationary point,\footnote{Recall that $\delta_i = \alpha_i'/\alpha_i$ and $\omega_i = \alpha_i'/\alpha_i - W^i/U^i$.}
\begin{equation*}
\begin{bmatrix}
\delta_i+\omega_i & \delta_i\\[5pt]
\delta_j & \delta_j+\omega_j
\end{bmatrix}
\begin{bmatrix}
G_1^i \\[5pt]
G_1^j
\end{bmatrix}
=
\begin{bmatrix}
\delta_if'\\[5pt]
\delta_jf'
\end{bmatrix}
.
\end{equation*}
Let $\Delta_1$ be the determinant of the coefficient matrix given by $\Delta_1:=\omega_i\omega_j+\delta_i\omega_j+\omega_i\delta_j$. It is clear that $\Delta_1 > 0$, so the solution is
\begin{subequations}
\label{eq:G1}
\begin{align}
\label{eq:G1i}
G_1^i &=\frac{\delta_i\omega_j}{\Delta_1}f'=\frac{\delta_i\omega_j}{\omega_i\omega_j+\delta_i\omega_j+\omega_i\delta_j}\,f',\\[5pt]
\label{eq:G1j}
G_1^j &=\frac{\omega_i\delta_j}{\Delta_1}f'=\frac{\omega_i\delta_j}{\omega_i\omega_j+\delta_i\omega_j+\omega_i\delta_j}\,f'.
\end{align}
\end{subequations}
Since $\delta_i > 0$ and $\omega_i > 0$, it follows that $0 < G_1^i < f'$ and $0 < G_1^j < f'$, as claimed in \ref{asm:M3}. We can think that $G_1^i$ and $G_1^j$ represent (in terms of output) direct effects of the strategic interaction between $i$ and $j$, so that $(f'-G_1^j)$ is the ``net marginal benefit'' for player $i$ of investing in capital, taking into account player $j$'s best response to this increase in $k$. A similar interpretation can be given to $(f'-G_1^i)$ for player $j$. We will refer to $G_1^i$ and $G_1^j$ as \emph{first-order effects}, since they depend exclusively on the primitives of the model.

Next, differentiate \eqref{eq:focM} with respect to $y$, which gives a way to solve $G_2^i$ and $G_2^j$ in terms of $G_1^i$ and $G_1^j$. This requires the introduction of some new parameters defined as
\begin{equation}
\label{eq:nuijdef}
\nu_{ij}:=\frac{\alpha_i'}{\alpha_i}(f'-1)+\frac{f''-G_{11}^j}{f'-G_1^j} \quad \text{and} \quad
\nu_{ji}:=\frac{\alpha_j'}{\alpha_j}(f'-1)+\frac{f''-G_{11}^i}{f'-G_1^i}.
\end{equation}
Note the similarity between $\nu_{ij}$, $\nu_{ji}$ and $\eta_i$, $\eta_j$, respectively, which we used extensively to characterize equilibria in open-loop strategies. This will be explored further in Section~\ref{sec:discussM}. The system of linear equations in the unknowns $G_2^i$ and $G_2^j$ is given by
\begin{equation*}
\begin{bmatrix}
\delta_i+\omega_i & \delta_i\\[5pt]
\delta_j & \delta_j+\omega_j
\end{bmatrix}
\begin{bmatrix}
G_2^i \\[5pt]
G_2^j
\end{bmatrix}
=
\begin{bmatrix}
\omega_i & & \delta_i\\[5pt]
\delta_j & & \omega_j\\
\end{bmatrix}
\begin{bmatrix}
G_1^i \\[5pt]
G_1^j \\
\end{bmatrix}
-
\begin{bmatrix}
\delta_i+\nu_{ij}\\[5pt]
\delta_j+\nu_{ji}\\
\end{bmatrix}
.
\end{equation*}
The partial derivatives $G_2^i$ and $G_2^j$ may be interpreted as the response of each player to a change in expected future returns to capital induced by a change in current returns (brought by a change in the current level of $k$). We call these \emph{second-order effects}, because they are combinations of $G_1^i$, $G_1^j$, $G_{11}^i$, and $G_{11}^j$. Using \eqref{eq:G1} to eliminate $G_1^i$ and $G_1^j$, the solution to this system of equations is given by
\begin{subequations}
\label{eq:G2}
\begin{align}
\label{eq:G2i}
G_2^i &= \frac{[\delta_i\omega_i\omega_j^2+\delta_i\delta_j(\omega_i\omega_j-\delta_i\omega_j+\omega_i\delta_j)]}{\Delta_1^2}\,f'-\frac{[\delta_i\omega_j+(\delta_j+\omega_j)\nu_{ij}-\delta_i\nu_{ji}]}{\Delta_1},\\[5pt]
\label{eq:G2j}
G_2^j &= \frac{[\omega_i^2\delta_j\omega_j+\delta_i\delta_j(\omega_i\omega_j+\delta_i\omega_j-\omega_i\delta_j)]}{\Delta_1^2}\,f'-\frac{[\omega_i\delta_j-\delta_j\nu_{ij}+(\delta_i+\omega_i)\nu_{ji}]}{\Delta_1}.
\end{align}
\end{subequations}

Finally, differentiating \eqref{eq:focM} with respect to $z$ yields
\begin{equation*}
\begin{bmatrix}
G_{12}^i \\ G_{12}^j\\
\end{bmatrix}
=
-
\begin{bmatrix}
f'-G_1^i & f'-G_1^j \\
\end{bmatrix}
\begin{bmatrix}
\delta_j & \omega_j\\
\omega_i & \delta_i\\
\end{bmatrix}
\begin{bmatrix}
G_2^i \\ G_2^j\\
\end{bmatrix}
\end{equation*}
and, after appropriate substitutions,
\begin{align*}
G_{12}^i &=-\frac{(\delta_jG_2^i+\omega_jG_2^j)(\omega_i\omega_j+\delta_i\omega_j)}{\Delta_1}\,f',\\[5pt]
G_{12}^j&=-\frac{(\omega_iG_2^i+\delta_iG_2^j)(\omega_i\omega_j+\omega_i\delta_j)}{\Delta_1}\,f',
\end{align*}
where $G_2^i$ and $G_2^j$ are given by \eqref{eq:G2i} and \eqref{eq:G2j}, respectively. We consider these \emph{third-order effects} that arise from the interaction between current and future returns to capital.

\subsection{Stability Analysis}
\label{sec:stabilityM}

We have seen that optimal strategies are characterized by complex interactions. But it is possible to determine the stability of a stationary point $k^*$ from the analysis of \emph{aggregate} responses, i.e., $(G_1^i+G_1^j)$ and $(G_2^i+G_2^j)$, because some of the interaction terms compensate each other. For this, we obtain a linear approximation of the dynamical system and apply the stable manifold theorem. Substitute \eqref{eq:aggMy} into \eqref{eq:aggMz} to obtain a reduced form for $\Psi$,
\begin{align}
\label{eq:Psireduc}
\Psi(x,y,z)= & f\left(f(x)-G^i(x,y)-G^j(x,y)\right)
-G^i\left(f(x)-G^i(x,y)-G^j(x,y),z\right)\\
&-G^j\left(f(x)-G^i(x,y)-G^j(x,y),z\right)-z=0.\nonumber
\end{align}
Next, differentiate \eqref{eq:Psireduc} with respect to $(x,y,z)$ and evaluate the results at a stationary point. In terms of $f$, $G^i$, and $G^j$, the partial derivatives are given by
\begin{subequations}
\label{eq:Psipar}
\begin{align}
\label{eq:Psi1}
\Psi_1^*&=(f'(k^*)-G_1^i(k^*,k^*)-G_1^j(k^*,k^*))^2,\\[5pt]
\label{eq:Psi2}
\Psi_2^*&=-(f'(k^*)-G_1^i(k^*,k^*)-G_1^j(k^*,k^*))(G_2^i(k^*,k^*)+G_2^j(k^*,k^*)),\\[5pt]
\label{eq:Psi3}
\Psi_3^*&=-(1+G_2^i(k^*,k^*)+G_2^j(k^*,k^*)).
\end{align}
\end{subequations}

Consider the polynomial function
\begin{equation}
\label{eq:poly2}
P(\lambda)=\Psi_1^* + \Psi_2^*\,\lambda + \Psi_3^*\,\lambda^2
\end{equation}
with $\Psi_3^* \neq 0$, so its discriminant is given by
$(\Psi_2^*)^2 - 4\Psi_1^*\Psi_3^* = (f'-G_1^i-G_1^j)^2(G_2^i+G_2^j+2)^2$.
From \eqref{eq:G1}, note that
\begin{equation}
\label{eq:asign}
f'-G_1^i-G_1^j=\left(1+\delta_i/\omega_i+\delta_j/\omega_j\right)^{-1}f'
\end{equation}
does not vanish near $k^*$. Then, it is assumed that $G_2^i+G_2^j \neq 2$ to obtain real and distinct characteristic roots. 

Adding \eqref{eq:G2i} to \eqref{eq:G2j}, it follows that
\begin{equation*}
G_2^i+G_2^j =-\frac{(\delta_i/\omega_i)^2+(\delta_j/\omega_j)^2}{(1+\delta_i/\omega_i+\delta_j/\omega_j)^2}\,f'-\frac{(1/\omega_j)(f''-G_{11}^i)+(1/\omega_i)(f''-G_{11}^j)}{(1+\delta_i/\omega_i+\delta_j/\omega_j)},
\end{equation*}
which, in principle, can take any real value.  Easy calculations show that these roots are
\begin{equation}
\label{eq:roots2}
\lambda_1^*= -(f'-G_1^i-G_1^j)  \quad \text{and} \quad
\lambda_2^*=\frac{f'-G_1^i-G_1^j}{1+G_2^i+G_2^j}.
\end{equation}
Given that the stable manifold theorem is valid for hyperbolic stationary points, we further assume that the eigenvalues $\lambda_1^*$ and $\lambda_2^*$ lie outside the unit circle.\footnote{The analysis of non-hyperbolic fixed points is based on the center manifold. See \citet[Ch. 4]{galor07} and references therein.}

Although we have imposed a number of restrictions on our parameter values, there are still several stability configurations that may arise from the analysis of  \eqref{eq:roots2}. To narrow them down, we use the same approach as in the case of precommitment strategies, evaluating the characteristic polynomial $P$ at three key points, namely, 0, 1 and $-1$, and exploiting some properties of its graph. For this, substitute \eqref{eq:Psipar} into \eqref{eq:poly2}, to obtain
\begin{align}
\label{eq:Ppoints}
P(0)&=(f'-G_1^i-G_1^j)^2,\nonumber\\[5pt]
P(1)&=\left[(f'-G_1^i-G_1^j)+1\right]\left[(f'-G_1^i-G_1^j)-(1+G_2^i+G_2^j)\right],\\[5pt]
P(-1)&=\left[(f'-G_1^i-G_1^j)-1\right]\left[(f'-G_1^i-G_1^j)+(1+G_2^i+G_2^j)\right]\nonumber.
\end{align}
Obviously, $P(0)>0$, but the signs of $P(-1)$ and $P(1)$  ultimately depend on parameter values. Introducing some notation (for this analysis only) will be helpful to visualize all  different cases and clarify the presentation of the results. Let $a:=(f'-G_1^i-G_1^j)$ and $b:=G_2^i+G_2^j$. From \eqref{eq:asign}, it follows immediately that $a > 0$. But, as we mentioned earlier, $b$ can take any value (with a couple of   exceptions pointed out above). From \eqref{eq:Ppoints},
\begin{equation*}
P(-1)=(a-1)[a+(1+b)], \quad P(0)=a^2, \quad \text{and} \quad P(1)=(a+1)[a-(1+b)].
\end{equation*}
All the graphs corresponding to these cases are parabolas, so additional information can be obtained from the principal coefficient of the characteristic polynomial, $-(1+b)$, and the $x$-coordinate of their vertices, $\lambda_v=-ab/2(1+b)$.  The stability analysis based on these formulas is summarized in Table \ref{tab:stable}.

\begin{table}[p]
\caption{Stability Analysis}
\label{tab:stable}
\centering
\begin{tabular}{|c|p{2.5cm}|c|c|c|c|}
\hline  {\textbf{Case}} & {\textbf{Parameters}} &
{$\bds{P(\lambda)}$} & {\textbf{Eigenvalues}} & {\textbf{Notes}} & {\textbf{Stab.}} \\
\hline
1 & \parbox[t][1.75cm]{2.5cm}{$0 < a < 1$, \\ $b > 0$}  & \raisebox{-.75\height} {\includegraphics[scale=0.15]{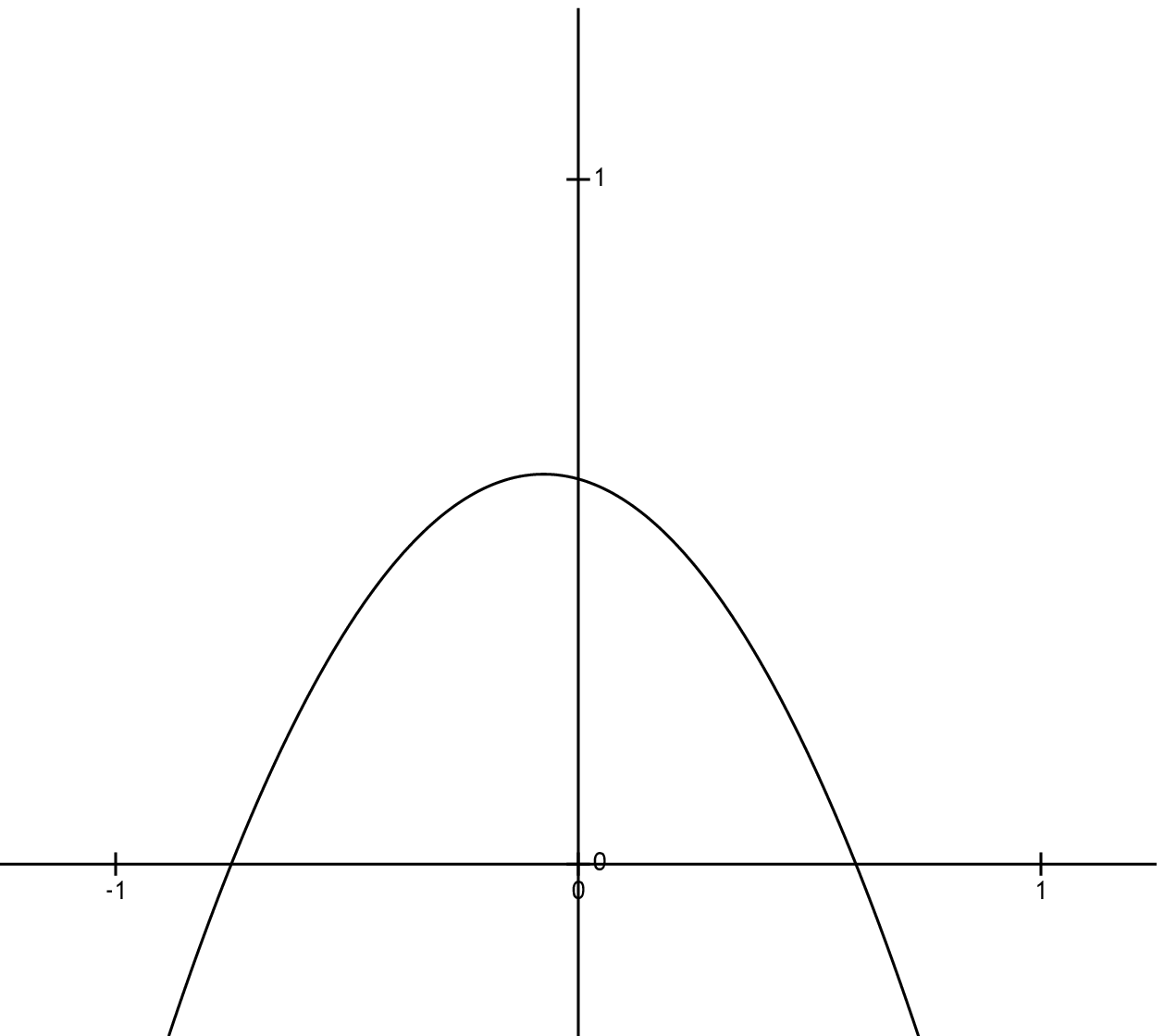}} & $-1 < \lambda_1^* < 0 < \lambda_2^* < 1$ & $|\lambda_1^*| > |\lambda_2^*|$ & stable\\
\hline
2 & \parbox[t][1.75cm]{2.5cm}{$a > 1$, $b > 0$, \\ $a-b < 1$}  & \raisebox{-.75\height} {\includegraphics[scale=0.15]{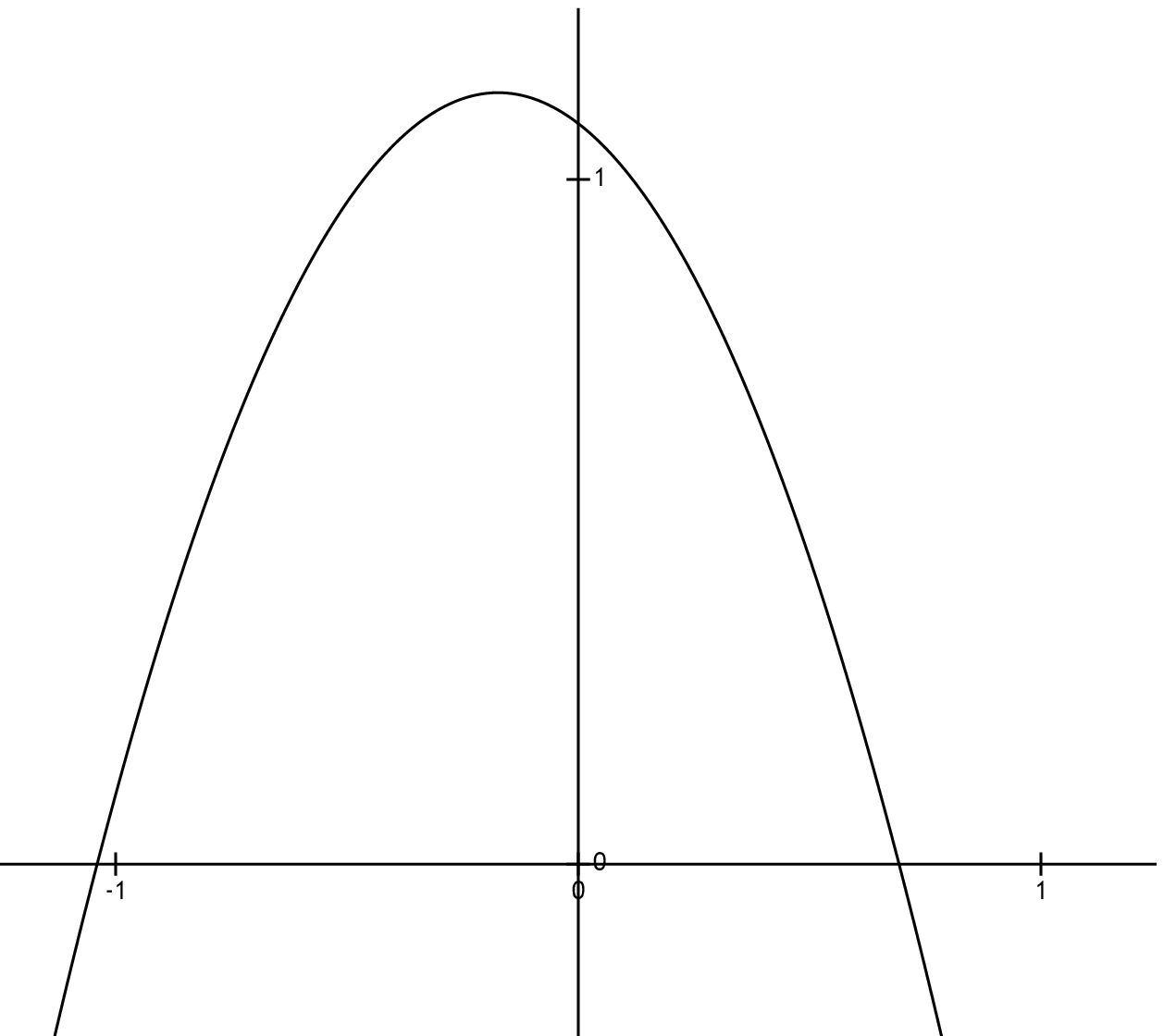}} & $\lambda_1^* < -1; \quad 0 < \lambda_2^* < 1$ & $|\lambda_1^*| > |\lambda_2^*|$ & saddle\\
\hline
3 & \parbox[t][1.75cm]{2.5cm}{$a > 1$, $b > 0$, \\ $a-b > 1$}  & \raisebox{-.75\height} {\includegraphics[scale=0.15]{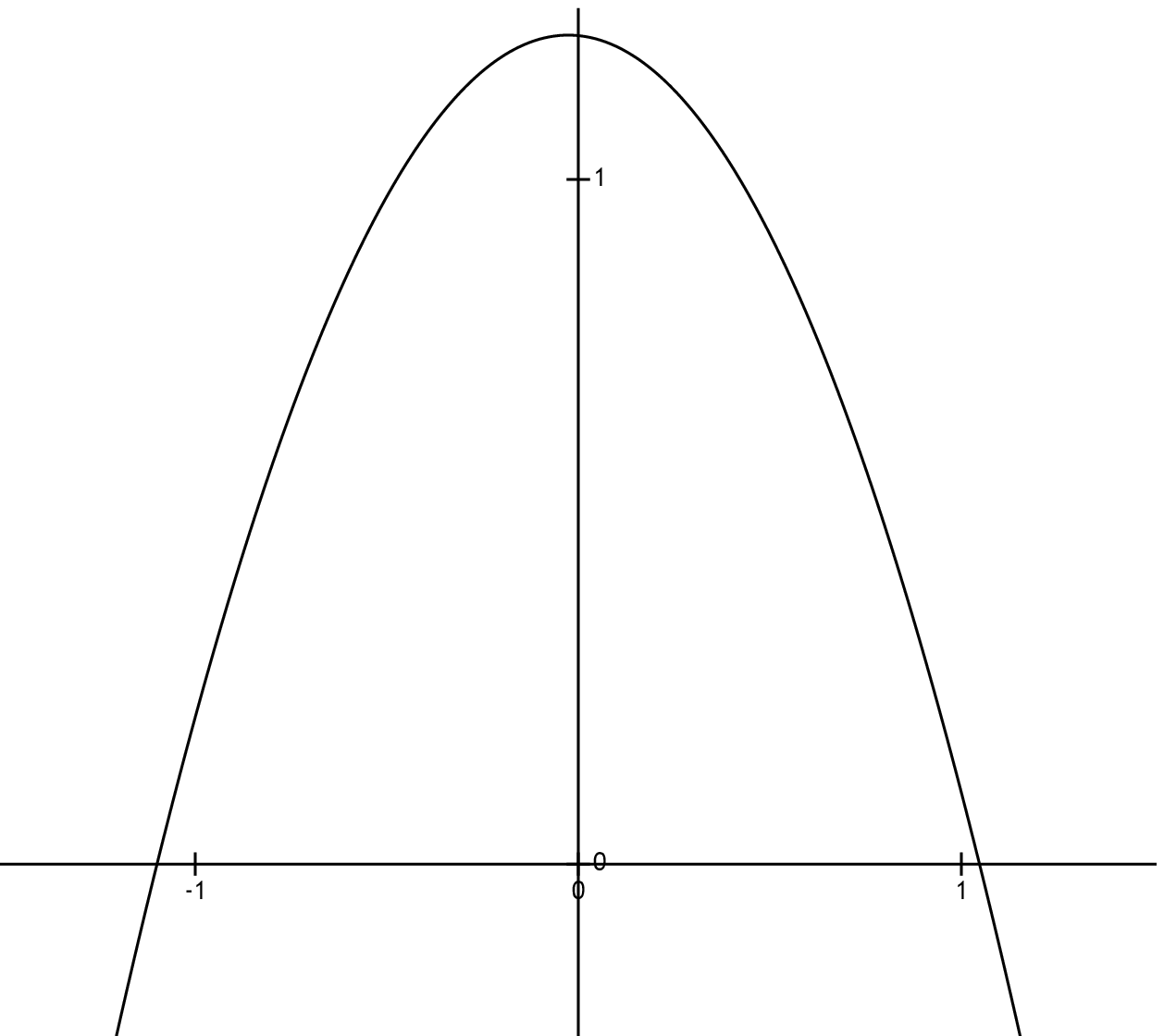}} & $\lambda_1^* < -1; \quad \lambda_2^* > 1$ & $|\lambda_1^*| > |\lambda_2^*|$ & unstable\\
\hline
4 & \parbox[t][1.75cm]{2.5cm}{$0 < a < 1$, \\ $-1 < b < 0$, \\ $a-b < 1$}  & \raisebox{-.75\height} {\includegraphics[scale=0.15]{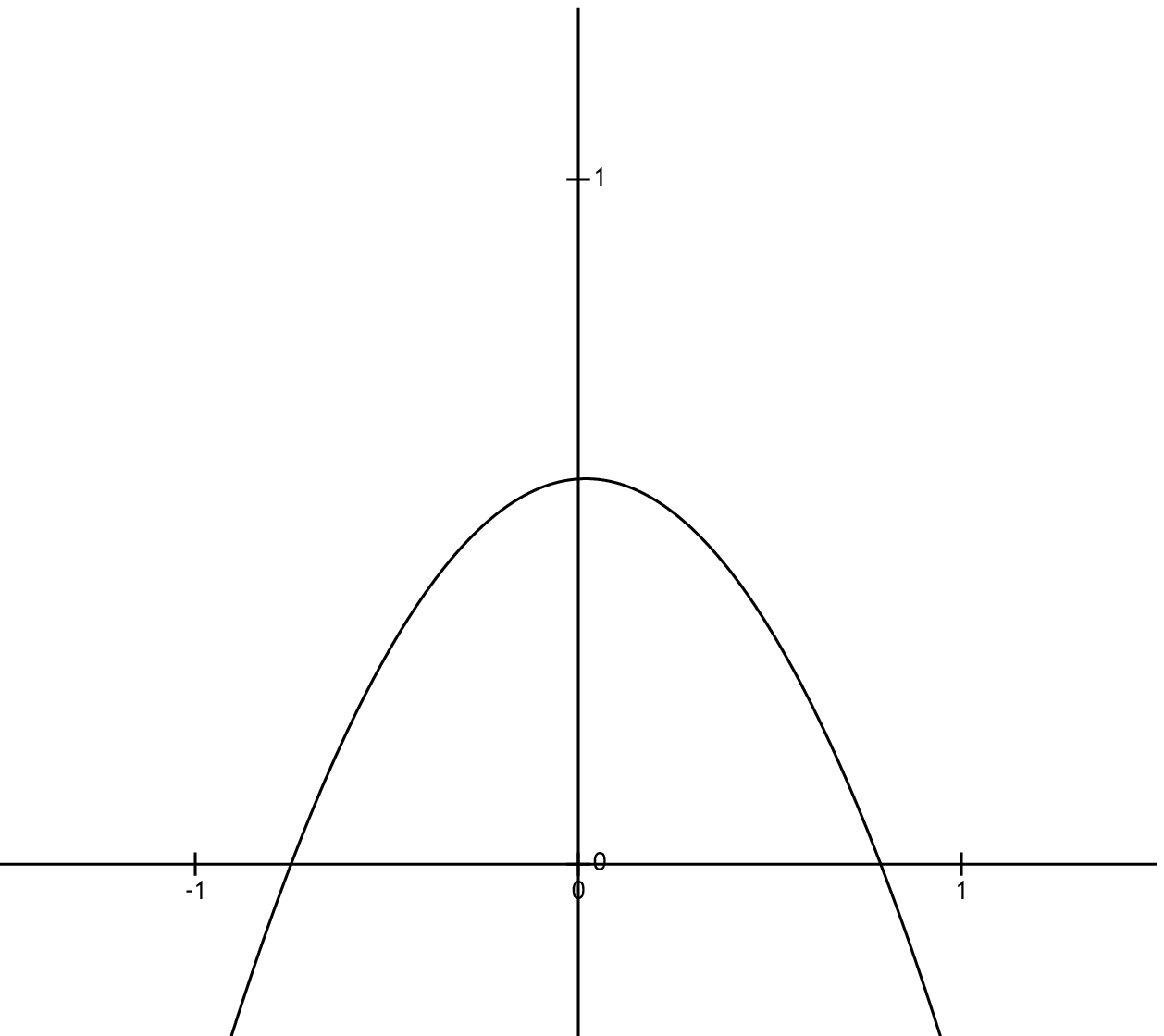}} & $-1 < \lambda_1^* < 0 < \lambda_2^* < 1$ & $|\lambda_1^*| < |\lambda_2^*|$  & stable \\
\hline
5 & \parbox[t][1.75cm]{2.5cm}{$0 < a < 1$, \\ $ -1 < b < 0$, \\ $a - b > 1$}  & \raisebox{-.75\height} {\includegraphics[scale=0.15]{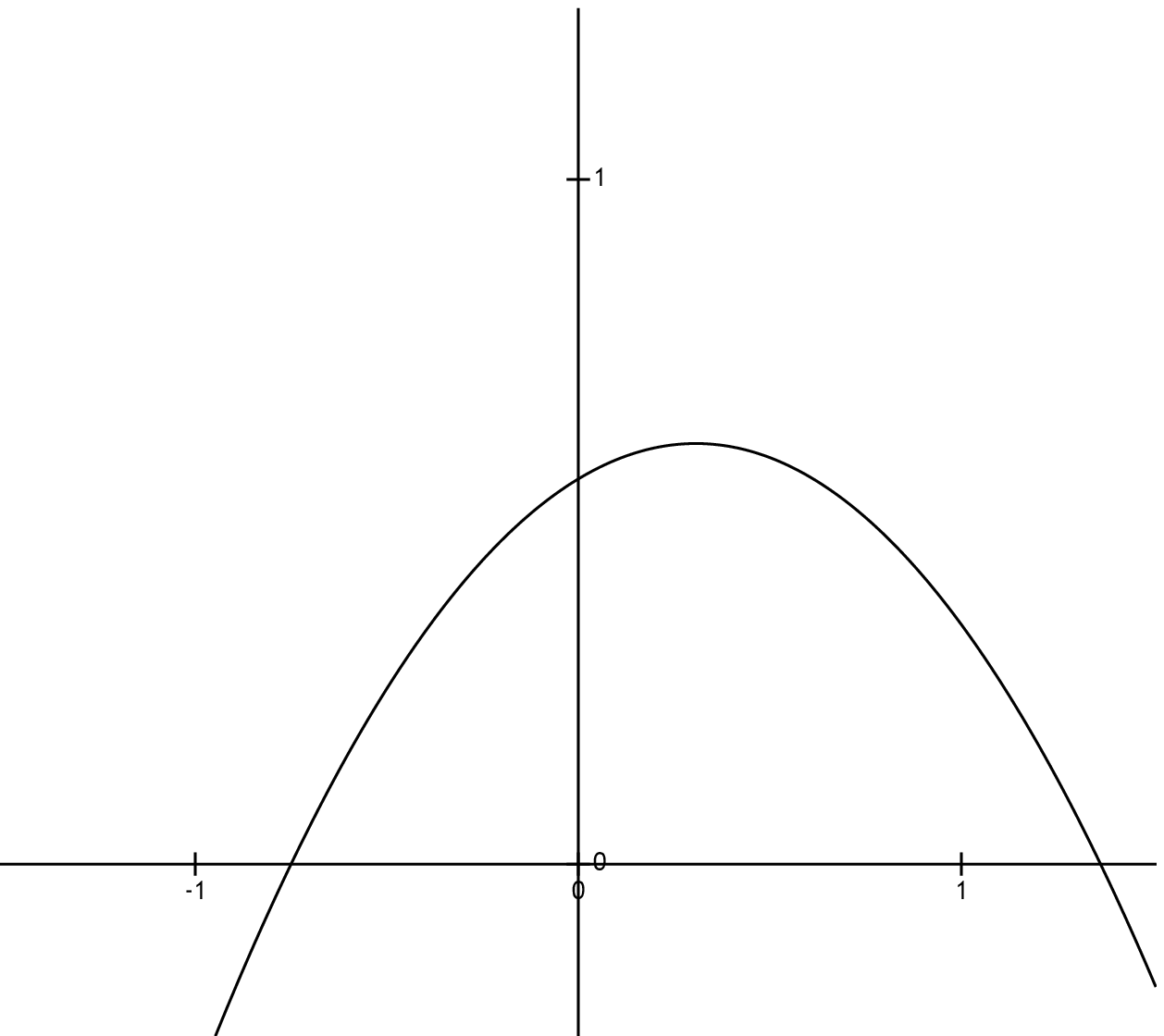}} & $-1 < \lambda_1^* < 0; \quad \lambda_2^* > 1$ & $|\lambda_1^*| < |\lambda_2^*|$  & saddle\\
\hline
6 & \parbox[t][1.75cm]{2.5cm}{$a > 1$, \\ $-1 < b < 0$}  & \raisebox{-.75\height} {\includegraphics[scale=0.15]{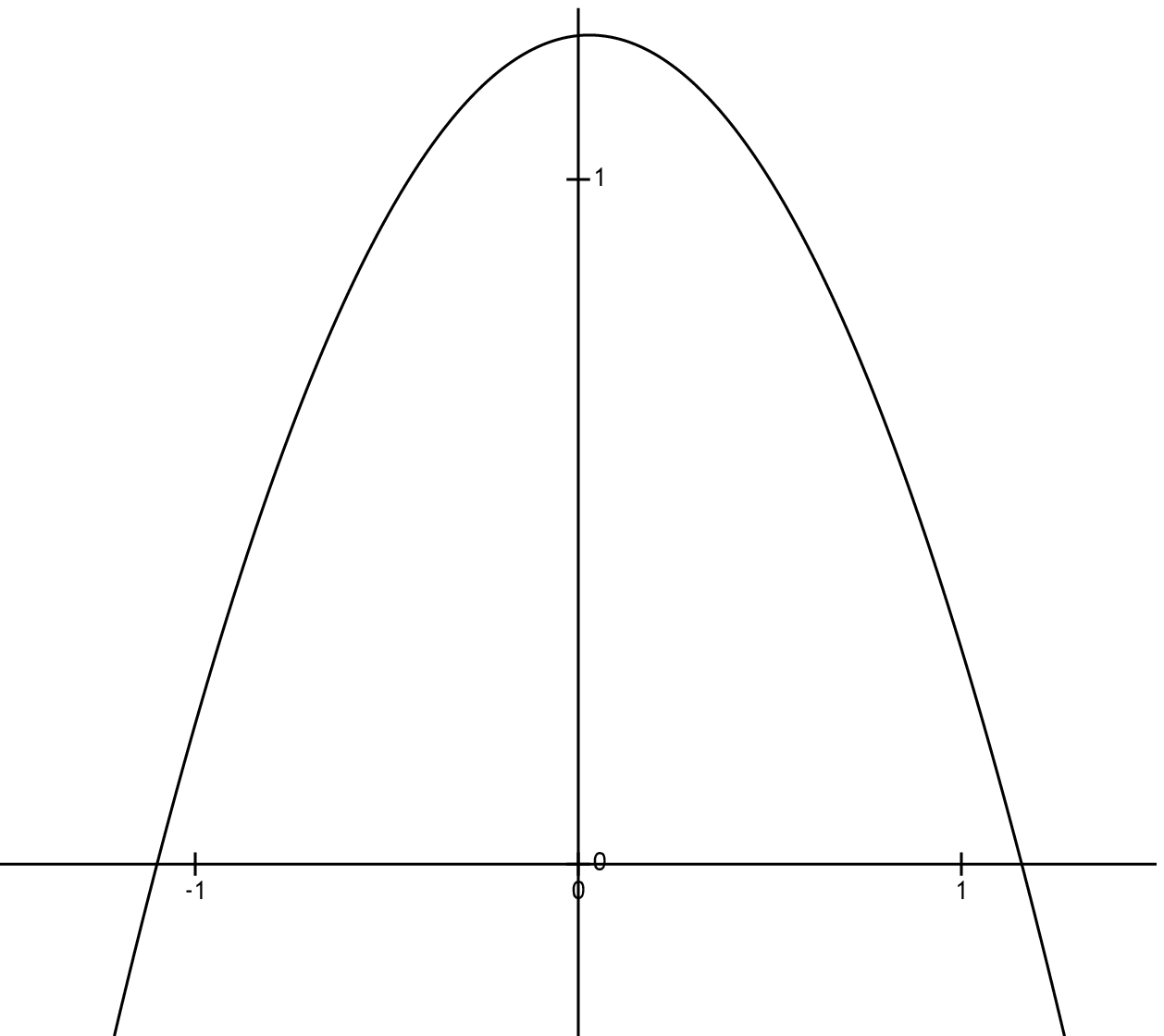}} & $\lambda_1^* < -1; \quad \lambda_2^* > 1$ & $|\lambda_1^*| < |\lambda_2^*|$  & unstable \\
\hline
7 & \parbox[t][1.75cm]{2.5cm}{$0 < a < 1$,\\ $ b < -1$,\\ $a+b < -1$} & \raisebox{-.75\height} {\includegraphics[scale=0.15]{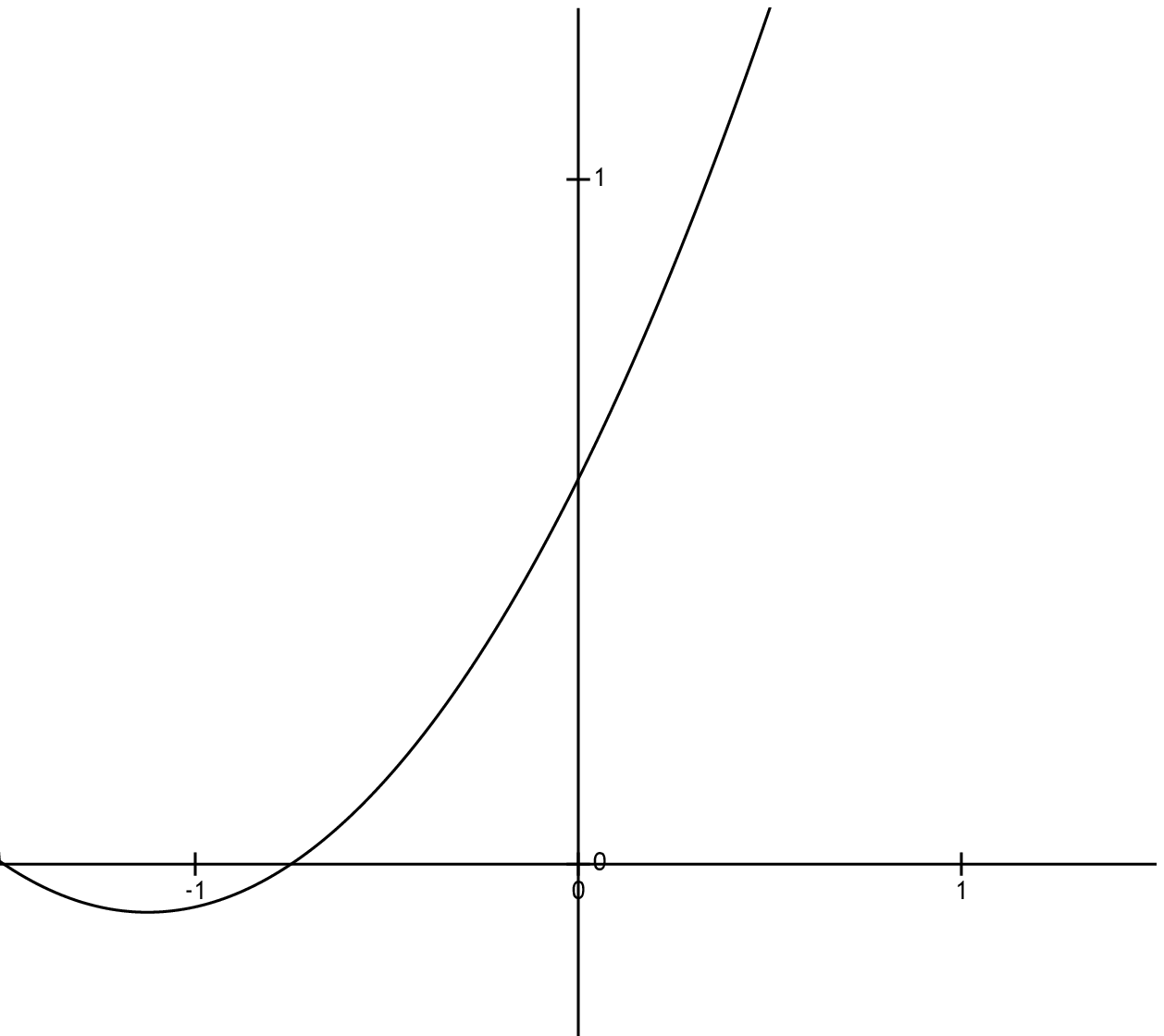}} & $-1 < \lambda_1^* < \lambda_2^* < 0$ & $|\lambda_1^*| > |\lambda_2^*|$  & stable \\
\hline
8 & \parbox[t][1.75cm]{2.5cm}{$0 < a < 1$,\\ $b < -1$,\\ $a + b > -1$} & \raisebox{-.75\height} {\includegraphics[scale=0.15]{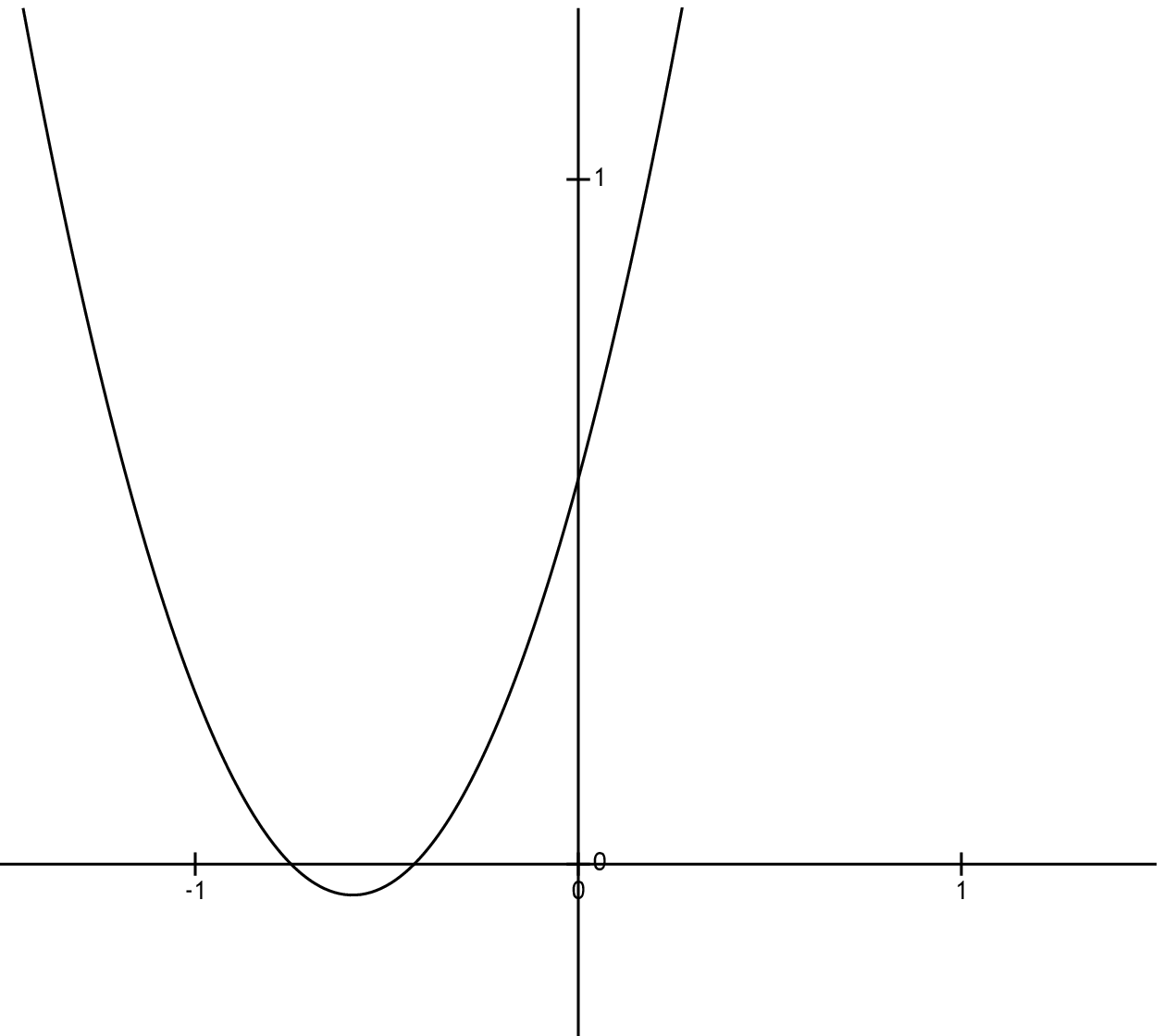}} & $\lambda_1^* < -1 < \lambda_2^* < 0 ;$ &  $|\lambda_1^*| > |\lambda_2^*|$ & saddle \\
\hline
9 & \parbox[t][1.75cm]{2.5cm}{$a > 1$,\\ $b < -1$ \\ $a + b > -1$} & \raisebox{-.75\height} {\includegraphics[scale=0.15]{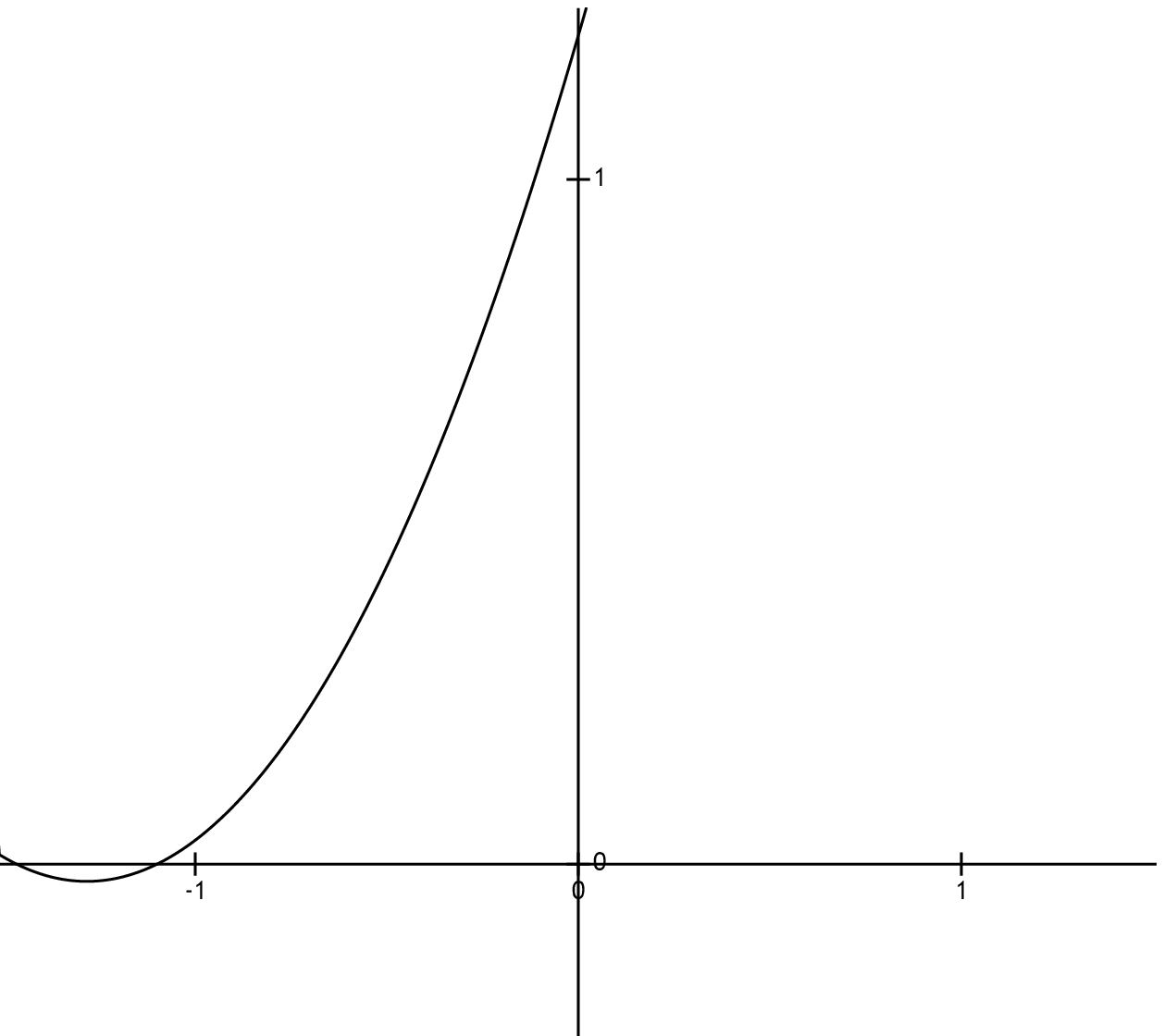}} & $\lambda_1^* < \lambda_2^* < -1$ & $|\lambda_1^*| > |\lambda_2^*|$  & unstable
\\
\hline
\end{tabular}
\end{table}

Note that oscillatory, i.e., nonmonotonic, behavior is pervasive, given that at least one eigenvalue is negative in all cases. This is in stark contrast with the majority of  models of optimal growth with many agents in the literature, with and without strategic interaction, where monotone trajectories of the state variable is one of their strongest predictions. In particular, cases 1, 4, and 7 are globally stable, in the sense that it is possible to define an interval $I \subset (0,k_m)$ and a continuous map $\psi$ such that for any $(k_0,k_1)$ that belong to the graph of $\psi$, the path of $k_t$ converges to the stationary point $k^*$. In other words, the stable manifold has $\dim =2$. Cases 1 and 7 imply oscillatory paths that end up converging, since the eigenvalue with largest modulus is negative, i.e., $|\lambda_1^*| > |\lambda_2^*|$. The reverse inequality holds in case 4, so the stable path may display oscillations during transitional dynamics, but convergence is eventually monotonic. In cases 2, 5, and 8, the stable manifold has $\dim=1$, so there is a continuous function, say $\psi_1$, such that for $k_0$ sufficiently close to $k^*$, the only points $(k_0,k_1)$ for which there is convergence are those on the graph of $\psi_1$. Given that the stable roots for cases 5 and 8 are negative, convergence is oscillatory; in case 2 there is monotone convergence. All remaining cases imply divergent paths, also with oscillations.

The stable manifold theorem provides an additional result to study the behavior of the aggregate saving function. There is a neighborhood $\mathcal{N}$ of the point $(k^*,k^*)$ and a continuously differentiable function $\psi: \mathcal{N} \to \mathbb{R}$ such that for $k_0$ sufficiently close to $k^*$, there exists $k_1$ with $(k_1,k_0) \in \mathcal{N}$ and $\psi(k_1,k_0)=0$. This is true if the Jacobian matrix $[D\psi(k^*,k^*)]$ has full rank. The linearized system $\Psi$ can be represented in terms of the coefficients  of the characteristic polynomial 
$P(\lambda)$ as
\begin{equation*}
\Psi_1^*(k_t-k^*) + \Psi_2^*(k_{t+1}-k^*) + \Psi_3^*(k_{t+2}-k^*)=0,
\end{equation*}
so that the behavior of $k_t$ near $k^*$ can be characterized by a $2 \times 2$ matrix
\begin{equation*}
A:=
\begin{bmatrix}
-\frac{\Psi_2^*}{\Psi_3^*}\hfill & -\frac{\Psi_1^*}{\Psi_3^*}\hfill\\[5pt]
\ 1 &\ 0\\
\end{bmatrix}
,
\end{equation*}
which is nonsingular. 

By Jordan decomposition, $A$ can be written as $A = B^{-1}\Lambda B$, where $B$ is a nonsingular matrix, and $\Lambda$ a diagonal matrix with the eigenvalues $\lambda_1^*$ and $\lambda_2^*$ in the main diagonal. Using the fact that $BA=\Lambda B$, the stable manifold can be characterized from the following system
\begin{equation*}
\begin{bmatrix}
b_{11} & b_{12}\\[5pt]
b_{21} & b_{22}\\
\end{bmatrix}
\begin{bmatrix}
-\frac{\Psi_2^*}{\Psi_3^*}\hfill & -\frac{\Psi_1^*}{\Psi_3^*}\hfill\\[5pt]
\ 1 &\ 0\\
\end{bmatrix}
=
\begin{bmatrix}
\lambda_1^* & 0\\[5pt]
0 & \lambda_2^*\\
\end{bmatrix}
\begin{bmatrix}
b_{11} & b_{12}\\[5pt]
b_{21} & b_{22}\\
\end{bmatrix}
,
\end{equation*}
and for any nonzero values of $b_{11}$ and $b_{22}$, we have that
\begin{equation*}
B=
\begin{bmatrix}
b_{11} &\left(\lambda_1^*+\frac{\Psi_2^*}{\Psi_3^*}\right)b_{11}\\[5pt]
\left(\lambda_2^*+\frac{\Psi_2^*}{\Psi_3^*}\right)^{-1}b_{22} & b_{22}\\
\end{bmatrix}
.
\end{equation*}

Note that the savings function must satisfy $\psi[g(k),k]=0$, hence its derivative is given by $g'(k^*)=-\psi^{-1}_1(k^*,k^*)\,\psi_2(k^*,k^*)$. The derivatives of the stable manifold at the stationary point are 
\begin{equation*}
\psi_1(k^*,k^*)=\left(\lambda_2^*+\tfrac{\Psi_2^*}{\Psi_3^*}\right)^{-1} b_{22}, \quad  \text{and} \quad \psi_2(k^*,k^*)=b_{22},
\end{equation*}
which implies 
\begin{equation*}
g'(k^*)=-\frac{\psi_2(k^*,k^*)}{\psi_1(k^*,k^*)}= \lambda_1^*.
\end{equation*}
If the system is governed by $\lambda_2^*$ instead, then simply replace $\lambda_1^*$ with $\lambda_2^*$ in all previous calculations.

\subsection{Stationary equilibrium}
\label{sec:stateqM}

In this section we establish the existence of a stationary equilibrium for the Markov game in terms of the map $\Psi$  defined in \eqref{eq:Psi}. For this, we focus on the first-order conditions for the players dynamic programs and incorporate the results obtained in the previous two sections.  The analysis is mainly based on \citet[Ch. 18.2]{stokeylucas89} for a dynamic nonoptimal economy with homogeneous agents, but there are enough differences to justify a separate treatment.

We need some preliminaries before presenting the main result. Assume that $G_2^i+G_2^j \neq 0$, so $\Psi_2^*$ does not vanish. By the implicit function theorem, there exists a rectangle $I \times I$ containing the point $(k^*,k^*)$ and a continuously differentiable map $H: I \times I \to I$ such that
\begin{equation}
\label{eq:Hdef}
\Psi(x,H(x,z),z)=0, \qquad \text{for all } (x,z) \in I \times I.
\end{equation}
Many important results depend on the characterization of this  mapping. The following monotonicity properties are introduced for that purpose.
\begin{defn}
\label{def:Hmonot}
A map $H:I \times I \to \mathbb{R}$ is said to be
\begin{enumerate}[label=(\alph*),leftmargin=*]
\item \emph{uniformly monotone} if $H(x,z)$ is increasing (resp.  decreasing) in $x$ for all $z \in I$, and increasing (resp. decreasing) in $z$ for all $x \in I$
\item \emph{mixed monotone} if $H(x,z)$ is increasing (resp.  decreasing) in $x$ for all $z \in I$ and decreasing (resp. increasing) in $z$, that is, for all $x \in I$.
\end{enumerate}
\end{defn}

All the stationary equilibria from Table \ref{tab:stable} can be grouped in terms of these properties.
From \eqref{eq:Hdef}, the partial derivatives of $H(x,z)$ evaluated at $(k^*,k^*)$, are $H_1^*=-\Psi_1^*/\Psi_2^*$ and $H_2^*=-\Psi_3^*/\Psi_2^*$. At the same time, it is well known that the eigenvalues $\lambda_1^*$ and $\lambda_2^*$ and the coefficients of $P$ are related by $\lambda_1^* + \lambda_2^* = -\Psi_2^*/\Psi_3^*$ and $\lambda_1^*\,\lambda_2^*=\Psi_1^*/\Psi_3^*$, therefore
\begin{align*}
H_1^*=\frac{\lambda_1^*\lambda_2^*}{\lambda_1^*+\lambda_2^*} \quad \text{and} \quad H_2^*=\frac{1}{\lambda_1^*+\lambda_2^*}.
\end{align*}
Results are shown in Table \ref{tab:Hmonot}.

\begin{table}[h!]
\caption{Properties of $H(x,z)$}
\label{tab:Hmonot}
\centering
\begin{tabular}{|c|c|c|c|c|}
\hline  {\textbf{Group}} & {\textbf{Cases}} & {\textbf{Partial Derivatives}} &  {\textbf{Monotonicity}} \\
\hline
I   & 1 -- 3 & $H_1^* > 0, \quad H_2^* < 0$ & Mixed \\
\hline
II  & 4 -- 6 & $H_1^* < 0, \quad H_2^* > 0$ & Mixed \\
\hline
III & 7 -- 9 & $H_1^* < 0, \quad H_2^* < 0$ & Uniform \\
\hline
\end{tabular}
\end{table}

Denote by $C(I)$ the space of bounded continuous functions $h:I \to \mathbb{R}$ with the sup norm, and define the operator $T$ on $C(I)$ by
\begin{equation}
\label{eq:Tdef}
Th(k)=H[k,h^2(k)], \qquad \text{for all } k \in I,
\end{equation}
where $H$ is the function defined in \eqref{eq:Hdef}. Clearly, a fixed point of $T$ is a solution to \eqref{eq:Psi} on $I$. We eliminate the unstable cases (3, 6, and 9) from Table \ref{tab:Hmonot}, because there has to be at least one eigenvalue smaller than one in absolute value to construct the relevant space of functions. The proof for $H$ uniformly  monotone is similar to the one in \citet{stokeylucas89}, so we omit it to concentrate on the cases where $H$ is mixed monotone.

Let $\varepsilon > 0$ and define $I_\varepsilon:=[k^*-\varepsilon,k^*+\varepsilon]$. Then, for any $\lambda \in (-1,1)$,  denote by $D_{|\lambda|}(I_\varepsilon)$ the space of continuous functions on $I_\varepsilon$ that have a stationary point $k^*$ and satisfy a Lipschitz condition with constant $|\lambda|$. More specifically, this set is given by
\begin{equation*}
D_{|\lambda|}(I_\varepsilon):=\big\{h \in C(I_\varepsilon): h(k^*)=k^*\ \text{and}\ |h(k)-h(k')| \leq |\lambda| |k - k'|,\ \text{all}\ k,k' \in I_\varepsilon\big\}.
\end{equation*}
We give full treatment to the case $-1 < \lambda_1^* < 0$ and $\lambda_2^* > 0$ (Group II in Table \ref{tab:Hmonot}), since the remaining cases can be handled in an analogous manner. Given that the stable manifold with $-1 < \lambda_1^* < 0$ implies that the policy function is strictly decreasing, we choose the subspace of $D_{|\lambda|}(I_\varepsilon)$ of decreasing functions,
\[
\upwt{D}_{|\lambda|}(I_\varepsilon):= \big\{h \in D_{|\lambda|}(I_\varepsilon) : \text{for all}\ k,k'\in I_\varepsilon,\ \text{if}\ k \leq k',\ \text{then}\ h(k') \leq h(k)\big\}.
\]
The space $\upwt{D}_{|\lambda|}(I_\varepsilon)$ endowed with the sup norm is a Banach space in the metric derived from this norm. The next result shows that for every $\lambda$ that lies in an interval specified below, the operator $T$ maps $\upwt{D}_{|\lambda|}(I_\varepsilon)$ to itself. See Appendix \ref{sec:appM} for a detailed proof.

\begin{lemm}
\label{lem:m1m2def}
There exists $-1 \leq \unwt{\lambda} < \lambda_1^*$ such that for every $\lambda \in (\unwt{\lambda},\lambda_1^*)$, there are constants $m_1, m_2$ with $ H_1^* < m_1 < 0$ and $0 < H_2^* < m_2$, that satisfy
\begin{equation}
\label{eq:m1m2def}
\lambda < H_1^* + H_2^*\lambda^2 < m_1 + m_2\lambda^2 < 0.
\end{equation}
\end{lemm}

Let $\unwt{\lambda}$, $m_1$, and $m_2$ be defined as in Lemma \ref{lem:m1m2def} and let $\lambda \in (\unwt{\lambda},\lambda_1^*)$. For any $\varepsilon > 0$, the continuity of $H$ implies that $Th$ is continuous for all $h \in \upwt{D}_{|\lambda|}(I_\varepsilon)$. By definition of a stationary point $\Psi(k^*,k^*,k^*)=0$, then $H(k^*,k^*)=k^*$. It follows that $Th(k^*)=k^*$ for any $h \in \upwt{D}_{|\lambda|}(I_\varepsilon)$. Now, since $H_1$ and $H_2$ are continuous on $I$, there exists some $\varepsilon > 0$ such that $I_\varepsilon \subset I$ and $H_1(x,z) \leq m_1 < 0$ and $
0 < H_2(x,z) \leq m_2$, for all $(x,z) \in I_\varepsilon \times I_\varepsilon$. Let $k,k' \in I_{\varepsilon'}$ and, without loss of generality, suppose that $k \geq k'$. The fact that $h$ is decreasing implies that $h(k) \leq h(k')$ and $h^2(k) \geq h^2(k')$. Then,
\begin{align*}
Th(k)-Th(k') & = \left[H(k,h^2(k))-H(k',h^2(k))\right]
+\left[H(k',h^2(k))-H(k',h^2(k'))\right],\\[5pt]
& \leq m_1(k-k')+m_2\left[h^2(k)-h^2(k')\right],\\[5pt]
& \leq m_1(k-k')+m_2\lambda\left[h(k)-h(k')\right],\\[5pt]
& \leq (m_1+m_2\lambda^2)(k-k'),
\end{align*}
where the third and fourth lines use the fact that $h \in \upwt{D}_{|\lambda|}(I_\varepsilon)$. By \eqref{eq:m1m2def}, we have that $(m_1+m_2\lambda^2) < 0$, so $Th(k) \leq Th(k')$. Hence $Th$ is decreasing.

It remains to prove that $Th$ satisfies the Lipschitz condition with constant $|\lambda|$. By the continuity of $Th$, it follows from Lemma \ref{lem:m1m2def} and the result from  the previous paragraph, that for $k,k' \in I_\varepsilon$ and $\lambda \in (\unwt{\lambda},\lambda_1^*)$, we have
\begin{align*}
\lambda \leq \frac{Th(k)-Th(k')}{k-k'} \leq m_1+m_2\lambda^2 < 0.
\end{align*}
This immediately implies that $|Th(k)-Th(k')| \leq |\lambda||k-k'|$. Hence, the operator $T$ maps the space $\upwt{D}_{|\lambda|}(I_\varepsilon)$ to itself. Note that $I_\varepsilon$ is compact and it is easy to verify that $\upwt{D}_{|\lambda|}(I_\varepsilon) \subset C(I_\varepsilon)$ is closed, bounded and convex, and that $\upwt{D}_{|\lambda|}(I_\varepsilon)$ is itself an equicontinuous family. Moreover, $T$ is continuous, but we leave the proof and several technical details to Appendix \ref{sec:appM}. All this means that $T$ satisfies the hypothesis of Schauder's fixed point theorem, so we have proved the main result of this section, which we state as a theorem.

\begin{thrm}
\label{thm:fixedpT}
There exist $\varepsilon > 0$ and $|\lambda|< 1$ such that the operator $T:\upwt{D}_{|\lambda|}(I_\varepsilon) \to \upwt{D}_{|\lambda|}(I_\varepsilon)$ defined in \eqref{eq:Tdef} has a fixed point.
\end{thrm}

\subsection{Discussion}
\label{sec:discussM}

We close this section with a stability analysis and a graphical analysis for the equilibrium in Markov strategies, along the lines of those made in Section \ref{sec:discussOL} for the precommitment equilibrium. First, differentiate the left-hand sides of \eqref{eq:kstar1} and \eqref{eq:kstar2} and evaluate at $k=k^*$, which yields
\begin{subequations}
\label{eq:nucondM}
\begin{align}
\label{eq:nucondMi}
\frac{\alpha_i'}{\alpha_i}(G_1^i+G_2^i)+\frac{f''-G_{11}^j-G_{12}^j}{f'-G_1^j} &< 0\\[5pt]
\label{eq:nucondMj}
\frac{\alpha_j'}{\alpha_j}(G_1^j+G_2^j)+\frac{f''-G_{11}^i-G_{12}^i}{f'-G_1^i} &< 0,
\end{align}
\end{subequations}
and then verify that for the aggregate savings function
$s(k^*):=f(k^*)-G^i(k^*,k^*)-G^j(k^*,k^*)$,
the following local stability condition $|s'(k^*)| < 1$ holds, i.e.,
\begin{equation*}
s'(k^*):=f'(k^*)-(G_1^i(k^*,k^*)+G_1^j(k^*,k^*))-(G_2^i(k^*,k^*)+G_2^j(k^*,k^*)).
\end{equation*}
To do this, use \eqref{eq:nucondM} to eliminate $(G_1^i+G_1^j)$ and $(G_2^i+G_2^j)$ from $s'(k^*)$ and rearrange terms, so the stability condition is equivalent to
\begin{align}
\label{eq:locstabM}
\left(\frac{\alpha_i'}{\alpha_i}\right)\left(\frac{\alpha_j'}{\alpha_j}\right)\left(f'-1\right)+\left(\frac{\alpha_i'}{\alpha_i}\right)\frac{f''-G_{11}^i-G_{12}^i}{f'-G_1^i}+\left(\frac{\alpha_j'}{\alpha_j}\right)\frac{f''-G_{11}^j-G_{12}^j}{f'-G_1^j} < 0,
\end{align}
if $s(k^*)$ is increasing, and with the reverse inequality if $s(k^*)$ is decreasing. Two remarks are in order: first, keeping in mind that the arguments of these functions are  best-response profiles for two different equilibria,  comparing \eqref{eq:locstabM} with the stability condition for the precommitment equilibrium \eqref{eq:locstabOL} is useful to get some understanding about the additional strategic interactions brought by Markov strategies; second, this stability condition has a close relation with the ``stability parameters'' $\nu_{ij}$ and $\nu_{ji}$ defined in \eqref{eq:nuijdef}.

For the graphical approach, by \eqref{eq:G1}, we have
$(f'-G_1^j)=\xi_{ij}\,f'$ and $(f'-G_1^i)=\xi_{ji}\,f'$,
where
\begin{equation*}
\xi_{ij}:=\frac{1+\delta_i/\omega_i}{1+\delta_i/\omega_i+\delta_j/\omega_j} \quad \text{and} \quad
\xi_{ji}:=\frac{1+\delta_j/\omega_j}{1+\delta_i/\omega_i+\delta_j/\omega_j}.
\end{equation*}
Thus, a stationary Markov equilibrium can be written as
\begin{equation}
\label{eq:SPMdisc}
\alpha_i\left(G^i(k^*,k^*)\right)\xi_{ji}(k^*,k^*)\,f'(k^*)=\alpha_j\left(G^j(k^*,k^*)\right)\xi_{ij}(k^*,k^*)\,f'(k^*)=1.
\end{equation}
Provided that these functions are sufficiently smooth so that $\alpha_i$ and $\alpha_j$ are invertible over the relevant range, it is possible to determine the optimal stationary strategies for $k \in I_\varepsilon$,
\begin{equation*}
G^i(k,k)=\alpha_i^{-1}\left(\frac{1}{\xi_{ji}(k,k)\,f'(k)}\right) \quad \text{and} \quad G^j(k,k)=\alpha_j^{-1}\left(\frac{1}{\xi_{ij}(k,k)\,f'(k)}\right).
\end{equation*}
Hence the Markov equilibrium results from the intersection of the aggregate consumption curve, $G(k,k):= G^i(k,k) + G^j(k,k)$, and the aggregate net output curve, $f(k)-k$. This is pictured in Figure~\ref{fig:stateqME} as point $E'$, which determines $k^*$. Points $C'$ and $D'$ represent the equilibrium consumption levels for players $i$ and $j$, respectively, associated with $k^*$.

\begin{figure}[h!]
\centering
\includegraphics[width=0.75\textwidth]{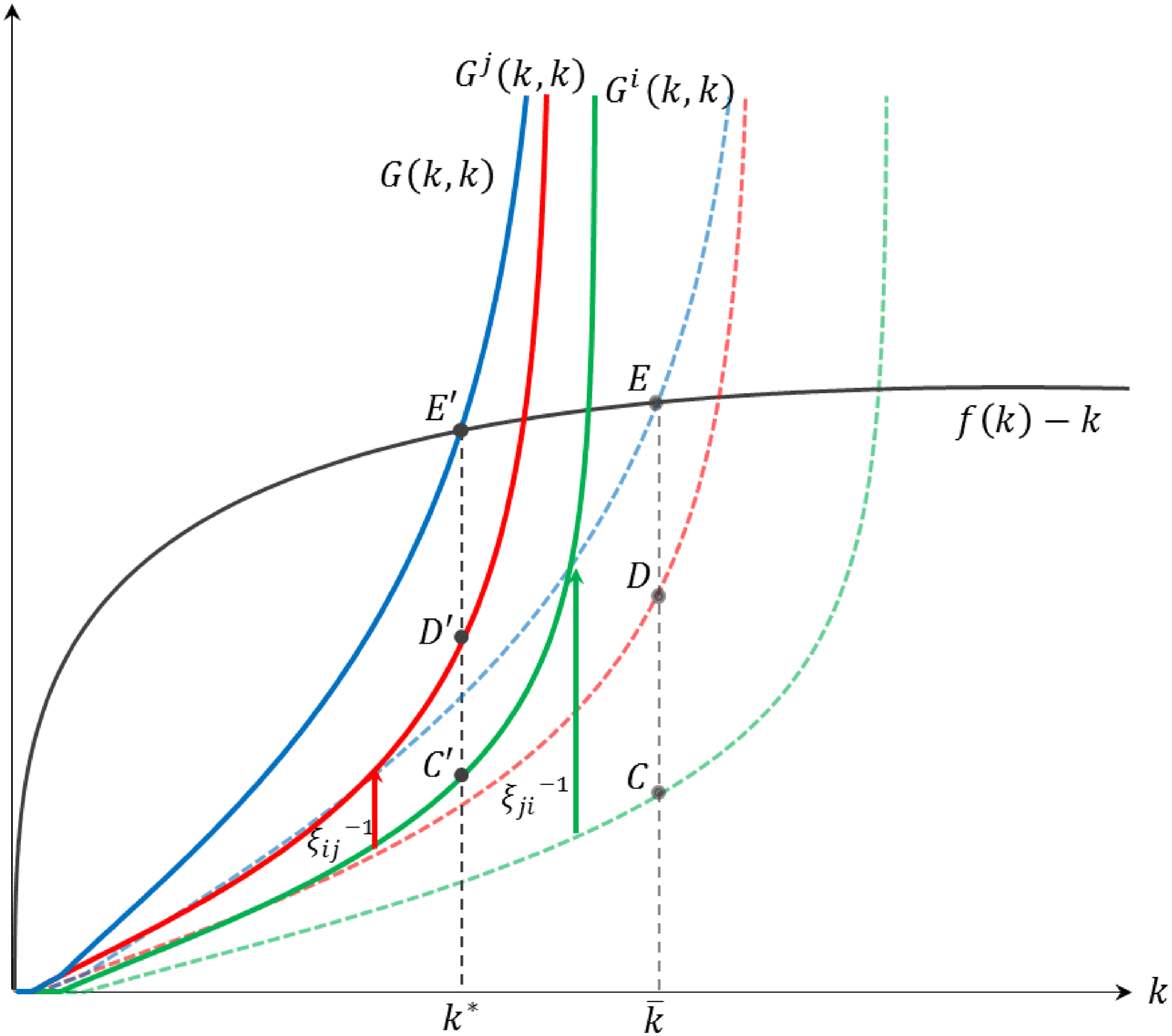}
\caption{Stationary equilibrium in Markov strategies}
\label{fig:stateqME}
\end{figure}

Given that both $\xi_{ij}$ and $\xi_{ji}$ are strictly positive and less than one, and ignoring second-order effects, it follows from \eqref{eq:SPMdisc} that $f'(k^*) > f'(\ov{k})$, hence $k^* < \ov{k}$. In other words, the precommitment equilibrium can be obtained from this condition by setting $\xi_{ij}=\xi_{ji}=1$ for all $k \in I$. As $\alpha_i \geq \alpha_j$ near $k^*$, it also follows from \eqref{eq:SPMdisc}  that $\xi_{ij}(k^*,k^*) < \xi_{ji}(k^*,k^*)$. This determines the shifts of the $G^i$ and $G^j$ curves in the graph (shown in upward arrows). Intuitively, when the possibility of commitment is no longer feasible, players internalize the effect that their actions have on their opponents. We can think that $\xi_{ij}$ and  $\xi_{ji}$ measure the burden that each player imposes on the other player by increasing the marginal cost of investment. Since $\xi_{ij} < \xi_{ji}$, the burden is heavier for the patient player $i$, than the other way around. As a result, the impatient player $j$ reduces consumption (from $D$ to $D'$) while player $i$ consumes more (from $C$ to $C'$).

\section{Concluding Remarks}
\label{sec:conclusions}

This paper studies a two-agent capital accumulation model with heterogeneity in preferences and income shares. Preferences are represented by recursive utility functions with decreasing marginal impatience. Given that agents share a single production technology, individual consumption and saving decisions are interrelated. This mutual influence among agents also creates incentives for strategic behavior. The stationary equilibria of this dynamic game are analyzed under two information structures: one in which agents precommit to future actions, and another one where agents use Markovian strategies. 

We develop a novel approach for the solution of stationary equilibria in this class of capital accumulation games with recursive preferences. The method results from a combination of different approaches, which include the dynamical systems approach, operator methods, and fixed point theory. In the case of Markovian strategies, Euler equations arising from each agent's dynamic program are transformed into a system of partial differential equations, which are solved by implicit programming. 

Despite the technical difficulties involved, we are able to prove the existence of stationary equilibria in open-loop and Markovian strategies, and characterize these equilibria based on general functional forms. Aggregate capital in any stationary equilibria is bounded above by the equilibrium level of the less patient agent in the autarky setting. Moreover, agents reduce capital accumulation in a Markovian equilibrium compared to the precommitment case. Under certain regularity conditions,  it is shown that convergence is monotone under precommitment, but Markovian equilibria may exhibit nonmonotonic paths for aggregate variables, even in the long-run. 

\appendix

\section{The Strategy Space: Proofs}
\label{sec:appSP}

\subsection{Proof of Lemma \ref{lem:stateqAU}}

For simplicity, assume that $\alpha_i(c) \geq \alpha_j(c)$ holds for all $c \in [0,k_m]$, since the argument does not hinge on this particular assumption. By \ref{asm:prodfp}, we have that $\alpha_i(0)f'(0^+) > 1$ (it could be infinity). From  \ref{asm:prodf}--\ref{asm:prodfp}, the maximum sustainable level satisfies $f'(k_m) < 1$. This, together with \ref{asm:alpha} and the fact that $f(k_m)-k_m > 0$, implies $\alpha_i(k_m)f'(k_m) < 1$. Hence the existence of $k_a^i \in (0,k_m)$ follows from the continuity of $\alpha_i$ and $f'$. The proof for $k_a^j \in (0,k_m)$ is analogous. 

By \eqref{eq:SPaut}, we have in a stationary equilibrium that
\begin{equation}
\label{eq:autequal}
1=\alpha_i(c_a^i)f'(k_a^i)=\alpha_j(c_a^j)f'(k_a^j),
\end{equation}
where $0 < c_a^i = f(k_a^i)-k_a^i$ and $0 < c_a^j = f(k_a^j)-k_a^j$. Suppose that $k_a^i < k_a^j$. Hence $f'(k_a^i) > f'(k_a^j) >1$, and it follows that $\alpha_i(c_a^i)<\alpha_j(c_a^j)$ from  \eqref{eq:autequal}. For this inequality to hold, $c_a^i$ must be sufficiently lower than $c_a^j$. In particular, $0 < f(k_a^i)-k_a^i < f(k_a^j)-k_a^j$, or, equivalently,
\begin{equation*}
\frac{f(k_a^i)-f(k_a^j)}{k_a^i-k_a^j} < 1.
\end{equation*}
Then, by the mean value theorem, there is a $k_a^i < x < k_a^j$ such that $f'(x) < 1$. But this contradicts the concavity of $f$. Hence, $k_a^j \leq k_a^i$. The other inequality, $c_a^j \leq c_a^i$, follows from the concavity of $f$ and the fact that both $k_j^a$ and $k_i^a$ satisfy $f'(k)> 1$. \qed

\section{Precommitment Strategies: Details and Proofs}
\label{sec:appOL}

\subsection{Proof of Proposition \ref{prp:localreg}}

Since \eqref{eq:stptOLi} and \eqref{eq:stptOLj} hold for any interior stationary point, we have that
\begin{subequations}
\label{eq:autvsOL}
\begin{align}
\label{eq:autvsOL1}
1 & = \alpha_i(\ov{c}^i)f'(\ov{k})= \alpha_i(c_a^i)f'(k_a^i),
\\[5pt]
\label{eq:autvsOL2}
1 & = \alpha_j(\ov{c}^j)f'(\ov{k})= \alpha_j(c_a^j)f'(k_a^j),
\end{align}
\end{subequations}
which implies $\alpha_i(\ov{c}^i) < 1$ and $\alpha_j(\ov{c}^j) < 1$ by \ref{asm:alpha}. This in turn implies $f'(\ov{k}) > 1$, $f'(k_a^i) > 1$, and $f'(k_a^j) > 1$. Next, divide \eqref{eq:autvsOL1} by \eqref{eq:autvsOL2} to obtain
\begin{equation*}
1=\frac{\alpha_i(\ov{c}^i)}{\alpha_j(\ov{c}^j)}=\frac{\alpha_i(c_a^i)f'(k_a^i)}{\alpha_j(c_a^j)f'(k_a^j)}.
\end{equation*}
Given that $\alpha_i(\cdot) \geq \alpha_j(\cdot)$, for the first equality to hold it must be the case that $\ov{c}^i \leq  \ov{c}^j$, which proves (i).

For part (ii), note that from \eqref{eq:autvsOL1}--\eqref{eq:autvsOL2} and the resource constraint, it follows that
\begin{align*}
1 = \alpha_i(f(\ov{k})-\ov{k}-\ov{c}^j)f'(\ov{k})
 =\alpha_i\left(f(\ov{k})-\ov{k}-\alpha_j^{-1}\left(1/f'(\ov{k})\right)\right)f'(\ov{k}).
\end{align*}
From the second inequality above, the nonnegativity condition $\ov{c}^i \geq 0$ is equivalent to
\begin{align*}
\alpha_j(f(\ov{k})-\ov{k})f'(\ov{k}) \geq 1 = \alpha_j(f(k_a^j)-k_a^j)f'(k_a^j).
\end{align*}
Hence, $\ov{k} \leq k_a^j$. The remaining inequality $\ov{k} \leq k_a^i$ follows from Lemma \ref{lem:stateqAU}.

To show part (iii), rearrange \eqref{eq:autvsOL1} and apply the result from part (ii) of this proposition to obtain
\begin{equation*}
\frac{\alpha_i(\ov{c}^i)}{\alpha_i(c_a^i)}=\frac{f'(k_a^i)}{f'(\ov{k})} \leq 1,
\end{equation*}
hence the monotonicity of the discount factor implies $\ov{c}^i \leq c_a^i$. The remaining inequality can be obtained applying a similar argument to \eqref{eq:autvsOL2}. This completes the proof.\qed

\subsection{Proof of Proposition \ref{prp:locstabOL}}

By local regularity, it follows that
\begin{equation*}
\frac{\alpha_i'}{\alpha_i}(f'-1)+\frac{f''}{f'} < 0 \quad \text{and} \quad \frac{\alpha_j'}{\alpha_j}(f'-1)+\frac{f''}{f'} < 0
\end{equation*}
on $I^i \times I^j \times I^k$. 

Multiplying both sides of the first inequality above by $\alpha_i'/\alpha_i$, multiplying both sides of the second inequality by $\alpha_j'/\alpha_j$, and adding up the result, we have
\begin{equation*}
\left[\left(\frac{\alpha_i'}{\alpha_i}\right)^2+\left(\frac{\alpha_j'}{\alpha_j}\right)^2\right](f'-1)+\left(\frac{\alpha_i'}{\alpha_i}+\frac{\alpha_j'}{\alpha_j}\right)\frac{f''}{f'} < 0.
\end{equation*}
Given that $\alpha_i'/\alpha_i,\,\alpha_j'/\alpha_j \geq 0$, Young's inequality\footnote{Young's inequality states that if $a$ and $b$ are nonnegative real numbers, and $p$ and $q$ positive real numbers such that $1/p+1/q=1$, then $ab \leq a^p/p+b^q/q$.} implies that
\begin{equation*}
\frac{\alpha_i'}{\alpha_i}\frac{\alpha_j'}{\alpha_j}\leq \frac{1}{2}\left(\frac{\alpha_i'}{\alpha_i}\right)^2+\frac{1}{2}\left(\frac{\alpha_j'}{\alpha_j}\right)^2 \leq \left(\frac{\alpha_i'}{\alpha_i}\right)^2+\left(\frac{\alpha_j'}{\alpha_j}\right)^2.
\end{equation*}
Hence, it follows that
\begin{equation*}
\frac{\alpha_i'}{\alpha_i}\frac{\alpha_j'}{\alpha_j}(f'-1)+\left(\frac{\alpha_i'}{\alpha_i}+\frac{\alpha_j'}{\alpha_j}\right)\frac{f''}{f'} < 0,
\end{equation*}
which is \eqref{eq:locstabOL}, the desired result.
\qed

\subsection{Proof of Theorem \ref{thm:manifoldOL}}

The proof of the theorem will be carried out in several steps which we formulate as independent lemmas. We assume that the matrix $A$ has no eigenvalues lying in the unit circle (hyperbolic fixed point) so the stable manifold theorem applies.

Generally speaking, the graph of a third-degree polynomial,
\begin{equation*}
p(\lambda)=\lambda^3 - \tr(A)\lambda^2 + \tfrac{1}{2}\left(\tr^2(A)-\tr(A^2)\right)\lambda - \det(A),
\end{equation*}
can be characterized by its roots, the intersection with the $y$-axis, a local maximum, a local minimum, and a point of inflection.

\begin{lemm}
\label{lem:Ps}
Assume that $\eta_i,\eta_j < 0$. Then $p(-1) < p(0) < 0 < p(1)$.
\end{lemm}

\begin{proof}
Reordering terms in \eqref{eq:trA} yields
\begin{equation*}
\tr(A)= \frac{2\omega_i\omega_j -\eta_i(\omega_j-\delta_j)-\eta_j(\omega_i-\delta_i)}{\Delta_0}+f'.
\end{equation*}
By \ref{asm:etaneg} and the fact that $\omega_i-\delta_i >0$ and $\omega_j-\delta_j > 0$, it follows that $\tr(A) > 0$. It has also been established that $\det(A) > 0$. The linear coefficient of $p$ is given by
\begin{equation*}
\frac{\tr^2(A)-\tr(A^2)}{2}=\frac{\omega_i\omega_j-\eta_i\omega_j-\omega_i\eta_j}{\Delta_0} + \frac{2\omega_i\omega_j}{\Delta_0}f',
\end{equation*}
which is clearly positive with $\eta_i,\eta_j < 0$. These values completely describe the coefficients of the polynomial $p$. Hence, we will evaluate $p(\cdot)$ at certain reference points to help determine the stable and unstable local manifolds. Here we choose 0, $1$, and $-1$, which gives
\begin{align*}
p(-1)&= -\left(\frac{4\omega_i\omega_j-\delta_i\delta_j}{\Delta_0}\right)(f'+1)+\frac{\eta_i(2\omega_j-\delta_j)+(2\omega_i-\delta_i)\eta_j}{\Delta_0} < 0,\\[5pt]
p(0)&= -\frac{\omega_i\omega_j}{\Delta_0} < 0,\\[5pt]
p(1)&= \frac{\delta_i\delta_j\,(f'-1)-(\delta_i\eta_j+\eta_i\delta_j)}{\Delta_0} > 0.
\end{align*}
Simple calculation shows that $p(-1)< p(0)$. Hence $p(-1) < p(0) < 0 < p(1)$ as claimed.
\end{proof}

The next result is useful to characterize the critical points of $p$.

\begin{lemm}
\label{lem:trAb3}
If $\eta_i < 0$ and $\eta_j < 0$, then $\frac{1}{3}\tr(A) > 1$.
\end{lemm}

\begin{proof}
We have already established that $f'(\bar{k}) > 1$ for any given stationary point. Since the remaining entries of the main diagonal of $A$, are positive, it suffices to show that their sum is greater than two, or equivalently, that $F_1^i + F_1^j -2 > 0$. From \eqref{eq:Fls}, we have that
\begin{align*}
F_1^i +F_1^j - 2 &=\frac{2\delta_i\delta_j-\eta_i(\omega_j-\delta_j)-(\omega_i-\delta_i)\eta_j}{\Delta_0} > 0,
\end{align*}
and the desired result follows.
\end{proof}

\begin{lemm}
\label{lem:critpoints}
Suppose that $\eta_i < 0$, $\eta_j < 0$ and $\tr(A^2)-\tfrac{1}{3}\tr^2(A)>0$. Then, $p(\lambda)$ has a local maximum at $r_1 > 0$, a local minimum at $r_2$, and a point of inflection at $r_3$ with $r_1 < r_3 < r_2$ and $r_3 > 1$.
\end{lemm}

\begin{proof}
First differentiate $p$ to find the critical points, which are the zeroes of $p'(\lambda)$ in $\mathbb{R}$, i.e., the solution of the quadratic equation
\begin{equation*}
\lambda^2-\tfrac{2}{3}\tr(A)\lambda
+\tfrac{1}{6}\left(\tr^2(A)-\tr(A^2)\right)=0.
\end{equation*}
Since $\tr(A^2)-\tfrac{1}{3}\tr^2(A)>0$, the critical points $r_1$ and $r_2$, given by
\begin{equation}
\label{eq:P2roots}
r_{1,2} = \tfrac{1}{3}\tr(A) \mp \tfrac{1}{2}
\sqrt{\tfrac{2}{3}\left(\tr(A^2)-\tfrac{1}{3}\tr^2(A)\right)}
\end{equation}
are well defined.\footnote{Allowing for complex roots does not alter the results in any significant way, so that case is omitted.} By Lemma \ref{lem:trAb3}, it is easy to verify that $0 < r_1 < r_2$.

Differentiating $p'$ once more, we have that
\begin{equation*}
p''(\lambda)=2\lambda-\tfrac{2}{3}\tr(A),
\end{equation*}
thus $p''$ vanishes at $r_3 = \frac{1}{3}\tr(A)$ and
\begin{equation*}
p''(\lambda) < (>)\ 0 \iff \lambda < (>)\ r_3,
\end{equation*}
respectively, so $r_3$ is a point of inflection. By simple inspection of \eqref{eq:P2roots}, it is clear that $r_1 < r_3 < r_2$, which immediately implies that $r_1$ is a local maximum and $r_2$ a local minimum of $p$ in $\mathbb{R}$. Finally, applying  Lemma \ref{lem:trAb3} again, it follows that $r_3 > 1$. This completes the proof.
\end{proof}

\subsection{Analysis of the Stable Manifold}

Let $\Phi:X \to X$ be a map describing the nonlinear discrete dynamical system
\begin{align}
\label{eq:PhiSup} 
x_{t+1} = \Phi(x_t), & & t=0,1,\ldots
\end{align}
and let $\ov{x} \in X$ be a point such that $\ov{x} = \Phi(\ov{x})$, i.e., a stationary point. By the stable manifold theorem, if $\Phi$ is continuously differentiable in a neighborhood $\mathcal{N}$ of $\ov{x}$ and $A=D\Phi(\ov{x})$ is the Jacobian matrix of $\Phi$, then there exists a neighborhood $\mathcal{U} \subset \mathcal{N}$, and a continuously differentiable function $\phi: \mathcal{U} \to \mathbb{R}^2$, for which the matrix $D\phi(\ov{x})$ has full rank 2. Moreover,  if $\{x_t\}$ is a solution to \eqref{eq:PhiSup} with $x_0 \in \mathcal{U}$ and $\phi(x_0)=0$, then $\lim_{t \to \infty} x_t=\ov{x}$. The set of $x$ values satisfying $\phi(x)=0$ is called the \emph{stable manifold} of the nonlinear dynamical system.

By Jordan decomposition, the matrix $A$ of the linearized system \eqref{eq:PhiSup} around $\ov{x}$ can be written as $A=B^{-1} \Lambda B$, where $B$ is nonsingular and $\Lambda$ is a diagonal matrix containing the eigenvalues of $A$. Hence, a solution can be expressed recursively in terms of these matrices
\begin{equation}
\label{eq:stabmanif0}
x_t = \ov{x} + B^{-1} \Lambda^t B\,(x_0-\ov{x}).
\end{equation}
It is clear that $x_t \to \ov{x}$ if and only if $B(x_0-\ov{x})=w_0$, where $w_0^i=w_0^j=0$. Let $\hat{x}_t:=x_t-\ov{x}$ denote deviations from stationary values for all $t$. This condition is equivalent to
\begin{equation}
\label{eq:stabmanif1}
\begin{bmatrix}
b_{11} & b_{12} & b_{13}\\
b_{21} & b_{22} & b_{23}\\
b_{31} & b_{32} & b_{33}\\
\end{bmatrix}
\begin{bmatrix}
\hat{c}_0^i\\
\hat{c}_0^j\\
\hat{k}_0\\
\end{bmatrix}
=
\begin{bmatrix}
0\\
0\\
w\\
\end{bmatrix}
,
\end{equation}
where $b_{31}\hat{c}_0^i + b_{32}\hat{c}_0^j + b_{33}\hat{k}_0=w$ is a constant to be determined.   

The solution \eqref{eq:stabmanif0} implies
\begin{equation*}
\begin{bmatrix}
\hat{c}_t^i \\
\hat{c}_t^j \\
\hat{k}_t   \\
\end{bmatrix}
=\frac{1}{\det(B)}
\begin{bmatrix}
M_{11} & M_{12} & M_{13}\\
M_{21} & M_{22} & M_{23}\\
M_{31} & M_{32} & M_{33}\\
\end{bmatrix}
\begin{bmatrix}
(\ov{\lambda}_1)^t &     0     & 0 \\
    0     & (\ov{\lambda}_2)^t & 0\\
    0     &     0     & (\ov{\lambda}_3)^t\\
\end{bmatrix}
\begin{bmatrix}
b_{11} & b_{12} & b_{13}\\
b_{21} & b_{22} & b_{23}\\
b_{31} & b_{32} & b_{33}\\
\end{bmatrix}
\begin{bmatrix}
\hat{c}_0^i \\
\hat{c}_0^j \\
\hat{k}_0   \\
\end{bmatrix}
,
\end{equation*}
where $M_{ij}$ is the $3 \times 3$ matrix whose elements are the principal minors of $B$ and $\det(B) \neq 0$. Solving the system above yields
\begin{align*}
\hat{c}_t^i =\frac{M_{13}}{\det(B)}\,w(\ov{\lambda}_1)^t, \quad \hat{c}_t^i =\frac{M_{23}}{\det(B)}\,w(\ov{\lambda}_1)^t, \quad \text{and} \quad \hat{k}_t =\frac{M_{33}}{\det(B)}\,w(\ov{\lambda}_1)^t,
\end{align*}
which represents a stable trajectory for all variables, since $0 < \ov{\lambda}_1 < 1$. Now, \eqref{eq:stabmanif1} can be expressed as
\begin{equation*}
\begin{bmatrix}
b_{11} & b_{12}\\
b_{21} & b_{22}\\ 
\end{bmatrix}
\begin{bmatrix}
\hat{c}_0^i\\
\hat{c}_0^j\\
\end{bmatrix}
=-
\begin{bmatrix}
b_{13}\\
b_{23}\\
\end{bmatrix}
\hat{k}_0,
\end{equation*}
and assuming $b_{11}b_{22}-b_{12}b_{21} \neq 0$, it follows that
\begin{equation*}
\begin{bmatrix}
\hat{c}_0^i\\
\hat{c}_0^j\\
\end{bmatrix}
= -\frac{1}{b_{11}b_{22}-b_{12}b_{21}}
\begin{bmatrix*}[r]
b_{22} & -b_{12}\\
-b_{21} & b_{11}\\ 
\end{bmatrix*}
\begin{bmatrix}
b_{13}\\
b_{23}\\
\end{bmatrix}
\hat{k}_0,
\end{equation*}
which yields
\begin{align*}
\hat{c}_0^i =\frac{M_{31}}{M_{33}}\,\hat{k}_0 \quad  \text{and}
\quad \hat{c}_0^j =\frac{M_{32}}{M_{33}}\,\hat{k}_0.
\end{align*}
And solving for $w$ again in \eqref{eq:stabmanif1}, it follows that 
\begin{align*}
w=\left(b_{31}\frac{M_{31}}{M_{33}}+b_{32}\frac{M_{32}}{M_{33}}+b_{33}\right)\hat{k}_0.
\end{align*}

In order to characterize the stable manifold, from the fact that $BA=\Lambda B$, we have
\begin{equation*}
\begin{bmatrix}
b_{11} & b_{12} & b_{13}\\
b_{21} & b_{22} & b_{23}\\
b_{31} & b_{32} & b_{33}\\
\end{bmatrix}
\begin{bmatrix}
\ov{F}_1^i & \ov{F}_2^i & \ov{F}_3^i\\
\ov{F}_2^j & \ov{F}_1^j & \ov{F}_3^j\\
-1    &   -1  & \ov{f}'\\
\end{bmatrix}
=
\begin{bmatrix}
\ov{\lambda}_1 &     0     & 0 \\
    0     & \ov{\lambda}_2 & 0\\
    0     &     0     & \ov{\lambda}_3\\
\end{bmatrix}
\begin{bmatrix}
b_{11} & b_{12} & b_{13}\\
b_{21} & b_{22} & b_{23}\\
b_{31} & b_{32} & b_{33}\\
\end{bmatrix}
,
\end{equation*}
which we can solve for arbitrary nonzero values of $b_{11}$, $b_{22}$, and $b_{33}$, therefore
\begin{align*}
&b_{12}=\left(\frac{\ov{F}_1^i-\ov{F}_2^i-\ov{\lambda}_1}{\ov{F}_1^j-\ov{F}_2^j-\ov{\lambda}_1}\right)b_{11}, \quad 
b_{13}=-\frac{1}{(\ov{f}'-\ov{\lambda}_1)}\left[\ov{F}_3^j\left(\frac{\ov{F}_1^i-\ov{F}_2^i-\ov{\lambda}_1}{\ov{F}_1^j-\ov{F}_2^j-\ov{\lambda}_1}\right)+\ov{F}_3^i\right]b_{11},\\[5pt] 
&b_{21}=\left(\frac{\ov{F}_1^j-\ov{F}_2^j-\ov{\lambda}_2}{\ov{F}_1^i-\ov{F}_2^i-\ov{\lambda}_2}\right)b_{22}, \quad b_{23}=-\frac{1}{(\ov{f}'-\ov{\lambda}_2)}\left[\ov{F}_3^i\left(\frac{\ov{F}_1^j-\ov{F}_2^j-\ov{\lambda}_2}{\ov{F}_1^i-\ov{F}_2^i-\ov{\lambda}_2}\right)+\ov{F}_3^j\right]b_{22},
\\[5pt]
&b_{31}=\left[\frac{\ov{F}_1^j-\ov{F}_2^j-\ov{\lambda}_3}{(\ov{F}_1^i-\ov{\lambda}_3)(\ov{F}_1^j-\ov{\lambda}_3)-\ov{F}_2^i\ov{F}_2^j}\right]b_{33}, \quad \text{and}\\[5pt]
&b_{32}=-\frac{1}{\ov{F}_3^j}\left[\frac{\ov{F}_3^i(\ov{F}_1^j-\ov{F}_2^j-\ov{\lambda}_3)}{(\ov{F}_1^i-\ov{\lambda}_3)(\ov{F}_1^j-\ov{\lambda}_3)-\ov{F}_2^i\ov{F}_2^j}+(\ov{f}'-\ov{\lambda}_3)\right]b_{33}.
 \end{align*}

The stable manifold is the set of values $(c^i,c^j,k)\in\mathcal{U}$ such that $\phi(c^i,c^j,k)=0$ holds. By the implicit function theorem, there exist functions $\pi_i$ and $\pi_j$ such that
\begin{align*}
\phi^i\left(\pi_i(k),\pi_j(k),k\right)=0,\quad \text{and} \quad 
\phi^j\left(\pi_i(k),\pi_j(k),k\right)=0.
\end{align*}
Differentiating around $\ov{k}$, we obtain
\begin{equation}
\label{eq:stabmanif2}
\begin{bmatrix}
\ov{\phi}_1^i & &\ov{\phi}_2^i\\[5pt]
\ov{\phi}_1^j & &\ov{\phi}_2^j\\
\end{bmatrix}
\begin{bmatrix}
\pi_i'(\ov{k})\\[5pt]
\pi_j'(\ov{k})\\
\end{bmatrix}
=
-
\begin{bmatrix}
\ov{\phi}_3^i\\[5pt]
\ov{\phi}_3^j\\
\end{bmatrix}
,
\end{equation}
where $\ov{\phi}_l^i$, $l=1,2,3$, denotes the partial derivative of $\phi^i$ with respect to the $l$-th argument,  evaluated at the stationary point, and similarly for $\ov{\phi}_l^j$. In fact, the derivatives of the stable manifold at the stationary point are related to the coefficients of $B$ as follows, 
\begin{equation*}
D\phi(\ov{x})=
\begin{bmatrix}
\ov{\phi}_1^i & \ov{\phi}_2^i & \ov{\phi}_3^i\\[5pt]
\ov{\phi}_1^j & \ov{\phi}_2^j & \ov{\phi}_3^j\\
\end{bmatrix}
=
\begin{bmatrix}
b_{11} & b_{12} & b_{13}\\[5pt]
b_{21} & b_{22} & b_{23}\\
\end{bmatrix}
.
\end{equation*}
This allows to solve for the derivatives of the policy functions from \eqref{eq:stabmanif2}, yielding
\begin{align*}
\ov{\pi}_i':=\pi_i'(\ov{k})&=-\left(\frac{b_{13}b_{22} - b_{12}b_{23}}{b_{11}b_{22} - b_{12}b_{21}}\right) \quad \text{and} \quad \ov{\pi}_j':=\pi_j'(\ov{k})=-\left(\frac{b_{11}b_{23} - b_{13}b_{21}}{b_{11}b_{22} - b_{12}b_{21}}\right).
\end{align*}
Equivalently,
\begin{align*}
\ov{\pi}_i'&=\frac{\left[\ov{F}_3^i+\ov{E}^{ij}(\ov{\lambda}_1)\ov{F}_3^j\right](\ov{f}'-\ov{\lambda}_2)-\ov{E}^{ij}(\ov{\lambda}_1)\left[\ov{F}_3^i\ov{E}^{ji}(\ov{\lambda}_2)+\ov{F}_3^j\right](\ov{f}'-\ov{\lambda}_1)}
{(\ov{f}'-\ov{\lambda}_1)(\ov{f}'-\ov{\lambda}_2)\left[1-\ov{E}^{ij}(\ov{\lambda}_1)\ov{E}^{ji}(\ov{\lambda}_2)\right]},
\\[5pt]
\ov{\pi}_j'&=\frac{\left[\ov{F}_3^i\ov{E}^{ji}(\ov{\lambda}_2)+\ov{F}_3^j\right](\ov{f}'-\ov{\lambda}_1)-\ov{E}^{ji}(\ov{\lambda}_2)\left[\ov{F}_3^i+\ov{E}^{ij}(\ov{\lambda}_1)\ov{F}_3^j\right](\ov{f}'-\ov{\lambda}_2)}
{(\ov{f}'-\ov{\lambda}_1)(\ov{f}'-\ov{\lambda}_2)\left[1-\ov{E}^{ij}(\ov{\lambda}_1)\ov{E}^{ji}(\ov{\lambda}_2)\right]},
\end{align*}
where 
\begin{align*}
\ov{E}^{ij}(\ov{\lambda}_1):=\frac{\ov{F}_1^i-\ov{F}_2^i-\ov{\lambda}_1}{\ov{F}_1^j-\ov{F}_2^j-\ov{\lambda}_1}
\quad \text{and} \quad
\ov{E}^{ji}(\ov{\lambda}_2):=\frac{\ov{F}_1^j-\ov{F}_2^j-\ov{\lambda}_2}{\ov{F}_1^i-\ov{F}_2^i-\ov{\lambda}_2}.
\end{align*}
Note that an appropriate set of conditions must be imposed on the model's parameters for all previous values to be well defined, in particular $f' - \ov{\lambda}_1 \neq 0$, $f' - \ov{\lambda}_2 \neq 0$, $\ov{F}_1^j-\ov{F}_2^j-\ov{\lambda}_1 \neq 0$, $\ov{F}_1^i-\ov{F}_2^i-\ov{\lambda}_2 \neq 0$, and 
$(\ov{F}_1^j-\ov{F}_2^j-\ov{\lambda}_1)(\ov{F}_1^i-\ov{F}_2^i-\ov{\lambda}_2)-(\ov{F}_1^i-\ov{F}_2^i-\ov{\lambda}_1)(\ov{F}_1^j-\ov{F}_2^j-\ov{\lambda}_2) \neq 0$.

\section{Markov Strategies: Details and Proofs}
\label{sec:appM}

\subsection{Proof of Proposition \ref{prp:kstardef}}

Let $g$ be a continuous function satisfying \eqref{eq:Psi}. If $g(k)$ has a stationary point $k^* > 0$, then
\begin{align*}
\Psi\left[k^*,g(k^*),g^2(k^*)\right]=\Psi(k^*,k^*,k^*)=0.
\end{align*}
But then, \eqref{eq:hdef1} and \eqref{eq:hdef2} imply
\begin{equation*}
k^*=g(k^*)=f(k^*)-G^i(k^*,k^*)-G^j(k^*,k^*).
\end{equation*}
Given that the strategies $G^i$, $G^j$ solve the first-order conditions \eqref{eq:focM}, and taking into account that $U^i,U^j > 0$ for all $k \in (0,k_m]$,
\begin{align*}
\alpha_i\left(G^i(k^*,k^*)\right)[f'(k^*)-G_1^j(k^*,k^*)]&=1,\\[5pt]
\alpha_i\left(G^i(k^*,k^*)\right)[f'(k^*)-G_1^j(k^*,k^*)]&=1,
\end{align*}
where $f'(k^*)-G_1^i(k^*,k^*)>0$, $i=1,2$, holds from optimality conditions. \qed

\subsection{Proof of Lemma \ref{lem:m1m2def}}

Given that $P(\lambda) < 0$ for any $\lambda < \lambda_1^*$,
it follows from \eqref{eq:poly2} that
$\Psi_1^*+\Psi_3^*\,\lambda^2 < -\Psi_2^*\,\lambda$,
and, since $\Psi_2^*>0$, it also follows that
\begin{align*}
H_1^*+H_2^*\,\lambda^2 = -\frac{1}{\Psi_2^*}\left(\Psi_1^*+\Psi_3^*\lambda^2\right) & > \lambda.
\end{align*}
Denote by $Q(\lambda)$ the quadratic polynomial on the left-hand side of the above equality. Note that $Q(\lambda_1^*)=\lambda_1^* < 0$, $Q(-\lambda_2^*) > 0$, and $Q$ is strictly increasing for every $\lambda < \lambda_1^*$. Then, there is a unique root for $Q(\lambda)=0$ in $(-\lambda_2^*,\lambda_1^*)$, which is $-\left(|\lambda_1^*\lambda_2^*|\right)^{-1/2}$. This value can be less or greater than $-1$, depending on the parameters. Hence the lower bound is set to
\begin{equation*}
\unwt{\lambda}:= \max\left\{-1,-(|\lambda_1^*\lambda_2^*|)^{-1/2}\right\},
\end{equation*}
so that $\lambda < Q(\lambda) < 0$ for every $\lambda \in (\unwt{\lambda},\lambda_1^*)$. The existence of $m_1 > H_1^*$ and $m_2 > H_2^*$ satisfying $\lambda < Q(\lambda) < m_1 + m_2 \lambda^2 < 0$ on that interval follows from continuity. This completes the proof. \qed

\subsection{Proofs ommited from Theorem \ref{thm:fixedpT}}

In order to prove the continuity of $T$, we need a preliminary result.

\begin{lemm}
\label{lem:hn2converg}
Let $\{h_n\}_{n \in \mathbb{N}}$ be a sequence of functions in $\upwt{D}_{|\lambda|}(I_\varepsilon)$ converging to $h$. Then, the sequence $h^2_n(k):=h(h_n(k))$ converges uniformly on $I_\varepsilon$ to $h^2(k):=h(h(k))$.
\end{lemm}

\begin{proof}
Since the family of functions in $\upwt{D}_{|\lambda|}(I_\varepsilon)$ is uniformly bounded and equicontinuous, by the Arzel\`{a}-Ascoli theorem, $h_n$ converges uniformly on $I_\varepsilon$. Given that $h$ is continuous, hence uniformly continuous, for every $\varepsilon > 0$, there is a $\delta > 0$ such that    $k,k' \in I_\varepsilon$ with $|k-k'|< \delta$ implies $|h(k)-h(k')|< \varepsilon$. On the other hand, there is a positive integer $N$ such that $|h_n(k) - h(k)|<\delta$ for all $n>N$ and all $k \in I_\varepsilon$. The combination of both results immediately implies that for every $\varepsilon > 0$, there is some $N$ such that $|h(h_n(k))-h(h(k))| < \varepsilon$ for all $n > N$ and all $k \in I_\varepsilon$. Hence $h^2_n(k)$ converges uniformly on $I_\varepsilon$ to $h^2(k)$.
\end{proof}

\begin{lemm}
\label{lem:Tcontin}
Let $\lambda \in (\lambda_2^*,\upwt{\lambda})$ and suppose  $I_\varepsilon$ is defined as in Lemma \ref{lem:m1m2def}. Then, the operator $T:\upwt{D}_{|\lambda|}(I_\varepsilon) \to \upwt{D}_{|\lambda|}(I_\varepsilon)$ is continuous in the sup norm.
\end{lemm}

\begin{proof}
Let $h_n$ be a sequence in $\upwt{D}_{|\lambda|}(I_\varepsilon)$ that converges to $h$. By a similar argument given in the main text, there exists $m_2' > 0$ such that
\begin{align*}
\left|(Th_n)(k)-(Th)(k)\right|&=\left|H[k,h_n^2(k)]-H[k,h^2(k)]\right| \\[5pt]
&\leq m_2'\left|h_n^2(k)-h^2(k)\right|\\[5pt]
& \leq m_2'|\lambda|\left|h_n(k)-h(k)\right|,
\end{align*}
for all $n$ and all $k \in I_\varepsilon$. Then, for some $0 < \delta < \varepsilon/(m_2'|\lambda|)$, if $\Vert h_n-h\Vert < \delta$, we have
\begin{align*}
\left\Vert Th_n-Th \right\Vert &= \sup_{k \in I_\varepsilon}\left|H[k,h_n^2(k)]-H[k,h^2(k)]\right|\\[5pt]
& \leq m_2'|\lambda|\left|h_n(k)-h(k)\right|\\[5pt]
& \leq m_2'|\lambda|\delta < \varepsilon,
\end{align*}
which proves that $T$ is continuous.
\end{proof}

\bibliographystyle{chicago}
\bibliography{SGRP}

\begin{thebibliography}{}

\bibitem[\protect\citeauthoryear{Arrow and Debreu}{Arrow and
  Debreu}{1954}]{arrowdebreu54}
Arrow, K.~J. and G.~Debreu (1954).
\newblock Existence of an equilibrium for a competitive economy.
\newblock {\em Econometrica\/}~{\em 22\/}(3), 265--290.

\bibitem[\protect\citeauthoryear{Banks and Duggan}{Banks and
  Duggan}{2004}]{banksduggan04}
Banks, J. and J.~Duggan (2004).
\newblock Existence of {N}ash equilibria on convex sets.
\newblock Working paper, W. Allen Wallis Institute of Political Economy,
  University of Rochester.

\bibitem[\protect\citeauthoryear{Beals and Koopmans}{Beals and
  Koopmans}{1969}]{bealskoop69}
Beals, R. and T.~Koopmans (1969).
\newblock Maximizing stationary utility in a constant technology.
\newblock {\em SIAM Journal of Applied Mathematics\/}~{\em 17\/}(5),
  1001--1015.

\bibitem[\protect\citeauthoryear{Becker and Mulligan}{Becker and
  Mulligan}{1997}]{beckermull97}
Becker, G. and C.~Mulligan (1997).
\newblock The endogenous determination of time preference.
\newblock {\em Quarterly Journal of Economics\/}~{\em 112\/}(3), 729--758.

\bibitem[\protect\citeauthoryear{Becker}{Becker}{2006}]{becker06}
Becker, R. (2006).
\newblock Equilibrium dynamics with many agents.
\newblock In R.-A. Dana, C.~Le~Van, T.~Mitra, and K.~Nishimura (Eds.), {\em
  Handbook on Optimal Growth 1. Discrete Time}, pp.\  385--442. Springer.

\bibitem[\protect\citeauthoryear{Becker, Dubey, and Mitra}{Becker
  et~al.}{2014}]{beckeretal14}
Becker, R., R.~Dubey, and T.~Mitra (2014).
\newblock On {R}amsey equilibrium: capital ownership pattern and inefficiency.
\newblock {\em Economic Theory\/}~{\em 55\/}(3), 565--600.

\bibitem[\protect\citeauthoryear{Becker and Foias}{Becker and
  Foias}{1998}]{beckerfoias98}
Becker, R. and C.~Foias (1998).
\newblock Implicit programming and the invariant manifold for {R}amsey
  equilibria.
\newblock In Y.~Abramovich, E.~Avgerinos, and N.~Yannelis (Eds.), {\em
  Functional Analysis and Economic Theory}, pp.\  119--144. Springer-Verlag.

\bibitem[\protect\citeauthoryear{Becker and Foias}{Becker and
  Foias}{2007}]{beckerfoias07}
Becker, R. and C.~Foias (2007).
\newblock Strategic {R}amsey equilibrium dynamics.
\newblock {\em Journal of Mathematical Economics\/}~{\em 43\/}(3), 318--346.

\bibitem[\protect\citeauthoryear{Boyd}{Boyd}{1990}]{boyd90}
Boyd, J. (1990).
\newblock Recursive utility and the {R}amsey problem.
\newblock {\em Journal of Economic Theory\/}~{\em 50\/}(2), 326--345.

\bibitem[\protect\citeauthoryear{Camacho, Saglam, and Turan}{Camacho
  et~al.}{2013}]{camachoetal13}
Camacho, C., C.~Saglam, and A.~Turan (2013).
\newblock Strategic interaction and dynamics under endogenous time preference.
\newblock {\em Journal of Mathematical Economics\/}~{\em 49\/}(4), 291--301.

\bibitem[\protect\citeauthoryear{Carlson}{Carlson}{2002}]{carlson02}
Carlson, D. (2002).
\newblock Uniqueness of normalized {N}ash equilibrium for a class of games with
  strategies in {B}anach spaces.
\newblock In G.~Zaccour (Ed.), {\em Decision \& Control in Management Science:
  Essays in honor of {A}lain {H}aurie}, pp.\  333--348. Springer.

\bibitem[\protect\citeauthoryear{Chade and Swinkels}{Chade and
  Swinkels}{2016}]{chadeswink16}
Chade, H. and J.~Swinkels (2016).
\newblock The no-upward-crossing condition and the moral hazard problem.
\newblock Working paper, Department of Economics, Arizona State University.

\bibitem[\protect\citeauthoryear{Coleman}{Coleman}{1991}]{coleman91}
Coleman, W.~J. (1991).
\newblock Equilibrium in a production economy with an income tax.
\newblock {\em Econometrica\/}~{\em 59\/}(4), 1091--1104.

\bibitem[\protect\citeauthoryear{Coleman}{Coleman}{1997}]{coleman97}
Coleman, W.~J. (1997).
\newblock Equilibria in distorted infinite-horizon economies with capital and
  labor.
\newblock {\em Journal of Economic Theory\/}~{\em 72\/}(2), 446--461.

\bibitem[\protect\citeauthoryear{Coleman}{Coleman}{2000}]{coleman00}
Coleman, W.~J. (2000).
\newblock Uniqueness of an equilibrium in infinite-horizon economies subject to
  taxes and externalities.
\newblock {\em Journal of Economic Theory\/}~{\em 95\/}(1), 71--78.

\bibitem[\protect\citeauthoryear{Das}{Das}{2003}]{das03}
Das, M. (2003).
\newblock Optimal growth with decreasing marginal impatience.
\newblock {\em Journal of Economic Dynamics and Control\/}~{\em 27\/}(10),
  1881--1898.

\bibitem[\protect\citeauthoryear{Datta, Mirman, and Reffett}{Datta
  et~al.}{2002}]{dattmirmreff02}
Datta, M., L.~J. Mirman, and K.~Reffett (2002).
\newblock Existence and uniqueness of equilibrium in distorted dynamic
  economies with capital and labor.
\newblock {\em Journal of Economic Theory\/}~{\em 103\/}(2), 377--410.

\bibitem[\protect\citeauthoryear{Dockner and Nishimura}{Dockner and
  Nishimura}{2004}]{docknishim04}
Dockner, E. and K.~Nishimura (2004).
\newblock Strategic growth.
\newblock {\em Journal of Difference Equations and Applications\/}~{\em
  10\/}(5), 515--527.

\bibitem[\protect\citeauthoryear{Dockner and Nishimura}{Dockner and
  Nishimura}{2005}]{docknishim05}
Dockner, E. and K.~Nishimura (2005).
\newblock Capital accumulation games with a non-concave production function.
\newblock {\em Journal of Economic Behavior and Organization\/}~{\em 57\/}(4),
  408--420.

\bibitem[\protect\citeauthoryear{Drugeon and Wigniolle}{Drugeon and
  Wigniolle}{2015}]{drugeonwign15}
Drugeon, J.-P. and B.~Wigniolle (2015).
\newblock On impatience, temptation and {R}amsey's conjecture.
\newblock {\em Economic Theory\/}, 1--26.

\bibitem[\protect\citeauthoryear{Epstein}{Epstein}{1987}]{epstein87}
Epstein, L.~G. (1987).
\newblock A simple dynamic general equilibrium model.
\newblock {\em Journal of Economic Theory\/}~{\em 41\/}(1), 68--95.

\bibitem[\protect\citeauthoryear{Erol, Le~Van, and Saglam}{Erol
  et~al.}{2011}]{eroletal11}
Erol, S., C.~Le~Van, and C.~Saglam (2011).
\newblock Existence, optimality and dynamics of equilibria with endogenous time
  preference.
\newblock {\em Journal of Mathematical Economics\/}~{\em 47\/}(2), 170--179.

\bibitem[\protect\citeauthoryear{Fisher}{Fisher}{1930}]{fisher30}
Fisher, I. (1930).
\newblock {\em The Theory of Interest}.
\newblock New York: Macmillan.

\bibitem[\protect\citeauthoryear{Galor}{Galor}{2007}]{galor07}
Galor, O. (2007).
\newblock {\em Discrete dynamical systems}.
\newblock Springer.

\bibitem[\protect\citeauthoryear{Gnana~Bhaskar and
  Lakshmikantham}{Gnana~Bhaskar and Lakshmikantham}{2006}]{bhasklaksh06}
Gnana~Bhaskar, T. and V.~Lakshmikantham (2006).
\newblock Fixed point theorems in partially ordered metric spaces and
  applications.
\newblock {\em Nonlinear Analysis: Theory, Methods \&Applications\/}~{\em
  65\/}(7), 1379--1393.

\bibitem[\protect\citeauthoryear{Greenwood and Huffman}{Greenwood and
  Huffman}{1995}]{greenhuff95}
Greenwood, J. and G.~Huffman (1995).
\newblock On the existence of nonoptimal equilibria in dynamic stochastic
  economies.
\newblock {\em Journal of Economic Theory\/}~{\em 65\/}(2), 611--623.

\bibitem[\protect\citeauthoryear{Gul and Pesendorfer}{Gul and
  Pesendorfer}{2004}]{gulpesendorf04}
Gul, F. and W.~Pesendorfer (2004).
\newblock Self-control and the theory of consumption.
\newblock {\em Econometrica\/}~{\em 72\/}(1), 119--158.

\bibitem[\protect\citeauthoryear{Guo and Lakshmikantham}{Guo and
  Lakshmikantham}{1987}]{guolaksh87}
Guo, D. and V.~Lakshmikantham (1987).
\newblock Coupled fixed points of nonlinear operators with applications.
\newblock {\em Nonlinear Analysis: Theory, Methods \& Applications\/}~{\em
  11\/}(5), 623--632.

\bibitem[\protect\citeauthoryear{Iwai}{Iwai}{1972}]{iwai72}
Iwai, K. (1972).
\newblock Optimal economic growth and stationary ordinal utility. a
  {F}ischerian approach.
\newblock {\em Journal of Economic Theory\/}~{\em 5\/}(1), 121--151.

\bibitem[\protect\citeauthoryear{Koopmans}{Koopmans}{1960}]{koopmans60}
Koopmans, T.~C. (1960).
\newblock Stationary ordinal utility and impatience.
\newblock {\em Econometrica\/}~{\em 28\/}(2), 287--309.

\bibitem[\protect\citeauthoryear{Laibson}{Laibson}{1997}]{laibson97}
Laibson, D. (1997).
\newblock Golden eggs and hyperbolic discounting.
\newblock {\em Quarterly Journal of Economics\/}~{\em 112\/}(2), 443--477.

\bibitem[\protect\citeauthoryear{Lakshmikantham and
  \'{C}iri\'{c}}{Lakshmikantham and \'{C}iri\'{c}}{2009}]{lakshciric09}
Lakshmikantham, V. and L.~\'{C}iri\'{c} (2009).
\newblock Coupled fixed point theorems for nonlinear contractions in partially
  ordered metric spaces.
\newblock {\em Nonlinear Analysis: Theory, Methods \& Applications\/}~{\em
  70\/}(12), 4341--4349.

\bibitem[\protect\citeauthoryear{Lucas and Stokey}{Lucas and
  Stokey}{1984}]{lucasstokey84}
Lucas, R.~E. and N.~Stokey (1984).
\newblock Optimal growth with many consumers.
\newblock {\em Journal of Economic Theory\/}~{\em 32\/}(1), 139--171.

\bibitem[\protect\citeauthoryear{Mantel}{Mantel}{1967}]{mantel67}
Mantel, R.~R. (1967).
\newblock Maximization of utility over time with a variable rate of time
  preference.
\newblock Discussion paper, Cowles Foundation for Research in Economics,
  CF-70525(2), Yale University.

\bibitem[\protect\citeauthoryear{Mantel}{Mantel}{1993}]{mantel93}
Mantel, R.~R. (1993).
\newblock Grandma's dress, or what's new for optimal growth.
\newblock {\em Revista de An\'{a}lisis Econ\'{o}mico\/}~{\em 8\/}(1), 61--81.

\bibitem[\protect\citeauthoryear{Mantel}{Mantel}{1995}]{mantel95}
Mantel, R.~R. (1995).
\newblock Why the rich get richer and the poor get poorer.
\newblock {\em Estudios de Econom\'{i}a\/}~{\em 22\/}(2), 177--205.

\bibitem[\protect\citeauthoryear{Mantel}{Mantel}{1999}]{mantel99}
Mantel, R.~R. (1999).
\newblock Optimal economic growth with recursive preferences: Decreasing rate
  of time preference.
\newblock {\em Econ\'{o}mica\/}~{\em 45\/}(2), 331--348.

\bibitem[\protect\citeauthoryear{McKenzie}{McKenzie}{1959}]{mckenzie59}
McKenzie, L. (1959).
\newblock On the existence of general equilibrium for a competitive market.
\newblock {\em Econometrica\/}~{\em 54}, 54--71.

\bibitem[\protect\citeauthoryear{Mirman, Morand, and Reffett}{Mirman
  et~al.}{2008}]{mirmorref08}
Mirman, L.~J., O.~Morand, and K.~Reffett (2008).
\newblock A qualitative approach to {M}arkovian equilibrium in infinite horizon
  economies with capital.
\newblock {\em Journal of Economic Theory\/}~{\em 139\/}(1), 75--98.

\bibitem[\protect\citeauthoryear{Nieto, Pouso, and
  Rodr\'{i}guez-L\'{o}pez}{Nieto et~al.}{2007}]{nietoetal07}
Nieto, J.~J., R.~L. Pouso, and R.~Rodr\'{i}guez-L\'{o}pez (2007).
\newblock Fixed point theorems in ordered abstract spaces.
\newblock {\em Proceedings of the American Mathematical Society\/}~{\em
  135\/}(8), 2505--2517.

\bibitem[\protect\citeauthoryear{Obstfeld}{Obstfeld}{1990}]{obstfeld90}
Obstfeld, M. (1990).
\newblock Intertemporal dependence, impatience, and dynamics.
\newblock {\em Journal of Monetary Economics\/}~{\em 26\/}(1), 45--75.

\bibitem[\protect\citeauthoryear{Petru{\c{s}}el, Petru{\c{s}}el, and
  Urs}{Petru{\c{s}}el et~al.}{2013}]{petrucetal13}
Petru{\c{s}}el, A., G.~Petru{\c{s}}el, and C.~Urs (2013).
\newblock Vector-valued metrics, fixed points and coupled fixed points for
  nonlinear operators.
\newblock {\em Fixed Point Theory and Applications\/}~{\em 2013\/}(1), 1--21.

\bibitem[\protect\citeauthoryear{Pichler and Sorger}{Pichler and
  Sorger}{2009}]{pichlersorger09}
Pichler, P. and G.~Sorger (2009).
\newblock Wealth distribution and aggregate time-preference.
\newblock {\em Journal of Economic Dynamics and Control\/}~{\em 33\/}(1),
  1--14.

\bibitem[\protect\citeauthoryear{Rosen}{Rosen}{1965}]{rosen65}
Rosen, J.~B. (1965).
\newblock Existence and uniqueness of equilibrium points for concave
  {$N$}-person games.
\newblock {\em Econometrica\/}~{\em 3\/}(33), 520--534.

\bibitem[\protect\citeauthoryear{Sorger}{Sorger}{2002}]{sorger02}
Sorger, G. (2002).
\newblock On the long-run distribution of capital in the {R}amsey model.
\newblock {\em Journal of Economic Theory\/}~{\em 105\/}(1), 226--243.

\bibitem[\protect\citeauthoryear{Sorger}{Sorger}{2008}]{sorger08}
Sorger, G. (2008).
\newblock Strategic saving decisions in the infinite-horizon model.
\newblock {\em Economic Theory\/}~{\em 36\/}(3), 353--377.

\bibitem[\protect\citeauthoryear{Stern}{Stern}{2006}]{stern06}
Stern, M. (2006).
\newblock Endogenous time preference and optimal growth.
\newblock {\em Economic Theory\/}~{\em 29\/}(1), 49--70.

\bibitem[\protect\citeauthoryear{Stokey, Lucas, and Prescott}{Stokey
  et~al.}{1989}]{stokeylucas89}
Stokey, N., R.~E. Lucas, and E.~C. Prescott (1989).
\newblock {\em Recursive Methods in Economic Dynamics}.
\newblock Cambridge, MA: Harvard University Press.

\bibitem[\protect\citeauthoryear{Tohm\'{e}}{Tohm\'{e}}{2006}]{tohme06}
Tohm\'{e}, F. (2006).
\newblock {R}olf {M}antel and the computability of general equilibria: On the
  origins of the {S}onnenschein-{M}antel-{D}ebreu theorem.
\newblock {\em History of Political Economy\/}~{\em 38\/}(Suppl.1), 213--227.

\bibitem[\protect\citeauthoryear{Uzawa}{Uzawa}{1968}]{uzawa68}
Uzawa, H. (1968).
\newblock Time preference, the consumption function and optimum assets
  holdings.
\newblock In J.~Wolfe (Ed.), {\em Value, Capital, and Growth: Papers in Honour
  of {S}ir {J}ohn {H}icks}, pp.\  485--504. Aldine.

\bibitem[\protect\citeauthoryear{Van~Luong and Thuan}{Van~Luong and
  Thuan}{2011}]{vanthuan11}
Van~Luong, N. and N.~X. Thuan (2011).
\newblock Coupled fixed points in partially ordered metric spaces and
  application.
\newblock {\em Nonlinear Analysis: Theory, Methods \& Applications\/}~{\em
  74\/}(3), 983--992.

\end{thebibliography}

\end{document}